\documentclass{amsproc}
\usepackage{amsmath, euscript}
\usepackage{cases}
\usepackage{mathrsfs}
\usepackage{bbm}
\usepackage{amssymb}
\usepackage{txfonts}
\usepackage{amscd}
\usepackage{amsfonts,latexsym,amsmath,amsthm,amsxtra,mathdots,amssymb,latexsym,mathabx}
\usepackage[all,cmtip]{xy}
\usepackage{color}
\usepackage{multicol}
\usepackage{hyperref}
\usepackage{tikz}

\allowdisplaybreaks

\def\mod{\mathrm{mod}\ }

     \newcommand{\BB}{{\mathbb {B}}}
    \newcommand{\BC}{{\mathbb {C}}} 
     \newcommand{\BF}{{\mathbb {F}}}
     \newcommand{\BH}{{\mathbb {H}}}
     
    \newcommand{\BK}{{\mathbb {K}}} \newcommand{\BL}{{\mathbb {L}}}
     \newcommand{\BN}{{\mathbb {N}}}
     
    \newcommand{\BQ}{{\mathbb {Q}}} \newcommand{\BR}{{\mathbb {R}}}
    \newcommand{\BS}{{\mathbb {S}}} 
    \newcommand{\BU}{{\mathbb {U}}} 
     \newcommand{\BX}{{\mathbb {X}}}
     \newcommand{\BZ}{{\mathbb {Z}}}

     \newcommand{\fC}{{\mathfrak{C}}}
    \newcommand{\GL}{{\mathrm{GL}}}
    \renewcommand{\Im}{{\mathfrak{Im}\,}}

    \newcommand{\PGL}{{\mathrm{PGL}}} 
    \renewcommand{\Re}{{\mathfrak{Re}\,}}
    
    \newcommand{\PSL}{{\mathrm{PSL}}}

\def\-{^{-1}}

\def\Ssis{\mathscr S_{\mathrm {sis}}}
\def\Ssiss{\mathscr T_{\mathrm {sis}}}

\def\Msis{\mathscr M_{\mathrm {sis}}}

\def\Nsis{\mathscr N_{\mathrm {sis}}}

\def\SS{\mathscr S }

\def\Hrd {\mathscr H_{\mathrm {rd}}}

\def\hh{{ \text{\usefont{T1}{pzc}{m}{it}{h}}} }

\def\hld{{ \text{\usefont{T1}{pzc}{m}{it}{h}}} _{(\umu, \udelta)}}
\def\hmld{{ \text{\usefont{T1}{pzc}{m}{it}{h}}} _{(-\umu, \udelta)}}
\def\hmum{{ \text{\usefont{T1}{pzc}{m}{it}{h}}} _{(\umu, \um)}}
\def\hmmum{{ \text{\usefont{T1}{pzc}{m}{it}{h}}} _{(-\umu, \um)}}

\def\Hsl{\EuScript H_{(\usigma, \ulambda)}}
\def\Hld{\EuScript H_{(\umu, \udelta)}}
\def\Hmum{\EuScript H_{(\umu, \um)}}
\def\Hmld{\EuScript H_{(-\umu, \udelta)}}
\def\Hmmum{\EuScript H_{(-\umu, -\um)}}

\def\BRx{\mathbb R^\times}
\def\BCx{\mathbb C^\times}
\def\BFx{\mathbb F^\times}

\def\ux{\boldsymbol x}
\def\uy{\boldsymbol y}
\def\udelta{\boldsymbol {\delta}}
\def\ulambda{\boldsymbol \lambda}

\def\usigma{\boldsymbol \varsigma}
\def\unu{\boldsymbol \nu}
\def\umu{\boldsymbol \mu}

\def\urho{\boldsymbol \varrho}
\def\ue{\boldsymbol e}
\def\um{\boldsymbol m}
\def\utheta{\boldsymbol \theta}

\def\-{^{-1}}

\def\ukappa{\boldsymbol \kappa}
\def\lp {\left (}
\def\rp {\right )}
\def\lpp {\left \{ }
\def\rpp {\right \} }
\def\EC{\EuScript C}
\def\EJ{\EuScript J}
\def\EM{\EuScript M}
\def\EF{\EuScript F}
\def\ET{\EuScript T}
\def\ES{\EuScript S}
\def\EH{\EuScript H}
\def\Voronoi{Vorono\" \i \hskip 3 pt}

\def\BZT{\BZ/2 \BZ}

\def\uk{\boldsymbol k}

\newcommand{\delete}[1]{}

    \newcommand{\SL}{{\mathrm{SL}}}
    \newcommand{\SO}{{\mathrm{SO}}}
    \newcommand{\SU}{{\mathrm{SU}}}
    \newcommand{\sgn}{{\mathrm{sgn}}} \newcommand{\Tr}{{\mathrm{Tr}}}

    \newcommand{\ds}{\displaystyle}
    \newcommand{\sstyle}{\scriptstyle}

    \newcommand{\ra}{\rightarrow}

    \theoremstyle{plain}

    \newtheorem{thm}{Theorem}[section] \newtheorem{cor}[thm]{Corollary}
    \newtheorem{lem}[thm]{Lemma}  \newtheorem{prop}[thm]{Proposition}
    
     \newtheorem{defn}[thm]{Definition}
    
    \newtheorem {rem}[thm]{Remark} \newtheorem {example}[thm]{Example}
    \newtheorem {notation}[thm]{Notation}

    \newtheorem{observation}{Observation}[section]

    \renewcommand{\theequation}{\arabic{equation}}
    \numberwithin{equation}{section}
    
    \newtheorem{acknowledgement}{Acknowledgements}

\makeatletter
\newsavebox\myboxA
\newsavebox\myboxB
\newlength\mylenA

\newcommand*\xoverline[2][0.75]{%
	\sbox{\myboxA}{$\m@th#2$}%
	\setbox\myboxB\null
	\ht\myboxB=\ht\myboxA%
	\dp\myboxB=\dp\myboxA%
	\wd\myboxB=#1\wd\myboxA
	\sbox\myboxB{$\m@th\overline{\copy\myboxB}$}
	\setlength\mylenA{\the\wd\myboxA}
	\addtolength\mylenA{-\the\wd\myboxB}%
	\ifdim\wd\myboxB<\wd\myboxA%
	\rlap{\hskip 0.5\mylenA\usebox\myboxB}{\usebox\myboxA}%
	\else
	\hskip -0.5\mylenA\rlap{\usebox\myboxA}{\hskip 0.5\mylenA\usebox\myboxB}%
	\fi}
\makeatother

\begin{document}

\title[Theory of Bessel Functions of High Rank - II]{Theory of Bessel Functions of HighRank - 
II: \\ Hankel Transforms and Fundamental Bessel Kernels}

\author{ Zhi Qi}


\address{Department of Mathematics\\ The Ohio State University\\100 Math Tower\\231 West 18th Avenue\\Columbus, OH 43210\\USA}
\email{qi.91@buckeyemail.osu.edu}


\begin{abstract}
In this article we shall study the analytic theory and the representation theoretic interpretations of Hankel transforms and fundamental Bessel kernels of an arbitrary rank over an archimedean field.
\end{abstract}

\maketitle

\begin{footnotesize}
\tableofcontents
\end{footnotesize}

\section{Introduction}

$ $ In this article, we shall study \textit{Hankel transforms}\footnote{They are called Bessel transforms in some literatures, for instance, \cite{Ichino-Templier}. However, this type of integral transforms should actually be attributed to Hermann Hankel. Moreover, we shall reserve the term {\it Bessel transforms} for the transforms shown in the Kuznetsov trace formula.} as well as their integral kernels, called {\it  fundamental Bessel kernels}\footnote{The adjective \textit{fundamental} is added for the distinction from the Bessel functions for $\GL_n$ in the Kuznetsov formula, and will be dropped when no confusion occurs.}, over an archimedean field. These Hankel transforms are the archimedean constituent of the \Voronoi summation formula over a number field.

\addtocontents{toc}{\protect\setcounter{tocdepth}{1}}

\subsection{Analytic theory}

Let $n$ be a positive integer. 
In the case $n \geq 3$ Hankel transforms of rank $n$ over $\BR$ have been investigated in the work of Miller and Schmid \cite{Miller-Schmid-2004, Miller-Schmid-2006, Miller-Schmid-2009} on the \Voronoi summation formula for $\GL_n (\BZ)$. The notion of {\it automorphic distributions}\footnote{According to Stephen Miller, the origin of automorphic distributions can be traced back to the 19th century in the work of Sim\'eon Poisson on Poisson's integral for harmonic functions on either the unit disk or the upper half-plane.} is used for their proof of this formula, and is also used to derive the analytic continuation and the functional equation of the $L$-function of a cuspidal $\GL_n(\BZ)$-automorphic representation of $\GL_n (\BR)$.
As the foundation of automorphic distributions, the harmonic analysis over $\BR$ is studied in \cite{Miller-Schmid-2006} from the viewpoint of the signed Mellin transforms.
As explained in \cite{Miller-Schmid-2004-1}, the cases $n = 1, 2$ can also be incorporated into their framework.
Furthermore, it is shown in the author's previous paper \cite{Qi} that all  Hankel transforms  over $\BR$ admit integral kernels, which can be partitioned into combinations of the so-called {\it fundamental Bessel functions}. These Bessel functions 
are studied from two approaches via their {\it formal integral representations} and {\it Bessel differential equations}.

In \S \ref{sec: notations} - \ref{sec: asymptotics}, we shall establish 
the analytic theory of Hankel transforms %
and their Bessel kernels  over $\BC$. The study of Hankel transforms  for $\GL_n(\BC)$ from the perspective of \cite{Miller-Schmid-2006}  is complete to some extent.  On the other hand, Bessel functions in \cite{Qi} play a fundamental role in our study of Bessel kernels over $\BC$, for instance, in finding their asymptotic expansions.
Although our main focus is on the theory over $\BC$, the theory of Hankel transforms over $\BR$ extracted from \cite{Miller-Schmid-2006} as well as some treatments of Bessel kernels over $\BR$ will also be included for the sake of comparison.


The sections \S \ref{sec: notations} - \ref{sec: asymptotics} are outlined  as follows.




In the preliminary section \S \ref{sec: notations}, some basic notions are introduced, such as gamma factors, Schwartz spaces, the Fourier transform and Mellin transforms. The three  kinds of Mellin transforms  $\EM$, $\EM_{\BR}$ and $\EM_{\BC}$ are first defined over the Schwartz spaces over $\BR_+ = (0, \infty)$, $\BRx = \BR \smallsetminus \{0\}$ and $\BCx = \BC \smallsetminus \{0\}$ respectively.

In \S \ref{sec: all Ssis}, the definitions of the Mellin transforms  $\EM$, $\EM_{\BR}$ and $\EM_{\BC}$ are extended onto certain function spaces  $\Ssis (\BR_+)$, $\Ssis (\BR^\times)$ and $\Ssis (\BCx)$ respectively. We shall precisely characterize their image spaces $\Msis$, $\Msis^\BR$ and  $\Msis^\BC$ under their corresponding Mellin transforms. In spite of their similar constructions, the  analysis of the Mellin transform $\EM_{\BC}$ is much more elaborate than that of   $\EM_{\BR}$ or  $\EM$.


In \S \ref{sec: Hankel transforms}, based on gamma factors and Mellin transforms, we shall construct Hankel transforms upon suitable subspaces of the $\Ssis$ function spaces just introduced in \S \ref{sec: all Ssis} and study their Bessel kernels. It turns out that all these Bessel kernels can be formulated in terms of the Bessel functions in \cite{Qi}. 

In \S \ref{sec: Fourier type transforms}, we  shall first introduce the Schmid-Miller transforms in companion with the Fourier transform and then use them to establish  a Fourier type integral transform expression of a Hankel transform.

In \S \ref{sec: integral representations}, we shall introduce   certain integrals, derived from the Fourier type integral transforms given in \S \ref{sec: Fourier type transforms}, that represents Bessel kernels. When the field is real, these integrals never absolutely converge and are closely connected to the formal integrals studied in \cite{Qi}. In the complex case, however,  some range of index can be found where such integrals are absolutely convergent.

The last two sections \S \ref{sec: two connection formulae for J mu m} and \S \ref{sec: asymptotics} are devoted to Bessel kernels over $\BC$.
In \S \ref{sec: two connection formulae for J mu m}, we shall prove two connection formulae  that relate a Bessel kernel over $\BC$ to the two kinds of Bessel functions of {\it positive} sign. These kinds of Bessel functions arise in the study of  Bessel equations in \cite[\S 7]{Qi}. In  \S \ref{sec: asymptotics}, as a consequence of  the second connection formula above, we shall derive the asymptotic expansion of a Bessel kernel over $\BC$ from \cite[Theorem 7.27]{Qi}.

\subsection{Representation theory}\label{sec: intro, representation theory} The work of Miller and Schmid is extended by Ichino and Templier \cite{Ichino-Templier} to any irreducible cuspidal automorphic representation
of $\GL_n$, $n \geq 2$, over an arbitrary number field $\BK$.
The \Voronoi summation formula for $\GL_n$  follows from the global theory of $\GL_n \times \GL_1$-Rankin-Selberg $L$-functions. For an archimedean completion $\BK_\varv$ of $\BK$, the defining identities of the Hankel transform associated with an infinite dimensional irreducible unitary generic representation of $\GL_n (\BK_\varv)$ are reformulations of the corresponding local functional equations for $\GL_n \times \GL_1$-Rankin-Selberg zeta integrals over $\BK_\varv$.

In \S  \ref{sec: Hankel, Ichino-Templier}, we shall first recollect the definition of the Hankel transform associated with an infinite dimensional irreducible admissible generic representation of $\GL_n( \BF)$  for an archimedean field $\BF$ 
in \cite{Ichino-Templier}. We stress that  this definition actually works for any irreducible admissible representation  of $\GL_n(\BF)$, including $n = 1$. We shall then give a detailed discussion on   Hankel transforms of rank $n$ over $\BF$ using the Langlands classification for $\GL_n(\BF)$.

In   \S \ref{sec: Bessel, GL2(F)}, according to the theory of local functional equations for $\GL_2 \times \GL_1$-Rankin-Selberg zeta integrals over $\BF$, we shall show  that the action of the long Weyl element on the Kirillov model of an infinite dimensional irreducible admissible representation of $\GL_2(\BF)$ is essentially a Hankel transform over $\BF$.
It follows the consensus that for $\GL_2 (\BF)$ the Bessel functions occurring in the Kuznetsov  trace formula should coincide with those in the \Voronoi summation formula. 
This will let us prove and generalize the Kuznetsov trace formula for $\PSL_2(\BZ[i])\backslash \PSL_2( \BC)$ in \cite{B-Mo}\footnote{In \cite{M-W-Kuz}, Miatello and Wallach gave the spherical Kuznetsov trace formula for real semisimple groups of real rank one, which include both $\SL_2 (\BR)$ and $\SL_2 (\BC)$. The first instance of Bessel functions for  spherical principal series representations of $\SL_2(\BC)$ can be found in their work. More generally, in generalizing the Kuznetsov formula to the non-spherical case, the formula and the integral representation of the Bessel functions associated with  principal series representations of $\SL_2(\BC)$ are discovered by Bruggeman, Motohashi and Lokvenec-Guleska \cite{B-Mo, B-Mo2}. Their approach is however entirely different from ours.}, in the same way that \cite{CPS} does for the Kuznetsov trace formula for $\PSL_2(\BZ )\backslash \PSL_2( \BR)$ in \cite{Kuznetsov}.

In \cite[\S 3.2]{Qi-Thesis}, the author found a formula of Bessel functions for $\GL_3(\BF)$ in terms of fundamental Bessel kernels, formally derived from Rankin-Selberg $\GL_3 \times \GL_2$ local functional equations and the $\GL_2$ Bessel-Plancherel formula. The Bessel functions for $\GL_3 (\BF)$ are two-variable special functions arising in the Kuznetsov trace formula for $\GL_3 (\BF)$. To be precise, the fundamental Bessel kernels occurring in the formula are for $\GL_3 \times \GL_2$ and $\GL_2 $, where the involved representations of $\GL_2$ are all tempered. This is the first instance of Bessel functions for groups of higher rank.

\subsection{Distribution theory} Although such a theory can be formulated, we shall not touch in this article the theory of Hankel transforms over $\BC$ from the perspective of distributions as in \cite{Miller-Schmid-2006}. It is very likely that this will lead to the theory of automorphic distributions on $\GL_n (\BC)$ with respect to congruence subgroups, as well as the \Voronoi summation formula for cuspidal automorphic representations of $\GL_n (\BC)$. The \Voronoi summation formula in this generality is already covered by \cite{Ichino-Templier}, but this approach would still be of its own interest.

\subsection{Applications} When $n=2$, there are numerous applications in analytic number theory of the \Voronoi summation formula and the Kuznetsov trace formula over $\BQ$, which include subconvexity, non-vanishing of automorphic $L$-functions and estimates for shifted convolution sums. The \Voronoi summation formula for $\GL_3(\BZ)$ is used to establish subconvexity in \cite{XLi2011} as well.
In order to work over an arbitrary number field, one  also needs to understand   Hankel transforms and   Bessel kernels at least for $\GL_2(\BC)$. We hope that the present work and its sequel will make these problems over a  number field more approachable from the analytic perspective.

\begin{acknowledgement}
	The author is grateful to his advisor, Roman Holowinsky, who brought him to the area of analytic number theory, gave him much enlightenment, guidance and encouragement. The author would like to thank James Cogdell, Ovidiu Costin, Stephen Miller, Pengyu Yang and Cheng Zheng for valuable comments and helpful discussions.
\end{acknowledgement}

\section{Notations and preliminaries}\label{sec: notations}

\addtocontents{toc}{\protect\setcounter{tocdepth}{2}}

\subsection{General notations} \begin{itemize}
\item[-] Denote $\BN=\{0,1,2...\}$ and $\BN_+ = \{1, 2, 3, ...\}$. 
\item[-] The group $\BZ/2\BZ$ is usually identified with the two-element set $\{0, 1 \} $.
\item[-] Denote $\BR_+ = (0, \infty)$, $\overline \BR_+ =[0, \infty)$, $\BR^\times = \BR \smallsetminus \{ 0 \}$ and $\BC^\times = \BC \smallsetminus \{ 0 \}$.
\item[-] Denote by $\BU \cong \BR _+ \times \BR$ the universal cover of $\BC \smallsetminus \{ 0 \}$. Each element $\zeta \in \BU$ is denoted by $\zeta = x e^{i \omega}$, with $(x, \omega) \in \BR _+ \times \BR$.
\item[-] For $m \in \BZ$ define $\delta (m) \in \BZT$ by $\delta(m) = m (\mod 2)$.
\item[-] For $s \in \BC$ and $\alpha \in \BN$, let
$[s]_{\alpha} =  \textstyle \prod_{\kappa = 0}^{\alpha-1} (s - \alpha)$ and $ (s)_{\alpha} =  \textstyle \prod_{\kappa = 0}^{\alpha-1} (s + \alpha)$  if  $\alpha\geq 1$, and let $[s]_{0} = (s)_0 = 1$.
\item[-] For $s \in \BC$ let $e(s) = e^{2 \pi i s}$.
\item[-] For a finite  closed interval $[a, b] \subset \BR$ define the closed vertical strip $\BS [a, b] = \{ s \in \BC : \Re s \in [a, b]  \}.$ The open vertical strip $\BS (a, b)$ for a finite open interval $(a, b)$ is similarly defined.
\item[-] For  $ \lambda \in \BC$ and $r > 0$, define $ \BB_{r} (\lambda) = \left\{ s \in \BC : |s - \lambda| < r \right \} $ to be the disc of  radius $r$ centered at $s = \lambda$.
\item[-] For $\ulambda = (\lambda_1, ..., \lambda_n) \in \BC^n$  denote $|\ulambda| = \sum_{l      = 1}^n \lambda_{l     }$ (this notation works for subsets of $\BC^n$, for instance, $(\BZT)^n = \{0, 1\}^n$ and $\BZ^n$).
\item[-] Define the hyperplane $\BL^{n-1} = \left \{ \ulambda \in \BC^n : |\ulambda|  = \sum_{l      = 1}^n \lambda_{l     } = 0 \right\} $.
\item[-] Denote by $\ue^n$ the $n$-tuple $(1, ..., 1)$.
\item[-]  For $\um = (m_1, ..., m_n)\in \BZ^n$ define $ \|\um\| = (|m_1|, ..., |m_n|)$.
\end{itemize}

\subsection{Gamma factors}


\subsubsection{} We define  the gamma factor
\begin{equation}\label{1def: G pm (s)}
G (s, \pm) = \Gamma (s) e\lp \pm \frac s 4\rp.
\end{equation}
For $( \usigma, \ulambda) = (\varsigma_1, ..., \varsigma_n, \lambda_1, ..., \lambda_n) \in  \{ +, - \}^n \times \BC^{n}$ let
\begin{equation}\label{1def: G(s; sigma; lambda)}
G (s; \usigma, \ulambda) = \prod_{l      = 1}^n G  (s - \lambda_l, \varsigma_l          ).
\end{equation}

\subsubsection{} For $\delta \in \BZ/ 2\BZ = \{0, 1\}$, we define the gamma factor
\begin{equation} \label{1def: G delta}
G_\delta (s) = i^\delta \pi^{ \frac 1 2 - s} \frac {\Gamma \lp \frac 1 2 ({s + \delta} ) \rp} {\Gamma \lp \frac 1 2 ({1 - s + \delta} ) \rp} = 
\left\{ \begin{split}
& 2(2 \pi)^{-s} \Gamma (s) \cos \left(\frac {\pi s} 2 \right), \hskip 10pt \text { if } \delta = 0,\\
& 2 i (2 \pi)^{-s} \Gamma (s) \sin  \left(\frac {\pi s} 2 \right), \hskip 9 pt \text { if } \delta = 1.
\end{split} \right.
\end{equation}
Here, we have used the duplication formula and Euler's reflection formula for the Gamma function,
\begin{equation*}
\Gamma (1-s) \Gamma (s) = \frac \pi {\sin (\pi s)}, \hskip 10 pt \Gamma (s) \Gamma \lp s + \frac 1 2 \rp = 2^{1-2s} \sqrt \pi \Gamma (2 s).
\end{equation*}
Let $(\umu, \udelta) = (\mu_1, ..., \mu_n, \delta_1, ..., \delta_n) \in \BC^{n} \times (\BZ/2 \BZ)^n$ and define 
\begin{equation}\label{1def: G (lambda, delta)}
G_{(\umu, \udelta) } (s) = \prod_{l      = 1}^n G_{\delta_{l     } } (s - \mu_l     ).
\end{equation}
One   observes the following simple functional relation
\begin{equation}\label{1eq: G mu delta (1-s) = G - mu delta (s)}
G_{(\umu, \udelta) } (1-s)  G_{(- \umu, \udelta) } (s) = 1.
\end{equation}

\subsubsection{}

For $m \in \BZ$, we define the gamma factor
\begin{equation}\label{1def: G m (s)}
G_m (s) = i^{|m| } (2\pi)^{1-2 s } \frac { \Gamma \lp s + \frac 1 2{|m|}   \rp} { \Gamma \lp 1 - s + \frac 1 2{|m|}   \rp }.
\end{equation}
Let $(\umu, \um) = (\mu_1, ..., \mu_n, m_1, ..., m_n) \in \BC^{n} \times \BZ^n$ and define 
\begin{equation}\label{1def: G (mu, m)}
G _{(\umu, \um)} (s) = \prod_{l      = 1}^n G_{m_{l     } } (s - \mu_l     ).
\end{equation}
We have the functional relation
\begin{equation}\label{1eq: G mu m (1-s) = G - mu m (s)}
G_{(\umu, \um) } (1-s)  G_{(- \umu, \um) } (s) = 1.
\end{equation}

\subsubsection{Relations between the three types of gamma factors} \label{sec: Gamma factor, R and C}

We first observe that
\begin{equation*}
G_{\delta} (s) = (2\pi)^{-s} \lp G  (s, +) + (-)^\delta G (s, -) \rp.
\end{equation*}
Hence
\begin{equation}\label{1eq: G (lambda, delta) = G (s; sigma, lambda)}
G_{(\umu, \udelta)} (s) = \sum_{\usigma \in \{+, -\}^n } \usigma ^{\udelta} (2\pi)^{|\umu| - ns}   G (s;  \usigma, \umu), \hskip 10 pt  \usigma ^{\udelta} = \prod_{l     =1}^n \varsigma _{l     }^{\delta_l     },\,  |\umu| = \sum_{l     =1}^n \mu_l     .
\end{equation}

Euler's reflection formula and certain trigonometric identities yield
\begin{equation}\label{1f: G m (s) = G 1 G delta(m)}
\begin{split}
i G_m (s) & = i^{|m|+1} 2 (2 \pi)^{-2s } \Gamma \lp s + \frac {|m|} 2 \rp \Gamma \lp s - \frac {|m|} 2 \rp \sin \lp \pi \lp  s - \frac {|m|} 2 \rp \rp\\
& = G_{\delta( m) + 1} \lp s - \frac {| m|} 2 \rp  G_{0} \lp s + \frac {|m |} 2 \rp\\
& = G_{\delta( m)} \lp s - \frac {| m|} 2 \rp  G_{1} \lp s + \frac {|m |} 2 \rp,
\end{split}
\end{equation}
with  $\delta (m) = m (\mod 2)$. Therefore,  $G _{(\umu, \um)} (s)$  may be viewed as a certain $  G_{(\boldsymbol \eta, \udelta) } (s)$ of doubled rank. 

\begin{lem}\label{1lem: complex and real gamma factors}

Suppose that $(\umu, \um) \in \BC^{n} \times \BZ^n$ and $(\boldsymbol \eta, \udelta) \in  \BC^{2n} \times (\BZ/2 \BZ)^{2n}$ are subjected to one of the following two sets of relations
\begin{align} \label{1eq: relation between (mu, m) and (lambda, delta), 1}
& \eta_{2 l      - 1} = \mu_l      + \frac {|m_l      |} 2, \   \eta_{2 l      } = \mu_l      - \frac {|m_l      |} 2,\ \delta_{2l      - 1} = \delta( m) + 1,\ \delta_{2l      } = 0; \\
\label{1eq: relation between (mu, m) and (lambda, delta), 2}
&  \eta_{2 l      - 1} = \mu_l      + \frac {|m_l      |} 2, \   \eta_{2 l      } = \mu_l      - \frac {|m_l      |} 2,\ \delta_{2l      - 1} =  \delta( m),\ \delta_{2l      } = 1.
\end{align}
Then  $i^n G _{(\umu, \um)} (s) = G_{(\boldsymbol \eta, \udelta) } (s)$.
\end{lem}

\subsubsection{Stirling's formula} Fix $s_0 \in \BC$, and let $|\arg s| < \pi - \epsilon$, $0 < \epsilon < \pi$. We have the following asymptotic as  $|s| \ra \infty$
\begin{equation*} 
\log \Gamma (s_0 + s) \sim \left( s_0 + s - \frac 1 2 \right) \log s - s + \frac 1 2 \log (2 \pi).
\end{equation*}
If one writes  $s_0 = \rho_0 + i t_0$ and $s = \rho + i t$, $\rho \geq 0$, then the right hand side is equal to
\begin{equation*}
\begin{split}
& \left(\rho_0 + \rho - \frac 1 2 \right) \log \sqrt {t^2 + \rho^2} - (t_0 + t) \arctan \left( \frac t \rho \right) - \rho + \frac 1 2 \log (2 \pi) \\
& + i (t_0 +  t) \log {\sqrt {t^2 + \rho^2}} - i t + i  \left(\rho_0 + \rho - \frac 1 2 \right) \arctan \left( \frac t \rho \right),
\end{split}
\end{equation*}
and therefore
\begin{equation}\label{1eq: Stirling's formula}
|\Gamma (s_0 + s)| \sim \sqrt {2 \pi} \lp t^2 + \rho^2 \rp^{\frac 1 2 \lp \rho_0 + \rho - \frac 1 2  \rp} e^{- (t_0 + t) \arctan \left(   t / \rho \right) - \rho}.
\end{equation}

\begin{lem} \label{1lem: vertical bound}
We have
\begin{equation}\label{1eq: vertical bound, G (s; sigma, lambda)}
\begin{split}
G(s; \usigma, \ulambda) \lll_{\ulambda, a, b, r} & (|\Im s| + 1)^{n \lp \Re s - \frac 1 2 \rp - \Re |\ulambda|},
\end{split}
\end{equation}
for all $s \in \BS [a, b] \smallsetminus \bigcup_{l      = 1}^n \bigcup_{\kappa \in \BN} \BB_r (\lambda_{l     } - \kappa)$, with small $r > 0$, 
\begin{equation}\label{1eq: vertical bound, G (lambda, delta) (s)}
\begin{split}
G_{(\umu, \udelta)} (s ) \lll_{\umu, a, b, r} & (|\Im s| + 1)^{n \lp \Re s - \frac 1 2 \rp - \Re |\umu|},
\end{split}
\end{equation}
for all $s \in \BS [a, b] \smallsetminus \bigcup_{l      = 1}^n \bigcup_{\kappa \in \BN} \BB_r (\mu_{l     } - \delta_{l     } - 2 \kappa)$, and
\begin{equation}\label{1eq: vertical bound, G (mu, m) (s)}
G_{(\umu, \um)} \lp   s   \rp \lll_{\umu, a, b, r} \prod_{l      = 1}^n (|\Im s| + |m_l     | + 1)^{2 \Re s - 2\Re \mu_{l     } - 1},
\end{equation}
for all $2 s \in \BS [a, b] \smallsetminus \bigcup_{l      = 1}^n \bigcup_{\kappa \in \BN} \BB_r (2 \mu_{l     } - |m_{l     }| - 2 \kappa)$.

In other words, if $\ulambda$ and $ \umu $ are given, then $G(s; \usigma, \ulambda)$, $G_{(\umu, \udelta)} (s )$ and $G_{(\umu, \um)} (s)$ are all of moderate growth with respect to $\Im s$,  uniformly on vertical strips {\rm(}with bounded width{\rm)}, and moreover  $G_{(\umu, \um)} (s )$ is also of uniform moderate growth with respect to $\um$.
\end{lem}

\subsection{Basic notions for $\BR_+$, $\BR^\times$ and $\BC^\times$}\label{sec: R+, Rx and Cx}

Define $\BR_+ = (0, \infty)$, $\BR^\times = \BR \smallsetminus \{ 0 \}$ and $\BC^\times = \BC \smallsetminus \{ 0 \}$. We observe the isomorphisms
$\BR^\times \cong \BR_+ \times  \{+, -\}\ $($\cong \BR_+ \times \BZT$) and $\BC^\times \cong \BR_+ \times \BR/ 2\pi \BZ$, the latter being realized via the \textit{polar coordinates} $z = x e^{ i \phi}$.

\subsubsection{}
Let $|\ |$ denote the ordinary absolute value on either $\BR$ or $\BC$, and set $\|\, \|_\BR = |\  |$ for $\BR$ and $\|\, \| = \|\, \|_\BC = |\  |^2$ for $\BC$. Let $d x$ be the Lebesgue measure on $\BR$, and let  $d^\times x = |x|\- d x$ be the standard choice of the multiplicative Haar measure on $\BR^\times $.
Similarly, let $d z$ be \textit{twice} the ordinary Lebesgue measure on $\BC$, and choose the standard multiplicative Haar measure $d^\times z = \| z\|^{-1} d z$ on $\BC^\times $. Moreover, in the polar coordinates, one has $d^\times z = 2 d^\times x d \phi$.
For $x \in \BR^\times$ the sign function $\sgn(x)$ is equal to $ x /{|x|}$, whereas for $z \in \BC^\times $ we introduce the notation $[z] =  z /{|z|}$. 


Henceforth, we shall let $\BF$ be either $\BR$ or $\BC$, and occasionally let $x, y$ denote elements in $\BF$ even if $\BF=\BC$.

\subsubsection{}
For  $\delta\in \BZT$,  we define the space $C_\delta^\infty (\BRx)$ of all smooth functions $\varphi \in C ^\infty (\BRx)$ satisfying the parity condition 
\begin{equation} \label{1eq: delta condition, R}
\varphi (-x) = (-)^\delta \varphi (x). 
\end{equation} 
Observe that a function $\varphi \in C_\delta^\infty (\BRx)$ is determined by its restriction on $\BR _+$, namely, $\varphi (x) = \sgn (x)^\delta \varphi (|x|)$. Therefore, 
\begin{equation}\label{1eq: C delta = sgn delta C}
C_\delta^\infty (\BRx) = \sgn (x)^\delta C^\infty (\BR _+) = \left \{ \sgn (x)^\delta \varphi  (|x|) : \varphi \in C^\infty ( \BR_+) \right \}.
\end{equation}
For a smooth function $\varphi  \in C ^\infty (\BRx)$, we define $\varphi_{\delta} \in C^\infty (\BR _+)$ by
\begin{equation}\label{1eq: upsilon delta}
\varphi_{\delta} (x) = \frac 1 2 \lp \varphi (x) + (-)^\delta \varphi (-x) \rp, \hskip 10 pt x \in \BR _+.
\end{equation}
Clearly,
\begin{equation}\label{1eq: C = C0 + C1}
\varphi (x) = \varphi_0 (|x|) + \sgn(x) \varphi_1 (|x|).
\end{equation}

For $m \in \BZ$, we define the space $C_m^\infty (\BCx)$ of all smooth functions $\varphi \in C ^\infty (\BCx)$ satisfying
\begin{equation} \label{1eq: m condition, C}
\varphi \big( x e^{i\phi} \cdot e^{i \phi'} \big) = e^{i m \phi'} \varphi \lp x e^{i \phi}\rp.
\end{equation} 
A function $\varphi \in C_m^\infty (\BCx)$ is determined by its restriction on $\BR _+$, namely, $\varphi (z) = [z]^m  \varphi (|z|)$, or, in the polar coordinates, $\varphi (x e^{i\phi} ) = e^{i m \phi } \varphi (x)$. Therefore,
\begin{equation}\label{1eq: C m = [z] m C}
C_m^\infty (\BRx) = [z]^m C^\infty (\BR _+) = \left \{ [z]^m \varphi  (|z|) = e^{i m \phi } \varphi (x) : \varphi \in C^\infty ( \BR_+) \right \}.
\end{equation}
For a smooth function $\varphi  \in C ^\infty (\BCx)$, we let $\varphi_{ m} \in C^\infty (\BR _+)$ denote the $m$-th Fourier coefficient of $\varphi$ given by
\begin{equation}\label{1eq: Fourier coefficients of upsilon}
\varphi_m (x) = \frac 1 { {2 \pi}} \int_0^{2\pi} \varphi \lp x e^{i \phi} \rp  e^{- i m \phi} d \phi.
\end{equation}
One has the Fourier expansion of $\varphi$,
\begin{equation}\label{1eq: Fourier series expansion}
\varphi \lp x e^{i \phi}\rp = \sum_{m \in \BZ} \varphi_m (x) e^{i m \phi}.
\end{equation}

\subsubsection{}
Subsequently, we shall encounter various subspaces of $C^\infty (\BFx)$, with $\BF = \BR, \BC$, for instance, $\SS (\BF)$, $\SS (\BFx)$, $\Ssis (\BFx)$, $\Ssis^{(\umu, \udelta)} (\BRx)$ and $\Ssis^{(\umu, \um)} (\BCx)$. Here, we list three central questions that will be the guidelines of our investigations of these function spaces.

For now, we let $D $ be a subspace of $C^\infty (\BFx)$. For $\BF = \BR$ (respectively $\BF = \BC$), we shall add a superscript or subscript $\delta$ (respectively $m$) to the notation of $D$, say $D_\delta$ (respectively $D_m$), to denote the space of $\varphi \in D$ satisfying \eqref{1eq: delta condition, R} (respectively \eqref{1eq: m condition, C}). In view of \eqref{1eq: C delta = sgn delta C} (respectively \eqref{1eq: C m = [z] m C}), there is a subspace of $C^\infty (\BR _+)$, say $E_\delta$ (respectively $E_m$), such that 
$D_\delta = \sgn(x)^\delta E_\delta$ (respectively $D_m = [z]^m E_m$).

Firstly, we are interested in the question,
\begin{itemize}
\item[] `` How to characterize the space $E_\delta$ (respectively $E_m$)?''.
\end{itemize}

Moreover, the subspaces $D \subset C^\infty (\BFx)$ that we shall consider always satisfy the following two hypotheses, 
\begin{itemize}
\item[-] $\varphi \in D $ implies $\varphi_\delta \in E_\delta$ for $\BF = \BR$ (respectively, $\varphi \in D $ implies $\varphi_m \in E_m$ for $\BF = \BC$), and
\item[-] $D$ is closed under addition.
\end{itemize}
For $\BF = \BR$, under these two hypotheses, it follows from \eqref{1eq: C = C0 + C1} that
$$D = D_0 \oplus D_1 \cong  E_0 \times E_1.$$
For $\BF = \BC$,  in view of \eqref{1eq: Fourier series expansion},  the map that sends $\varphi$ to the sequence $\left\{\varphi_m \right \} $ of its Fourier coefficients is injective. The second question arises,
\begin{itemize}
\item[] `` What is the image of $D$ in $\prod_{m \in \BZ} E_m $ under this map?'', or equivalently,
\item[] `` What conditions should a sequence $\left\{\varphi_m \right \}  \in \prod_{m \in \BZ} E_m$ satisfy in order for the Fourier series defined by \eqref{1eq: Fourier series expansion} giving a function $\varphi \in D$?''.
\end{itemize}

Finally, after introducing the Mellin transform $\EM_\BF$, 
we shall focus on the question,
\begin{itemize}
\item[] `` What is the image of $D$ under the Mellin transform  $\EM_\BF$?''. 
\end{itemize}

\subsection{Schwartz spaces} \label{sec: Schwartz spaces}

We say that a function $\varphi \in C^\infty (\BR_+)$ is {\it smooth at zero} if all of its derivatives admit asymptotics as below,
\begin{equation}\label{1eq: asymptotics R+}
\varphi^{(\alpha)} (x) = \alpha!  a_{ \alpha}  + O_{ \alpha } \lp x \rp \text{ as } x \ra 0, \text{ for any } \alpha \in \BN, \text{ with } a_\alpha \in\BC.
\end{equation}

\begin{rem}
Consequently, one has the asymptotic expansion $\varphi (x) \sim \sum_{\kappa =0}^\infty a_\kappa  x^\kappa$, which means that 
$\varphi (x) = \sum_{\kappa  = 0}^{A } a_{\kappa } x^{\kappa } + O_{ A } \lp x^{A + 1 }\rp$ as $x \ra 0$ for any $A \in \BN$.
It is not required that the series $\sum_{\kappa =0}^\infty a_\kappa  x^\kappa $ be convergent for any $x \in \BR^\times$.

Actually, \eqref{1eq: asymptotics R+} is equivalent to the following
\begin{equation}\label{1eq: asymptotics R+, 2}
\varphi^{(\alpha)} (x) = \sum_{\kappa  = \alpha}^{\alpha + A} a_{\kappa } [\kappa ]_{\alpha} x^{\kappa - \alpha} + O_{ \alpha, A } \lp x^{A + 1 }\rp \text{ as } x \ra 0, \text{ for any } \alpha, A \in \BN.
\end{equation}

Another observation is that, for a given constant $1> \rho > 0$, \eqref{1eq: asymptotics R+} is equivalent to the following seemingly weaker statement,
\begin{equation}\label{1eq: asymptotics R+, 3}
\begin{split}
\varphi^{(\alpha)} (x) = \alpha!  a_{ \alpha}  + O_{ \alpha, \rho } \lp x^\rho \rp \text{ as } x \ra 0, \text{ for any } \alpha \in \BN, \text{ with } a_\alpha \in\BC.
\end{split}
\end{equation}
\end{rem}

Let $C^\infty (\overline \BR_+)$ denote the subspace of $C^\infty (\BR_+)$ consisting of smooth functions on $\BR_+$ that are also smooth at zero.

Let $\mathscr S (\overline \BR_+)$ denote the space of functions in $C^\infty (\overline \BR_+)$ that rapidly decay at infinity along with all of their derivatives. 
Let  $\mathscr S (\BF)$ denote the Schwartz space on $\BF$, with $\BF = \BR, \BC$.



Let $\mathscr S (\BR_+)$ denote the space of Schwartz functions on $ \BR_+ $, that is, smooth functions  on $ \BR_+ $ whose derivatives rapidly decay at {\it both} zero and infinity.
Similarly, we denote by $\mathscr S (\BF^\times )$ the space of Schwartz functions on $ \BF^\times$.


The following lemma provides  criteria for characterizing functions in these Schwartz spaces, especially functions in $\mathscr S (\BC)$ or $\mathscr S (\BCx )$ in the polar coordinates. Its proof is left as an easy excise in analysis for the reader.
\begin{lem}\label{lem: Schwartz}
Let notations be as above.

{\rm (1.1).} Let $\varphi \in C^\infty (\overline \BR _+)$ satisfy the asymptotics \eqref{1eq: asymptotics R+}. Then $\varphi \in \mathscr S (\overline \BR_+)$ if and only if $\varphi$ also satisfies
\begin{equation}\label{1eq: Schwartz R+}
x^{\alpha + \beta} \varphi ^{(\alpha)} (x) \lll_{\alpha, \beta} 1\, \text{ for all } \alpha, \beta \in \BN.
\end{equation}

{\rm (1.2).} A smooth function $\varphi$ on $\BR_+$ belongs to $ \mathscr S (\BR_+)$ if and only if $\varphi $ satisfies \eqref{1eq: Schwartz R+} with $\beta \in \BN$ replaced by $\beta \in \BZ$.

Let $\varphi \in \mathscr S (\overline \BR_+)$ and $a_\alpha$ be as in \eqref{1eq: asymptotics R+}. Then $\varphi \in \SS ( \BR _+)$ if and only if $a_\alpha = 0$ for all $ \alpha \in \BN$.

{\rm (2.1).}  A smooth function $\varphi $ on $\BRx$ extends to a function in $\mathscr S (\BR )$ if and only if
\begin{itemize}
\item[-] $\varphi $ satisfies \eqref{1eq: Schwartz R+} with  $x^{\alpha + \beta}$ replaced by $|x|^{\alpha + \beta}$, and
\item[-] all the derivatives of $\varphi $ admit asymptotics
\begin{equation} \label{1eq: asymptotics R}
\varphi^{(\alpha)} (x) = \alpha! a_{ \alpha} + O_{ \alpha } \lp |x| \rp \text{ as } x \ra 0, \text{ for any }  \alpha \in \BN, \text{ with } a_\alpha \in \BC.
\end{equation}

\end{itemize}

{\rm (2.2).} Let  $\varphi$ be a smooth function on $\BRx$. Then $ \varphi \in \mathscr S (\BRx )$ if and only if $\varphi $ satisfies \eqref{1eq: Schwartz R+} with $x^{\alpha + \beta}$ replaced by $|x|^{\alpha + \beta}$ and $\beta \in \BN$ by $\beta \in \BZ$.

Suppose $\varphi \in \mathscr S (\BR )$, then $\varphi \in \mathscr S (\BRx )$ if and only if $ \varphi ^{(\alpha)} (0) = 0$ for all $\alpha \in \BN$, or equivalently, $a_\alpha = 0$ for all $\alpha \in \BN$, with $a_\alpha $ given in \eqref{1eq: asymptotics R}.

{\rm (3.1).}  Write $\partial_x = \partial/ \partial x$ and $\partial_\phi = \partial/ \partial \phi$.  In the polar coordinates, a smooth function $\varphi \lp x e^{i \phi}\rp \in C^\infty (\BC ^\times)$ extends to a function in $\mathscr S (\BC )$ if and only if
\begin{itemize}
\item[-] $\varphi \lp x e^{i \phi}\rp $ satisfies
\begin{equation}\label{1eq: Schwartz Cx}
x^{ \alpha + \beta} \partial_x ^{\alpha }  \partial_\phi ^{\gamma} \varphi \lp x e^{i \phi}\rp \lll_{\alpha, \beta, \gamma} 1\, \text{ for all } \alpha, \beta, \gamma \in \BN,
\end{equation}
\item[-] all the partial derivatives of $\varphi $ admit asymptotics

\vskip 7 pt
\item[\noindent \refstepcounter{equation}(\theequation) \label{1eq: Schwartz Cx asymptotic} \hskip 6 pt]
$\ds x^\alpha \partial_x ^{\alpha }  \partial_\phi ^{\beta} \varphi \lp x e^{i \phi}\rp = \sum_{|m| \leq \alpha + \beta} \ \sum_{\sstyle |m| \leq \kappa \leq \alpha + \beta \atop {\sstyle \kappa \equiv m (\mod 2)}} a_{m, \kappa } [\kappa ]_{\alpha}   (im)^{\beta} x^{ \kappa }  e^{i m \phi} + O_{\alpha, \beta} \lp x^{\alpha + \beta+1} \rp$
\vskip 5 pt
\noindent as $  x \ra 0$, for any $ \alpha, \beta \in \BN$,
with $a_{m, \kappa } \in\BC$ for $\kappa \geq |m|$ and $\kappa \equiv m (\mod 2)$.

\end{itemize}

Let $\varphi \in \mathscr S (\BC )$ and $\varphi_m$ be the $m$-th Fourier coefficient of $ \varphi $ given by \eqref{1eq: Fourier coefficients of upsilon},
then it follows from {\rm(\ref{1eq: Schwartz Cx}, \ref{1eq: Schwartz Cx asymptotic})} that
\begin{itemize}
\item[-]  $ \varphi _m$ satisfies
\begin{equation}\label{1eq: bounds for Fourier coefficients}
\begin{split}
x^{ \alpha + \beta} \varphi^{(\alpha)}_m (x) & \lll_{\alpha, \beta, A} (|m| + 1)^{- A} \, \text{ for all } \alpha, \beta, A \in \BN,
\end{split}
\end{equation}
\item[-] all the derivatives of $\varphi_m$ admit asymptotics

\vskip 5 pt
\item[\noindent \refstepcounter{equation}(\theequation) \label{1eq: varphi m asymptotic, 0}  \hskip 6 pt]
{ \hfill  $
\ds \varphi^{(\alpha)} _m (x) = \sum_{\kappa  = \alpha}^{\alpha + A} a_{m, \kappa } [\kappa ]_{\alpha} x^{\kappa - \alpha}  +  O_{\alpha, A } \lp (|m| + 1)^{-A} x^{ A + 1} \rp 
$ \hfill}

\vskip 5 pt
\item[ ]  as   $x \ra 0$, for any given $ \alpha, A \in \BN$, with $a_{m, \kappa } \in \BC$ satisfying $a_{m, \kappa }= 0 $ if either $\kappa < |m| $ or $\kappa \notequiv m (\mod 2)$.

\end{itemize}
Observe that \eqref{1eq: varphi m asymptotic, 0} is equivalent to the following two conditions,
\begin{itemize}
\vskip 5 pt
\item[\noindent \refstepcounter{equation}(\theequation) \label{1eq: phi m, 1}  \hskip 6 pt]
$\ds \varphi^{(\alpha)} _m (x) = \alpha ! a_{m, \alpha } + O_{\alpha } \lp x \rp $ as $x \ra 0$,  
for any $ \alpha \geq |m|$,   with $ a_{m, \alpha } \in \BC $ satisfying $ a_{m, \alpha }= 0 $ if $ \alpha \notequiv m (\mod 2)$,

\vskip 7 pt
\item[\noindent \refstepcounter{equation}(\theequation) \label{1eq: phi m, 2}  \hskip 6 pt]
for any given $ \alpha, A \in \BN$, $\varphi^{(\alpha)} _m (x) =  O_{\alpha, A } \lp (|m| + 1)^{-A} x^{ A + 1} \rp $ as $  x \ra 0$,   if $|m| > \alpha + A$.

\end{itemize}
In particular, $\varphi_m \in \mathscr S (\overline \BR _+)$.

Conversely, if a sequence $\lpp \varphi_m \rpp $ of functions in $C^\infty (\BR _+)$ satisfies \eqref{1eq: bounds for Fourier coefficients}, \eqref{1eq: phi m, 1} and \eqref{1eq: phi m, 2}, then the Fourier series defined by $\lpp \varphi_m \rpp $, that is, the right hand side of \eqref{1eq: Fourier series expansion}, is a Schwartz function on $\BC$.

{\rm (3.2).} In the polar coordinates, a smooth function $\varphi \lp x e^{i \phi}\rp \in C^\infty (\BC ^\times)$ is a Schwartz function on $\BC^\times$ if and only if $\varphi$ satisfies  \eqref{1eq: Schwartz Cx} with $\beta \in \BN$ by $\beta \in \BZ$.

Let $\varphi \in \mathscr S (\BC^\times)$ and $\varphi_m$ be the $m$-th Fourier coefficient of $ \varphi $, then it is necessary that $\varphi _m$ satisfies  \eqref{1eq: bounds for Fourier coefficients}  with $\beta \in \BN$ replaced by $\beta \in \BZ$.
In particular, $\varphi_m \in \mathscr S (\BR _+)$.

Conversely, if a sequence $\left\{\varphi_m \right \} $ of functions in $C^\infty (\BR _+)$ satisfies the condition \eqref{1eq: bounds for Fourier coefficients}  with $\beta \in \BN$ replaced by $\beta \in \BZ$, then the Fourier series defined by $\left\{\varphi_m \right \}$ gives rise to a Schwartz function on $\BC ^\times$.

Let $\varphi \in \mathscr S (\BC)$ and $a_{m, \kappa}$ be given in \eqref{1eq: Schwartz Cx asymptotic}, \eqref{1eq: varphi m asymptotic, 0} or \eqref{1eq: phi m, 1}. $\varphi \in \SS ( \BCx)$ if and only  if  $a_{m, \kappa} = 0$ for all $m \in \BZ, \kappa \in \BN$.

\end{lem}

\subsubsection{Some subspaces of $\SS (\overline \BR_+)$} \label{sec: Schwartz subspaces}

In the following, we introduce several subspaces of  $\SS (\overline \BR_+)$ which are closely related to $\SS (\BR)$ and $\SS (\BC)$.

We first define for $\delta \in \BZT $  the subspace $C^\infty_{\delta} (\overline \BR_+) \subset C^\infty (\overline \BR_+)$ of functions with an asymptotic expansion of the form $ \sum_{\kappa =0}^\infty a_{ \kappa} x^{\delta + 2 \kappa}$ at zero.

\begin{rem}\label{rem: C=C0+C1}
A question arises, ``whether $C^\infty (\overline \BR_+) = C^\infty_0 (\overline \BR_+) + C^\infty_1 (\overline \BR_+)${\rm ?}''. 

The answer is affirmative. 

To see this, we define the space $C_\delta^\infty (\BR)$ of smooth functions $\varphi$ on $\BR$ satisfying \eqref{1eq: delta condition, R}.
One has $\sgn (x)^\delta \varphi (|x|) \in C_\delta^\infty (\BR)$ if $\varphi  \in C^\infty_{\delta} (\overline \BR_+)$, and conversely, $\varphi \restriction_{\BR_+} \in  C^\infty_{\delta} (\overline \BR_+) $ if $\varphi \in C_\delta^\infty (\BR)$. Thus, with the simple observation $C^\infty ( \BR ) = C^\infty_0 ( \BR ) \oplus C^\infty_1 ( \BR )$, one sees that $C^\infty_0 (\overline \BR_+) + C^\infty_1 (\overline \BR_+)$ is the subspace  of $C^\infty (\overline \BR_+)$ consisting of functions on $\BR _+$ that admit a smooth extension onto $\BR$.

On the other hand, the Borel theorem {\rm ({\it \cite[1.5.4]{Nara}})}, which is a special case of the Whitney extension theorem  {\rm ({\it \cite[1.5.5, 1.5.6]{Nara}})}, states that for any sequence $\lpp a_\alpha\rpp $ of constants there exists a smooth function $\varphi \in C^\infty (\BR)$ such that $ \varphi^{(\alpha)} (0) = \alpha! a_\alpha $. Clearly, this theorem of Borel implies our assertion above.

In {\rm \S \ref{sec: refinements Msis Sis, R+}}, we shall give an alternative proof of this  using the Mellin transform. See Remark {\rm \ref{rem: S=S0+S1}}.
\end{rem}

We define $ \SS_\delta (\overline \BR_+) = \SS (\overline \BR_+) \cap C_\delta^\infty (\overline \BR_+)$. The following identity is obvious 
\begin{equation*}
\SS_\delta (\overline \BR_+) = x^\delta \SS_0 (\overline \BR_+).
\end{equation*}
In view of Lemma \ref{lem: Schwartz} (1.2), we have $\SS_0  (\overline \BR_+) \cap \SS_1 (\overline \BR_+) = \SS (\BR _+)$.

If we let $\SS_\delta (\BR) $ be the space of functions $\varphi \in \SS (\BR)$ satisfying \eqref{1eq: delta condition, R}, then 
$$ \SS_\delta (\BR) = \sgn (x)^\delta \SS_\delta (\overline \BR_+) = \left \{ \sgn (x)^\delta \varphi  (|x|) : \varphi \in \SS_\delta (\overline \BR_+) \right \}.$$ 
Clearly, $\SS (\BR) = \SS_0 (\BR ) \oplus \SS_1 (\BR )$.

We define the subspace $\SS_m (\overline \BR _+) \subset \SS_{\delta(m)}  (\overline \BR _+)$, with  $\delta (m) = m (\mod 2)$,  of functions with an asymptotic expansion of the form $ \sum_{\kappa = 0}^\infty a_{ \kappa} x^{|m| + 2 \kappa}$ at zero.  
We have
\begin{equation*}
\SS_m (\overline \BR_+) = x^{|m|} \SS_0 (\overline \BR_+).
\end{equation*}
If we define $\SS_m (\BC)$ to be the space of $\varphi \in \SS (\BC)$ satisfying \eqref{1eq: m condition, C}, then 
$$ \SS_m (\BC) = [z]^m \SS_m (\overline \BR_+) = \left \{ [z]^m \varphi  (|z|) = e^{im\phi} \varphi (x) : \varphi \in \SS_m (\overline \BR_+) \right \}.$$ 
The last two paragraphs in  Lemma  \ref{lem: Schwartz} (3.1) can be recapitulated as below
$$
\SS (\BC) \xrightarrow{\cong}  \left\{ \lpp \varphi_m\rpp \in \prod_{m \in \BZ} \SS_m (\overline \BR _+) : \varphi_m \text{ satisfies (\ref{1eq: bounds for Fourier coefficients}, \ref{1eq: phi m, 1}, \ref{1eq: phi m, 2})} \right \} \twoheadrightarrow  \SS_m (\overline \BR _+),
$$
where the first map sends $\varphi \in \SS (\BC)$ to the sequence $ \lpp \varphi _m \rpp_{m \in \BZ}$ of its Fourier coefficients, and the second is the $m$-th projection.
According to Lemma  \ref{lem: Schwartz} (3.1), the first map is an isomorphism, and the second projection is surjective.

\subsubsection{$\SS_\delta (\BRx)$ and $\SS_m (\BCx)$} \label{sec: S delta and Sm}
Let $\delta \in \BZ/2 \BZ$ and $m \in \BZ$. We define $\SS_\delta (\BRx) = \SS (\BRx) \cap \SS_\delta (\BR)$ and $\SS_m (\BCx) = \SS (\BCx) \cap \SS_m (\BC)$.
Clearly, $\SS_\delta (\BRx) = \sgn(x)^\delta \SS (\BR _+)$ and $\SS_m (\BCx) = [z]^m \SS (\BR _+)$.


\subsection{The Fourier transform}
According to the local theory in Tate's thesis 
for an archimedean local field $\BF$, the Fourier transform $\widehat \varphi = \EF \varphi$ of a Schwartz function $\varphi\in \SS(\BF)$ is defined by
\begin{equation}\label{1eq: Fourier, R}
\widehat \varphi(y) = \int_{\BF} \varphi(x) e(- \Lambda (xy)) d x,
\end{equation}
with
\begin{equation} \label{1eq: Lambda (x)}
  \Lambda (x) = 
\left\{ \begin{split}
  & x , \hskip 70 pt  \text{ if } \BF=\BR;\\
  & \Tr (x) = x + \overline x, \hskip 18 pt  \text{ if } \BF=\BC.
  \end{split} \right.
\end{equation}
The Schwartz space $\SS(\BF)$ is invariant under the Fourier transform. Moreover, with our choice of measure in \S \ref{sec: R+, Rx and Cx}, the following inversion formula holds
\begin{equation}\label{1eq: Fourier, C}
\widehat{\widehat {\varphi}} (x) = \varphi (-x), \hskip 10 pt x \in \BF.
\end{equation}

\subsection{The Mellin transforms $\EM $, $\EM_{\delta}$ and $\EM_m$}\label{sec: Mellin preliminaries}

Corresponding to $\BR_+$, $\BRx$ and $\BCx$, there are three kinds of Mellin transforms  $\EM $, $\EM_{\delta}$ and $\EM_m$.

\begin{defn}[Mellin transforms]
\ 

{\rm (1).} The {\rm Mellin transform} $\EM  \varphi$ of a Schwartz function $\varphi \in \mathscr S (\BR_+)$
is given by
\begin{equation} \label{1def: Mellin transform}
\EM  \varphi (s) = \int_{\BR_+} \varphi (x) x ^{s } d^\times x.
\end{equation}

{\rm (2).} For $\delta \in \BZ/ 2\BZ$, the {\rm (signed) Mellin transform $\EM _{\delta} \varphi$ with order $\delta$} of a Schwartz function $\varphi \in \mathscr S (\BR ^\times)$
is defined by
\begin{equation} \label{1def: signed Mellin transform}
\EM _{\delta} \varphi (s) = \int_{\BR^\times} \varphi (x) \sgn (x)^\delta |x|^{s } d^\times x.
\end{equation} 
Moreover, define $\EM_{\BR} = (\EM_0, \EM_1)$.

{\rm (3).} For $m \in \BZ$, the {\rm Mellin transform $\EM _m \varphi $ with order $m$} of a Schwartz function $\varphi \in \mathscr S (\BC^\times)$ is defined by
\begin{equation}\label{1def: Mellin transform over complex}
\EM _m \varphi (s) = \int_{\BC^\times} \varphi (z) [z]^{ m} \|z\| ^{\frac 12 s  }\ d^\times z = 2 \int_0^\infty \int_0^{2 \pi} \varphi \lp x e^{i \phi} \rp e^{ i m \phi} d\phi \cdot x^{ s } d^\times x.
\end{equation}
Moreover, define $\EM_{\BC} = \prod_{m \in \BZ} \EM_{- m} $.

\end{defn}

\begin{observation}\label{observ: Mellin} For $\varphi \in \mathscr S (\BRx)$, we have
\begin{equation}\label{1eq: M delta = M}
\EM _\delta \varphi (s) =  2 \EM  \varphi_{\delta}  (s), \hskip 10 pt \delta \in \BZT.
\end{equation}
Similarly, for $\varphi \in \mathscr S (\BCx)$, we have
\begin{equation}\label{1eq: Mm = M}
\EM _{-m} \varphi (s) = 4 \pi \EM  \varphi_{ m} ( s), \hskip 10 pt m \in \BZ.
\end{equation}
The relations  \eqref{1eq: M delta = M} and \eqref{1eq: Mm = M} reflect the identities $\BR^\times \cong \BR_+ \times  \{+, -\}  $ and $\BC^\times \cong \BR_+ \times \BR/ 2\pi \BZ$ respectively.
\end{observation}

\begin{lem}[Mellin inversions] \label{lem: Mellin inversions} Let $\sigma  $ be real. Denote by  $  (\sigma)  $ the vertical line from $\sigma - i \infty$ to $\sigma + i \infty$. 

{\rm (1).} For $\varphi \in \mathscr S (\BR_+)$, we have
\begin{equation}\label{1eq: Mellin inversion}
\varphi (x) =  \frac 1{2 \pi i} \int_{(\sigma)} \EuScript M \varphi (s) x^{-s} d s.
\end{equation}

{\rm (2).} For $\varphi \in \mathscr S (\BRx)$, we have
\begin{equation}\label{1eq: Mellin inversion, R}
\varphi (x) =  \frac { 1} {4 \pi i} \sum_{\delta \in \BZ/ 2\BZ } \sgn (x)^\delta \int_{(\sigma)} \EuScript M_\delta \varphi (s) |x|^{-s} d s.
\end{equation}

{\rm (3).} For $\varphi \in \mathscr S (\BCx)$, we have
\begin{equation}\label{1eq: Mellin inversion, C}
\varphi (z) =  \frac { 1} {8 \pi^2 i} \sum_{m \in \BZ}  [z]^{-m} \int_{(\sigma)} \EuScript M_{m} \varphi (s) \|z\|^{- \frac 1 2 s  } d s,
\end{equation}
or, in the polar coordinates,
\begin{equation}\label{1eq: Mellin inversion, C, polar}
\varphi \lp x e^{i \phi}\rp = \frac { 1} {8 \pi^2 i}  \sum_{m \in \BZ} e^{- i m \phi} \int_{(\sigma)} \EuScript M_{m} \varphi (s) x^{- s} d s.
\end{equation}
\end{lem}

\begin{defn}\

{\rm (1).} Let $\mathscr H_{\mathrm {rd}}$ denote the space of all entire functions $H(s)$ on the complex plane that  rapidly decay along vertical lines, uniformly on vertical strips. 

{\rm (2).} Define $\mathscr H_{\mathrm {rd}}^{\BR} = \mathscr H_{\mathrm {rd}} \times \mathscr H_{\mathrm {rd}}$.

{\rm (3).} Let $\mathscr H_{\mathrm {rd}}^{\BC}$ be the subset of $\, \prod_{\BZ} \Hrd$ consisting of sequences $\lpp H_m  (s) \rpp $ of entire functions in $\Hrd$ satisfying the following condition,
\begin{itemize}
\item[\noindent \refstepcounter{equation}(\theequation) \label{1eq: rapid decay, C} \hskip 6 pt]
for any given $ \alpha, A \in \BN$ and vertical strip $ \BS[a, b]$,  

\vskip 5 pt
{\centering $\ds  H_m  (s) \lll_{ \alpha, A, a, b} (|m|+1)^{-A}  (|\Im s| + 1)^{-\alpha} \text{ for all } s \in \BS[a, b].$}
\end{itemize}
\end{defn}


\begin{cor}\label{cor: Mellin, Schwartz}\

{\rm (1).} The Mellin transform $\EM$ and its inversion establish an isomorphism between $\SS (\BR _+)$ and $\Hrd$.

{\rm (2).} For each $\delta \in \BZT$, $\EM_\delta$ establishes an isomorphism between $\SS_\delta (\BRx)$ and $\Hrd $. Hence, $\EM_{\BR} $ establishes an isomorphism between $\SS (\BRx)$ and $\Hrd^{\BR}$.

{\rm (3).} For each $m\in \BZ$, $\EM_{-m}$ establishes an isomorphism between $\SS_m (\BCx)$ and $\Hrd $. Moreover, $\EM_{\BC}$ establishes an isomorphism between $\SS (\BCx)$ and $\Hrd^{\BC}$.
\end{cor}
\begin{proof}
(1) is a well-known consequence of Lemma \ref{lem: Mellin inversions} (1), whereas (2) directly follows from (1) and Lemma  \ref{lem: Mellin inversions} (2). As for (3), in addition to (1) and Lemma \ref{lem: Mellin inversions} (3),  Lemma \ref{lem: Schwartz} (3.2) is also required  for the rapid decay in $m$.
\end{proof}

\section{The function spaces $\Ssis (\BR_+)$, $\Ssis (\BR^\times)$ and $\Ssis (\BC^\times)$}\label{sec: all Ssis}



The goal of this section is to extend the definitions of the Mellin transforms $\EM$, $\EM_{\BF}$ and generalize the   settings in   \S \ref{sec: Mellin preliminaries} to the function spaces $\Ssis (\BR _+)$, $\Msis$, $\Ssis (\BFx)$ and $\Msis^{\BF}$. These spaces are much more sophisticated than $\SS (\BR_+)$,  $\Hrd$, $\SS (\BFx)$ and $\Hrd^\BF$ but most suitable for investigating Hankel transforms over $\BR_+$ and $\BFx$. 

We shall first construct  the function spaces $\Ssis (\BR_+)$, $\Msis$ and   establish an isomorphism between them using the Mellin transform $\EM$.
Based on these, we shall then turn to the spaces $\Ssis (\BFx)$, $\Msis^{\BF}$ and the Mellin transform $\EM_{\BF}$. 
The case $\BF = \BR$ has been worked out in \cite[\S 6]{Miller-Schmid-2006}. Since $\BRx \cong \BR_+ \times \{+, -\}$ is simply two copies of $\BR _+$, the properties of $\Ssis (\BRx)$ and $\Msis^\BR$ are in substance the same as those of $\Ssis (\BR_+)$ and $\Msis$.
In the case $\BF = \BC$, $\Ssis (\BCx)$ and $\Msis^\BC$ can be constructed in a parallel way. The study on $\BCx$ is however  much more elaborate, since $\BCx \cong \BR_+ \times \BR/ 2\pi \BZ$ and the analysis on the circle $ \BR/ 2\pi \BZ$ is also  taken into account.  


\subsection{The spaces $\Ssis (\BR_+)$  and $\Msis$}
\label{sec: Ssis and Msis, R+} \footnote{According to \cite[Definition 6.4]{Miller-Schmid-2006},  a function in $\Ssis (\BR_+)$ is said to have a \textit{simple singularity} at zero. Thus, the subscript ``{sis}'' stands for ``{simple singularity}''.}



\subsubsection{The spaces $ x^{-\lambda}  (\log x)^j \SS (\overline \BR _+)$ and $\Msis^{\lambda, j}$}
Let $\lambda \in \BC$ and $j \in \BN$. 

We define $$ x^{-\lambda}  (\log x)^j \SS (\overline \BR _+) = \left\{ x^{-\lambda}  (\log x)^j \varphi (x) : \varphi \in \SS (\overline \BR _+) \right\} .$$

We say that a meromorphic function $H(s)$ has a pole of {\it pure order} $j + 1$  at $s = \lambda $ if the principal part of $H(s)$ at $s = \lambda $ is $ a {(s - \lambda )^{-j-1}}$ for some constant $a \in \BC$.  Of course, $H (s) $ does not have a genuine pole at $s = \lambda $ if $a = 0$.
We define the space $\Msis^{\lambda, j}$ of all meromorphic functions $H(s)$ on the complex plane such that
\begin{itemize}
\item[-] the only possible singularities of $H  (s)$ are poles of pure order  $j+1$ at the points in $ \lambda - \BN = \left\{ \lambda - \kappa : \kappa \in \BN \right \}$, and
\item[-] $H (s)$ decays rapidly along vertical lines, uniformly  on vertical strips, 
that is,
\item[\noindent \refstepcounter{equation}(\theequation)\hskip 13 pt \label{2eq: Mellin rapid decay, 0}]
\hskip 5pt for any given $ \alpha \in \BN$,  vertical strip $ \BS [a, b] $ and  $ r > 0$,
$$H (s) \lll_{\lambda, j, \alpha, a, b, r} (|\Im s| + 1)^{-\alpha} \text{ for all } \textstyle s \in \BS [a, b] \smallsetminus \bigcup_{\kappa \in \BN} \BB_{r} (\lambda - \kappa).$$
\end{itemize}

The constructions of the Mellin transform $\EM$ and its inversion (\ref{1def: Mellin transform}, \ref{1eq: Mellin inversion}) identically extend from $\SS (\BR _+)$ onto $\Ssis^{\lambda, j} (\BR _+)$, except that the conditions $\Re s > \Re \lambda$ and $\sigma > \Re \lambda$ are required to guarantee convergence.

\begin{lem}\label{2lem: Mellin isomorphism, Msis lambda j}
Let $\lambda \in \BC$ and $j \in \BN$. The Mellin transform $\EM$ and its inversion establish an isomorphism of between $ x^{-\lambda}  (\log x)^j \SS (\overline \BR _+) $ and $\Msis^{\lambda, j}$.
\end{lem}

This lemma  is essentially  \cite[Lemma 6.13, Corollary 6.17]{Miller-Schmid-2006}.
Nevertheless, we shall include its proof as the reference for the constructions of $\Nsis^{\BC, \lambda, j}$ and $\Msis^{\BC}$ in \S \ref{sec: Msis C} as well as the proof of Lemma {\rm \ref{2lem: Ssis to Msis, C}}. 

\begin{proof}
Let $\upsilon (x) = x^{-\lambda} (\log x)^j \varphi (x)$ for some $\varphi \in \SS (\overline \BR _+)$. Suppose that the derivatives of $\varphi$ satisfy  \eqref{1eq: asymptotics R+, 2} and \eqref{1eq: Schwartz R+}, that is, asymptotic expansions at zero and the Schwartz condition at infinity.

\vskip 3 pt
{\textsc {Claim}} 1. Let  
\begin{equation*} 
 H (s) = \EM \upsilon (s) = \int_0^\infty \upsilon (s) x^{s-1} d x, \hskip 10 pt \Re s > \Re \lambda.
\end{equation*}
Then $H$ admits a meromorphic continuation onto the whole complex plane.
The only singularities of $H $ are poles of pure order $j+1$ at the points in $ \lambda - \BN$. More precisely, $H (s) $ has a pole at $s = \lambda - \kappa$ of principal part $ {(-)^j j! a_\kappa } {(s - \lambda + \kappa)^{-j-1}}$.
Moreover,  $H $ decays rapidly along vertical lines, uniformly on vertical strips. 
To be concrete, we have
\begin{itemize}
\item[\noindent \refstepcounter{equation}(\theequation)\hskip 13 pt \label{2eq: Mellin rapid decay}]
for any given $ \alpha, A \in \BN, b \geq a > \Re \lambda - \alpha - A - 1$   and   $ r > 0$,
\vskip 5 pt
{\centering $H (s) \lll_{\lambda, j, \alpha, A, a, b, r} (|\Im s| + 1)^{-\alpha} \text{ for all } \textstyle s \in \BS [a, b] \smallsetminus \bigcup_{\kappa = 0}^{\alpha + A} \BB_{r} (\lambda - \kappa).$}
\end{itemize}
We remark that \eqref{2eq: Mellin rapid decay, 0} and \eqref{2eq: Mellin rapid decay} are equivalent.

\vskip 5 pt

{\textsc {Proof of Claim}} 1. In view of $\EM \lp x^{-\lambda} (\log x)^j \varphi (x) \rp (s) = \EM \lp (\log x)^j \varphi (x) \rp (s - \lambda)$, one may assume $\lambda = 0$. As such, $\upsilon (x) = (\log x)^j \varphi (x)$.  

Let $A \in \BN$. We have for $\Re s > 0$
\begin{equation}\label{2eq: poles of Mellin}
\begin{split}
\EM \upsilon (s) = & \int_0^1 (\log x)^j \lp \varphi (x) - \sum_{\kappa = 0}^A a_{\kappa} x^{\kappa} \rp x^{s -1} d x + \sum_{\kappa = 0}^A \frac { (-)^j j!  a_\kappa } { \lp s + \kappa\rp^{j+1} }\\
& + \int_1^\infty (\log x)^j \varphi (x) x^{s -1} d x.
\end{split}
\end{equation}
Here, we have used $$\int_0^1 (\log x)^j x^{s -1} d x = \frac { (-)^j j! } { s^{j+1} }, \hskip 10 pt \Re s > 0.$$
In view of $\varphi (x) - \sum_{\kappa = 0}^A a_{\kappa} x^{\kappa} = O_A (x^{A+1})$, the first integral in \eqref{2eq: poles of Mellin} converges in the half-plane $\left \{ s : \Re s > - A - 1 \right\}$. The last integral converges for all $s$ on the whole complex plane due to the rapid decay  of $\varphi$. Thus $H (s) = \EM \upsilon (s)$ admits a  meromorphic extension onto $\left \{ s : \Re s > - A - 1 \right\}$ and, since $A$ was arbitrary, onto the whole complex plane, with poles of pure order  $j+1$ at the points in $- \BN$. 

For any given $\alpha \in \BN$, repeating partial integration  $\alpha $ times to the defining integral of $\EM \upsilon (s)$ yields
\begin{equation*}
(-)^\alpha (s )_{\alpha} \EM \upsilon (s) = \EM \upsilon^{(\alpha)} (s + \alpha).
\end{equation*}
In view of this, we first  expand  $\EM \upsilon^{(\alpha)} (s + \alpha)$ according to the expansion of $\upsilon^{(\alpha)} (x) = (d/dx)^\alpha \big( \lp\log x\rp^j \varphi (x) \big) $.  We then   write  each   term in the expansion of $\EM \upsilon^{(\alpha)} (s + \alpha)$ in the same fashion as \eqref{2eq: poles of Mellin} and apply  \eqref{1eq: asymptotics R+, 2} and \eqref{1eq: Schwartz R+} to estimate the first and the last integral respectively. We conclude that
\begin{equation*}
\EM \upsilon (s) \lll_{j, \alpha, A, a, b} \frac 1 { \left| (s)_{\alpha} \right|} \lp 1 + \sum_{\kappa = 0}^{\alpha + A} \lp \frac 1 {|s + \kappa|} + ... +  \frac 1 {|s + \kappa|^{j+1}} \rp \rp,
\end{equation*}
for all $s \in \BS [a, b]$, with $b \geq a > -\alpha - A - 1$. In particular, \eqref{2eq: Mellin rapid decay} is proven.

\vskip 7 pt


Let $H \in\Msis^{\lambda, j}$. Suppose that the  principal part of $H (s)$ at $s = \lambda - \kappa$ is equal to $ {(-)^j j! a_\kappa } {(s + \lambda + \kappa)^{-j-1}}$ and that $H (s)$ satisfies the condition \eqref{2eq: Mellin rapid decay}.

\vskip 3 pt

{\textsc {Claim}} 2. If we denote by $\upsilon (x)$ the following integral
\begin{equation*}
\upsilon (x) = \frac 1 {2\pi i}\int_{(\sigma)} H(s) x^{-s} d s, \hskip 10 pt \sigma > \Re \lambda,
\end{equation*}
then all the derivatives of $\varphi (x) = x^{\lambda} (\log x)^{- j}\upsilon (x)$ satisfy the asymptotics in \eqref{1eq: asymptotics R+, 3} at zero and rapidly decay at infinity.

\vskip 5 pt

{\textsc {Proof of Claim}} 2. Again, let us assume $\lambda = 0$. 

Let $1 > \rho > 0$. We left shift the  contour of integration from $  (\sigma) $ 
to $  (- \rho)  $. When moving across $s = 0$, we obtain
$a_{0} (\log x)^j$ in view of Cauchy's differentiation formula\footnote{Recall   Cauchy's differentiation formula,  $$f^{(j)} (\zeta) = \frac {j !} {2\pi i} \sideset{}{_{ \partial \BB_r (\zeta) } }\oint \frac {f(s)} {(s-\zeta)^{j+1}} d s,$$ where $f$ is a holomorphic function on a neighborhood of the closed disc $\overline \BB_r (\zeta)$ centered at $\zeta$, and the integral is  taken counter-clockwise on the circle $\partial \BB_r (\zeta) $. In the present situation, this formula is applied for $f(s) = x^{- s}$.}. It follows that
\begin{equation*}
\upsilon (x) = a_{0} (\log x)^j + \frac 1 {2\pi i} \int_{(- \rho )}  H(s) x^{-s} d s.
\end{equation*}
Using \eqref{2eq: Mellin rapid decay} with $r$ small, say $r < \rho $, to estimate the above integral, we arrive  at
\begin{equation*}
\upsilon (x) = a_{0} (\log x)^j + O \big(x^{\rho } \big) = (\log x)^j \lp a_{0} + O (x^{\rho}) \rp, \ \text{ as } x \ra 0.
\end{equation*}
Thus  $\varphi (x) = (\log x)^{- j} \upsilon (x)$ satisfies the asymptotic \eqref{1eq: asymptotics R+, 3} with $\alpha = 0$. For the general case $\alpha \in \BN$, we have
\begin{equation}\label{2eq: upsilon (alpha)}
\upsilon^{(\alpha)} (x) = (-)^\alpha \frac {1} {2\pi i} \int_{(\sigma)} (s)_{\alpha} H(s) x^{-s - \alpha} d s.
\end{equation}
Shifting the contour from  $  (\sigma) $ to  $ (- \alpha - \rho) $ and following the same lines of arguments as above, combined with some straightforward algebraic manipulations, one may show \eqref{1eq: asymptotics R+, 3} by an induction. 

We are left to show the Schwartz condition for $\varphi (x) = (\log x)^{- j} \upsilon (x)$, or equivalently, that for $\upsilon (x)$. Indeed, the bound \eqref{1eq: Schwartz R+} for $\upsilon^{(\alpha)} (x)$ follows from right shifting the contour of the integral in \eqref{2eq: upsilon (alpha)} to the vertical line $ (\beta) $ 
and applying the 
estimates in \eqref{2eq: Mellin rapid decay}. 
\end{proof}

\subsubsection{The spaces $\Ssis (\BR_+)$ and $\Msis $}

Let $\lambda , \lambda' \in \BC$. We write $\lambda \preccurlyeq_1 \lambda'$ if $\lambda'  - \lambda  \in \BN$ and  
 $\lambda \sim_1 \lambda'$ if $\lambda'  - \lambda  \in \BZ$. Observe that ``$\preccurlyeq_1$'' and ``$\sim_1$'' define an  order relation and an equivalence relation on $\BC$ respectively.

Define
\begin{equation*}
\Ssis (\BR _+) = \sum_{\lambda \in \BC} \sum_{j \in \BN} x^{-\lambda}  (\log x)^j \SS (\overline \BR _+),
\end{equation*}
where the sum $\sum_{\lambda \in \BC}  \sum_{j \in \BN}$ is in the {\it algebraic} sense.
It is clear that
$\lambda  \preccurlyeq_1 \lambda' $ if and only if $ x^{-\lambda}  (\log x)^j \SS (\overline \BR _+) \subseteq  x^{-\lambda'}  (\log x)^{j} \SS (\overline \BR _+)$. 
One also observes that $ x^{-\lambda}  (\log x)^j \SS (\overline \BR _+) \cap  x^{-\lambda'}  (\log x)^{j'} \SS (\overline \BR _+) = \SS (\BR_+)$ if either $j \neq j'$ or $\lambda \nsim_1 \lambda'$. Therefore,
\begin{equation}\label{2eq: decompose Ssis (R+)}
\begin{split}
\Ssis (\BR_+)/ \SS (\BR _+) = \bigoplus_{ \omega\, \in \BC/ {\scriptscriptstyle \sim_1} } \bigoplus_{j \in \BN} \ \varinjlim_{\lambda \in \omega}  \big( x^{-\lambda}  (\log x)^j \SS (\overline \BR _+) \big) \big/\SS (\BR _+).
\end{split}
\end{equation}
Here the direct limit $\varinjlim_{\lambda \in \omega} $ is taken on the totally ordered set $(\omega, \preccurlyeq_1)$ and may be simply viewed as the union $ \bigcup_{\lambda \in \omega}$.
More precisely, each function $\upsilon \in \Ssis (\BR_+)$ can be expressed as a sum
\begin{equation*}
\upsilon (x) = \upsilon^{0} (x) + \sum_{\lambda \in \Lambda} \sum_{j=0}^{N } x^{- \lambda } (\log x)^j \upsilon_{\lambda, j } (x),
\end{equation*}
with $\Lambda \subset \BC$ a finite set such that $\lambda  \nsim_1 \lambda'$ for any two distinct points $\lambda, \lambda' \in \Lambda$,  $ N \in \BN$, $\upsilon^{0} \in \SS (\BR _+)$ and $\upsilon_{\lambda, j } \in \SS (\overline \BR _+)$. 
This expression is unique up to addition of Schwartz functions in $\SS (\BR _+)$.

On the other hand, we define the space $\Msis$ of all meromorphic functions $H $   satisfying the following conditions,
\begin{itemize}
\item[-] the poles of $H$ lie in a finite number of sets $ \lambda - \BN $, 
\item[-] the orders of the poles of $H$ are uniformly bounded, and
\item[-] $H$ decays rapidly along vertical lines, uniformly on vertical strips. 
\end{itemize}
Appealing to certain Gamma identities for  the Gamma function in \cite[Lemma 6.24]{Miller-Schmid-2006}, one may show, in the same way as \cite[Lemma 6.35]{Miller-Schmid-2006}, that
\begin{equation*}
\Msis = \sum_{\lambda \in \BC} \sum_{j \in \BN} \Msis^{\lambda, j}.
\end{equation*}
We have $\Msis^{\lambda, j} \subseteq \Msis^{\lambda', j}$ if and only if $\lambda \preccurlyeq_1 \lambda'$, and $\Msis^{\lambda, j} \cap \Msis^{\lambda', j'} = \Hrd $ if either $j \neq j'$ or $\lambda \nsim \lambda'$. Therefore
\begin{equation}\label{2eq: decompose Msis}
\Msis / \Hrd
= \bigoplus_{\omega\, \in \BC/ {\scriptscriptstyle \sim_1} } \bigoplus_{j \in \BN} \ \varinjlim_{\lambda \in \omega}  \Msis^{\lambda, j} \big/ \Hrd.
\end{equation}

The following lemma is a direct consequence of Lemma \ref{2lem: Mellin isomorphism, Msis lambda j}.
\begin{lem}\label{2lem: Sis and Msis, R+}
The Mellin transform $\EM$ is an isomorphism between $\Ssis (\BR _+)$ and $\Msis$ which respects their decompositions \eqref{2eq: decompose Ssis (R+)} and \eqref{2eq: decompose Msis}.
\end{lem}

\subsubsection{More refined decompositions of $\Ssis (\BR _+)$ and $\Msis$}\label{sec: refinements Msis Sis, R+}

Alternatively, we define  an order relation on $\BC$,  $\lambda  \preccurlyeq_2 \lambda' $ if $\lambda'  - \lambda  \in 2 \BN$, as well as  an equivalence relation, $\lambda \sim_2 \lambda'$ if $\lambda'  - \lambda  \in 2 \BZ$. 


Define  $\Nsis^{\lambda, j}$ in the same way as $\Msis^{\lambda, j}$ with $ \lambda - \BN$ replaced by $ \lambda - 2 \BN$. Under the isomorphism via $\EM$ in Lemma \ref{2lem: Mellin isomorphism, Msis lambda j}, $\Nsis^{\lambda, j}$ is then isomorphic to $x^{-\lambda} (\log x)^j \SS_0 (\overline \BR _+)$.

According to \cite[Lemma 6.35]{Miller-Schmid-2006}, we have the following decomposition,
\begin{equation} \label{2eq: Msis = Nsis + Nsis}
\Msis^{\lambda, j} / \Hrd = \Nsis^{\lambda, j}  / \Hrd \oplus \Nsis^{\lambda-1, j}  / \Hrd.
\end{equation}
Inserting this into \eqref{2eq: decompose Msis}, we obtain   the following refined decomposition of $\Msis / \Hrd$
\begin{equation*}
\begin{split}
 \bigoplus_{\omega\, \in \BC/ {\scriptscriptstyle \sim_1} } \bigoplus_{j \in \BN} \ \varinjlim_{\lambda \in \omega} \lp \Nsis^{\lambda, j} \big/ \Hrd \oplus 
  \Nsis^{\lambda-1, j} \big/ \Hrd \rp 
= \bigoplus_{\omega\, \in \BC/ {\scriptscriptstyle \sim_2} } \bigoplus_{j \in \BN} \ \varinjlim_{\lambda \in \omega}  \Nsis^{\lambda, j} \big/ \Hrd.
\end{split}
\end{equation*}
Under the isomorphism via $\EM$ in Lemma \ref{2lem: Sis and Msis, R+}, the reflection of this refinement  on the decomposition of $\Ssis (\BR _+)/ \SS (\BR _+)$ is
\begin{equation*} 
\begin{split}
\bigoplus_{\omega\, \in \BC/ {\scriptscriptstyle \sim_2} } \bigoplus_{j \in \BN} \ \varinjlim_{\lambda \in \omega} \big( x^{-\lambda} (\log x)^j \SS_0 (\overline \BR _+) \big) \big/\SS (\BR _+).
\end{split}
\end{equation*}

\begin{lem}\label{2lem: refined decomp, R+}
We have the following refinements of the decompositions {\rm (\ref{2eq: decompose Ssis (R+)}, \ref{2eq: decompose Msis})},
\begin{equation} \label{2eq: refined, Ssis (R+)}
\begin{split}
\Ssis (\BR_+)/ \SS (\BR _+) = \bigoplus_{\omega\, \in \BC/ {\scriptscriptstyle \sim_2} } \bigoplus_{j \in \BN} \ \varinjlim_{\lambda \in \omega}  \big( x^{-\lambda} (\log x)^j \SS_0 (\overline \BR _+) \big) \big/\SS (\BR _+).
\end{split}
\end{equation}
\begin{equation}\label{2eq: refined, Msis}
\Msis / \Hrd
= \bigoplus_{\omega\, \in \BC/ {\scriptscriptstyle \sim_2} }  \bigoplus_{j \in \BN} \ \varinjlim_{\lambda \in \omega} \Nsis^{\lambda, j} \big/ \Hrd.
\end{equation}
The Mellin transform $\EM$ respects these two decompositions.
\end{lem}

\begin{cor}\label{2cor: refined decomp, R+} Let $\delta \in \BZT $ and $m \in \BZ$, and recall the definitions of  $\SS_\delta (\overline \BR _+)$ and $\SS_m (\overline \BR _+)$ in {\rm \S \ref{sec: Schwartz subspaces}}.

{\rm (1).}  The Mellin transform $\EM$ respects the following decompositions,
\begin{equation}\label{2eq: refined, Ssis (R+), delta}
\begin{split}
\Ssis (\BR_+)/ \SS (\BR _+) = \bigoplus_{\omega\, \in \BC/ {\scriptscriptstyle \sim_2} } \bigoplus_{j \in \BN} \ \varinjlim_{\lambda \in \omega}  \big( x^{-\lambda} (\log x)^j \SS_\delta (\overline \BR _+) \big) \big/\SS (\BR _+),
\end{split}
\end{equation}
\begin{equation}\label{2eq: refined, Msis, delta}
\Msis / \Hrd
= \bigoplus_{\omega\, \in \BC/ {\scriptscriptstyle \sim_2} } \bigoplus_{j \in \BN} \ \varinjlim_{\lambda \in \omega} \Nsis^{\lambda - \delta, j} \big/ \Hrd.
\end{equation}

{\rm (2).} The Mellin transform $\EM$ respects the following decompositions,
\begin{equation}\label{2eq: refined, Ssis (R+), m}
\begin{split}
\Ssis (\BR_+)/ \SS (\BR _+) = \bigoplus_{\omega\, \in \BC/ {\scriptscriptstyle \sim_2} } \bigoplus_{j \in \BN} \ \varinjlim_{\lambda \in \omega}  \big( x^{-\lambda} (\log x)^j \SS_m (\overline \BR _+) \big) \big/\SS (\BR _+),
\end{split}
\end{equation}
\begin{equation}\label{2eq: refined, Msis, m}
\Msis / \Hrd
= \bigoplus_{\omega\, \in \BC/ {\scriptscriptstyle \sim_2} } \bigoplus_{j \in \BN} \ \varinjlim_{\lambda \in \omega} \Nsis^{\lambda - |m|, j} \big/ \Hrd.
\end{equation}
\end{cor}

\begin{proof}
These follow from Lemma \ref{2lem: refined decomp, R+} in conjunction with  $x^{\delta} \SS_0 (\overline \BR_+) = \SS_\delta (\overline \BR_+) $ and $x^{|m|} \SS_0 (\overline \BR_+) = \SS_m (\overline \BR_+)$.
\end{proof}

\begin{rem}\label{rem: S=S0+S1}
Set $\lambda = 0$ and $j = 0$ in \eqref{2eq: Msis = Nsis + Nsis}. It follows from the isomorphism $\EM$ the decomposition as below,
\begin{equation*}
\SS  (\overline \BR _+) /\SS (\BR _+) = \SS_0 (\overline \BR _+) /\SS (\BR _+) \oplus x \SS_0 (\overline \BR _+) /\SS (\BR _+).
\end{equation*}
Since $x \SS_0 (\overline \BR_+) = \SS_1 (\overline \BR_+) $, one obtains $\SS (\overline \BR _+) = \SS_0 (\overline \BR _+) + \SS_1 (\overline \BR _+)$ and therefore $C^\infty (\overline \BR _+) = C^\infty_0 (\overline \BR _+) + C^\infty_1 (\overline \BR _+)$. See Remark {\rm \ref{rem: C=C0+C1}}.
\end{rem}

\subsection{The spaces $\Ssis (\BRx)$ and $\Msis^\BR$}\label{sec: Ssis and Msis, R}


Following \cite[(6.10)]{Miller-Schmid-2006}, we  write $(\lambda , \delta ) \preccurlyeq ( \lambda' , \delta')$ if $\lambda'  - \lambda  \in \BN$ and $\lambda'  - \lambda  \equiv \delta' + \delta  (\mod 2)$ and $(\lambda , \delta ) \sim ( \lambda' , \delta')$ if $\lambda'  - \lambda - (\delta' + \delta) \in 2 \BZ$. Again, these define an order relation and an equivalence relation on  $\BC \times \BZT$.

\subsubsection{The space $\Ssis (\BRx)$}\label{sec: defn Ssis (Rx)}



According to  \cite[Definition 6.4]{Miller-Schmid-2006} and \cite[Lemma 6.35]{Miller-Schmid-2006}, define
\begin{equation*}
\Ssis (\BRx) =  \sum_{\delta \in \BZT} \ \sum_{\lambda \in \BC } \ \sum_{j \in \BN}  \sgn(x)^{\delta} |x|^{-\lambda} (\log |x|)^j \SS (\BR).
\end{equation*}
We have the following decomposition,
\begin{align*}
  \Ssis (\BRx)/ \SS (\BRx)  
=   \bigoplus_{\omega\, \in \BC \times \BZT / {\scriptscriptstyle \sim } } \bigoplus_{j \in \BN}  \ \varinjlim_{(\lambda, \delta)\, \in \omega} \big( \sgn(x)^{\delta} |x|^{-\lambda} (\log |x|)^j \SS (\BR) \big) \big/\SS (\BRx).
\end{align*}
It follows from $\sgn(x) |x| \SS (\BR) = x \SS (\BR)  \subset \SS (\BR)$ that
\begin{equation*}
\Ssis (\BRx)/ \SS (\BRx) = \bigoplus_{\omega\, \in \BC/ {\scriptscriptstyle \sim_2} } \bigoplus_{j \in \BN} \ \varinjlim_{\lambda \in \omega}  \big(  |x|^{-\lambda} (\log |x|)^j \SS (\BR) \big) \big/\SS (\BRx).
\end{equation*}
We let $\Ssis^\delta (\BRx)$ denote the space of functions $\upsilon \in \Ssis (\BRx)$ satisfying the parity condition \eqref{1eq: delta condition, R}. Clearly, $\Ssis (\BRx) = \Ssis^0 (\BRx) \oplus \Ssis^1 (\BRx)$. Then, 
\begin{equation*}
\Ssis^\delta (\BRx)/ \SS_\delta (\BRx) = \bigoplus_{\omega\, \in \BC/ {\scriptscriptstyle \sim_2} } \bigoplus_{j \in \BN} \ \varinjlim_{\lambda \in \omega}    \big(  |x|^{-\lambda} (\log |x|)^j \SS_\delta (\BR) \big) \big/\SS_\delta (\BRx),
\end{equation*}
where $\SS_\delta (\BR)$ and $\SS_\delta (\BRx)$ are defined in \S \ref{sec: Schwartz subspaces} and \S \ref{sec: S delta and Sm} respectively.
Since 
$ \SS_\delta (\BR) = \sgn(x)^\delta \SS_\delta (\overline \BR _+)  $, 
\begin{equation}\label{2eq: decompose Ssis (Rx), 2}
\Ssis^\delta (\BRx)/ \SS_\delta (\BRx) = \bigoplus_{\omega\, \in \BC/ {\scriptscriptstyle \sim_2} } \bigoplus_{j \in \BN} \ \varinjlim_{\lambda \in \omega}  \big( \sgn(x)^{\delta } |x|^{-\lambda} (\log |x|)^j \SS_\delta (\overline \BR_+) \big) \big/\SS_\delta (\BRx).
\end{equation}
Then, $  |x|^{- \delta} \SS_{\delta} (\overline \BR_+) = \SS_{0} (\overline \BR_+)$ together with $ \SS_0 (\overline \BR_+)/\SS (\BR _+) \oplus |x| \SS_0 (\overline \BR_+)/\SS (\BR _+) = \SS (\overline \BR_+)/\SS (\BR _+)$ (see Remark \ref{rem: S=S0+S1}) yields
\begin{equation}\label{2eq: decompose Ssis (Rx)}
\Ssis^\delta (\BRx)/ \SS_\delta (\BRx) = \bigoplus_{\omega\, \in \BC/ {\scriptscriptstyle \sim_1} } \bigoplus_{j \in \BN} \ \varinjlim_{\lambda \in \omega}  \big( \sgn(x)^{\delta } |x|^{-\lambda} (\log |x|)^j \SS (\overline \BR_+) \big) \big/\SS_\delta (\BRx).
\end{equation}
In particular,
\begin{equation*} 
\Ssis^\delta (\BRx) = \sgn(x)^{\delta } \Ssis (\BR_+) = \left \{ \sgn (x)^\delta \upsilon  (|x|) : \upsilon \in \Ssis ( \BR_+) \right \}.
\end{equation*}

\subsubsection{The space $\Msis^\BR $}
We simply define $\Msis^{\BR} = \Msis \times \Msis$. 

\subsubsection{Isomorphism between  $\Ssis (\BRx)$ and $\Msis^\BR$ via the Mellin transform $\EM_{\BR}$}

Let  $\upsilon \in \Ssis (\BRx)$. Since $\upsilon_\delta \in \Ssis (\overline \BR_ +)$, the identity 
$$\EM_\delta \upsilon (s) = 2 \EM \upsilon_\delta (s)$$ 
extends the definition of the Mellin transform $\EM_\delta$ onto the space  $\Ssis (\BRx)$. Therefore, as a consequence of Lemma \ref{2lem: Mellin isomorphism, Msis lambda j}, Lemma \ref{2lem: Sis and Msis, R+} and Corollary \ref{2cor: refined decomp, R+} (1), the following lemma is readily established.

\begin{lem}\label{2lem: Ssis to Msis, R}
For $\delta \in \BZT$, the Mellin transform $\EM_\delta$ establishes an isomorphism between the spaces $\Ssis^\delta (\BRx)$ and $\Msis$ which respects their decompositions \eqref{2eq: decompose Ssis (Rx)} and \eqref{2eq: decompose Msis}  as well as \eqref{2eq: decompose Ssis (Rx), 2} and \eqref{2eq: refined, Msis, delta}. Therefore, $\EM^\BR = (\EM_0, \EM_1)$ establishes an isomorphism between $ \Ssis (\BRx) = \Ssis^0 (\BRx) \oplus \Ssis^1 (\BRx)$ and $\Msis^\BR = \Msis \times \Msis$.
\end{lem}


\subsubsection{An alternative decomposition of $\Ssis^\delta (\BRx)$}\label{sec: alternative decomp, R} 

The following lemma follows from Corollary \ref{2cor: refined decomp, R+} (1) (compare \cite[Corollary 6.17]{Miller-Schmid-2006}).

\begin{lem}\label{2lem: Ssis delta, M delta}
Let $\delta \in \BZ/2 \BZ$. 
The Mellin transform $\EM_\delta$ respects the following  decompositions,\begin{equation}\label{2eq: refinement Ssis, R}
\begin{split}
& \ \Ssis^\delta (\BRx)/ \SS_\delta (\BRx) \\
= & \bigoplus_{\omega\, \in \BC \times \BZT/ {\scriptscriptstyle \sim } } \bigoplus_{j \in \BN} \ \varinjlim_{(\lambda, \epsilon)\, \in \omega} \big( \sgn(x)^{ \epsilon} |x|^{-\lambda} (\log |x|)^j \SS_{ \epsilon + \delta} ( \BR ) \big) \big/\SS_\delta (\BRx),
\end{split}
\end{equation}
\begin{equation}\label{2eq: refinement Msis, R}
\Msis / \Hrd
= \bigoplus_{\omega\, \in \BC \times \BZT/ {\scriptscriptstyle \sim } } \bigoplus_{j \in \BN} \ \varinjlim_{(\lambda, \epsilon)\, \in \omega} \Nsis^{\lambda - (\epsilon + \delta), j} \big/ \Hrd.
\end{equation}
\end{lem}

\subsection{The spaces $\Ssis (\BCx)$ and $\Msis^\BC$}\label{sec: Ssis (Cx) and Msis C}

We write $(\lambda , m ) \preccurlyeq ( \lambda' , m')$ if $\lambda'  - \lambda  \in |m'-m| + 2\BN$ and  $(\lambda , m) \sim ( \lambda' , m')$ if $\lambda'  - \lambda - |m' - m| \in 2 \BZ$.  These define an order relation and  an equivalence relation on $\BC \times \BZ$. 

\subsubsection{The space $\Ssis (\BCx)$}
In parallel to \S \ref{sec: defn Ssis (Rx)}, we first define
\begin{equation*}
\Ssis (\BCx) =  \sum_{m \in \BZ } \ \sum_{\lambda \in \BC } \ \sum_{j \in \BN}  [z]^{-m} |z|^{-\lambda} (\log |z|)^j \SS (\BC).
\end{equation*}
We have the following decomposition,
\begin{equation*}
\begin{split}
\Ssis (\BCx)/ \SS (\BCx) =  \bigoplus_{\omega\, \in \BC \times \BZ / {\scriptscriptstyle \sim } } \bigoplus_{j \in \BN}  \ \varinjlim_{(\lambda, m)\, \in \omega}   \big( [z]^{-m} |z|^{-\lambda} (\log |z|)^j \SS (\BC) \big) \big/\SS (\BCx).
\end{split}
\end{equation*}
It follows from $[z] |z| \SS (\BC) = z \SS (\BC)  \subset \SS (\BC)$ that
\begin{equation}\label{2eq: decompose Ssis (Cx)}
\Ssis (\BCx)/ \SS (\BCx) =  \bigoplus_{\omega\, \in \BC/ {\scriptscriptstyle \sim_2} } \bigoplus_{j \in \BN} \ \varinjlim_{\lambda \in \omega}  \big(  |z|^{-\lambda} (\log |z|)^j \SS (\BC) \big) \big/\SS (\BCx).
\end{equation}
We let $\Ssis^m (\BCx)$ denote the space of functions $\upsilon \in \Ssis (\BCx)$ satisfying \eqref{1eq: m condition, C}.
Then,
\begin{equation*}
\Ssis^m (\BCx)/ \SS_m (\BCx) =  \bigoplus_{\omega\, \in \BC/ {\scriptscriptstyle \sim_2} } \bigoplus_{j \in \BN} \ \varinjlim_{\lambda \in \omega}   \big(  |z|^{-\lambda} (\log |z|)^j \SS_m (\BC) \big) \big/\SS_m (\BCx),
\end{equation*}
where $\SS_m (\BC)$ and $\SS_m (\BCx)$ are defined in \S \ref{sec: Schwartz subspaces} and \S \ref{sec: S delta and Sm} respectively.
Since 
$ \SS_m (\BC) = [z]^m \SS_m (\overline \BR _+) $, 
\begin{equation}\label{2eq: decompose Ssis m (Cx), 2}
\Ssis^m (\BCx)/ \SS_m (\BCx) =  \bigoplus_{\omega\, \in \BC/ {\scriptscriptstyle \sim_2} } \bigoplus_{j \in \BN} \ \varinjlim_{\lambda \in \omega} \big( [z]^{m } |z|^{-\lambda} (\log |z|)^j \SS_m (\overline \BR_+) \big) \big/\SS_m (\BCx).
\end{equation}
Then, $  |z|^{- |m|} \SS_{m} (\overline \BR_+) = \SS_{0} (\overline \BR_+)$ together with  $ \SS_0 (\overline \BR_+)/\SS (\BR _+) \oplus |z| \SS_0 (\overline \BR_+)/\SS (\BR _+) = \SS (\overline \BR_+)/\SS (\BR _+)$  yields
\begin{equation}\label{2eq: decompose Ssis m (Cx)}
\Ssis^m (\BCx)/ \SS_m (\BCx) =  \bigoplus_{\omega\, \in \BC/ {\scriptscriptstyle \sim_1} } \bigoplus_{j \in \BN} \ \varinjlim_{\lambda \in \omega} \big( [z]^{m } |z|^{-\lambda} (\log |z|)^j \SS (\overline \BR_+) \big) \big/\SS_m (\BCx).
\end{equation}
In particular,
\begin{equation*} 
\Ssis^m (\BCx) = [z]^{m } \Ssis (\BR_+) = \left \{ [z]^m \upsilon  (|z|) : \upsilon \in \Ssis ( \BR_+) \right \}.
\end{equation*}

\subsubsection{The space $\Msis^\BC$}\label{sec: Msis C}

For $\lambda \in \BC $ and $j \in \BN$, we define the space  $\Nsis^{\BC, \lambda, j}$ of all sequences $\lpp H_m (s) \rpp $ of meromorphic functions  such that
\begin{itemize}
\item[-] the only singularities of $H_m$ are poles of pure order  $j+1$ at the points in  $ \lambda - |m| - 2 \BN $,
\item[-] Each $H_m $ decays rapidly along vertical lines, uniformly on vertical strips 
(see \eqref{2eq: Mellin rapid decay, 0}), and
\item[-] $H_m (s)$ also decays rapidly with respect to $m$, uniformly on vertical strips, 
in the sense that
\item[\noindent \refstepcounter{equation}(\theequation)  \label{2eq: rapid decay Nsis, C, 0} \hskip 6 pt] for any given $ \alpha, A \in \BN $ and vertical strip $ \BS [a, b]$,
\vskip 5 pt 
\noindent $H_m (s) \lll_{ \lambda, j, \alpha, A, a, b } (|m| + 1)^{-A} (|\Im s| + 1)^{-\alpha} \ \text{ for all }  s \in \BS[a, b],$
if $|m| > \Re \lambda - a$.
\end{itemize}

Observe that the first two conditions amount to $H_m \in \Nsis^{ \lambda - |m|, j}$. Therefore, $\Nsis^{\BC, \lambda, j} \subset \prod_{m\in \BZ} \Nsis^{ \lambda - |m|, j}$.


Define   the space $\Msis^{\BC}$ of all sequences $\lpp H_m \rpp $ of meromorphic functions  such that
\begin{itemize}
\item[-] the poles of each $H_m$ lie in  $ \lambda - |m| - 2 \BN $, for a finite number of $\lambda$,
\item[-] the orders of the poles of $H_m$ are uniformly bounded,
\item[-] Each $H_m $ decays rapidly along vertical lines, uniformly on vertical strips, 
and
\item[-] $H_m $ decays rapidly with respect to $m$, uniformly on vertical strips. 
\end{itemize}
Using the refined Stirling's asymptotic formula \eqref{1eq: Stirling's formula} in  place of \cite[(6.22)]{Miller-Schmid-2006} and the following bound  in  place of \cite[(6.23)]{Miller-Schmid-2006}
\begin{equation*}
\begin{split}
\left|\frac { \Gamma^{(j)} \lp \frac 1 2 {(s - \lambda + |m| )}    \rp } { \Gamma \lp \frac 1 2 {(s + |m|) }    \rp } \right| \lll_{\lambda, j, a, b, r}  
\lp |\Im s| + |m| + 1 \rp^{- \frac 1 2 {\Re \lambda}  }
\end{split}
\end{equation*}
for $\lambda \in \BC$, $j \in \BN$,  $  s \in \BS [a, b]  \smallsetminus \bigcup_{\sstyle \kappa \geq |m| \atop \sstyle \kappa \equiv m (\mod 2)} \BB_{r} (\lambda - \kappa)$, with $ r > 0$, we may follow the same lines of the proofs of \cite[Lemma 6.24]{Miller-Schmid-2006} and \cite[Lemma 6.35]{Miller-Schmid-2006} to show that 
\begin{equation*}
\Msis^\BC = \sum_{\lambda \in \BC} \sum_{j \in \BN} \Nsis^{\BC, \lambda, j},
\end{equation*}
and consequently
\begin{equation}\label{2eq: decompose Msis, C}
\Msis^\BC / \Hrd^\BC
=  \bigoplus_{\omega\, \in \BC/ {\scriptscriptstyle \sim_2} } \bigoplus_{j \in \BN} \ \varinjlim_{\lambda \in \omega} \Nsis^{\BC, \lambda, j} \big/ \Hrd^\BC.
\end{equation}

\subsubsection{Isomorphism between $\Ssis (\BCx)$ and $\Msis^\BC$ via the Mellin transform $\EM_{\BC}$}
 
For $\upsilon \in \Ssis (\BCx)$, its $m$-th Fourier coefficient $\upsilon_m$ is a function in $ \Ssis (\overline \BR_ +)$. Hence the identity 
$$\EM_{-m} \upsilon (s) = 4 \pi \EM \upsilon_m (s)$$ extends the definition of the Mellin transform $\EM_{-m}$ onto the space $\Ssis (\BCx)$.

\begin{lem}\label{2lem: Ssis to Msis, C}
For $m \in \BZ$, the Mellin transform $\EM_{-m}$ establishes an isomorphism between the spaces $\Ssis^{ m} (\BCx)$ and $\Msis$ which respects their decompositions \eqref{2eq: decompose Ssis m (Cx)} and \eqref{2eq: decompose Msis} as well as \eqref{2eq: decompose Ssis m (Cx), 2} and \eqref{2eq: refined, Msis, m}.
Furthermore, $\EM_{\BC} = \prod_{m\in \BZ} \EM_{-m}$ establishes an isomorphism between $ |z|^{-\lambda} (\log |z|)^j \SS (\BC ) $ and $\Nsis^{\BC, \lambda, j}$ for any $\lambda \in \BC$ and $j \in \BN$, and hence an isomorphism between $ \Ssis(\BCx) $ and $\Msis^\BC$ which respects their decompositions  \eqref{2eq: decompose Ssis (Cx)} and \eqref{2eq: decompose Msis, C}.
\end{lem}

\begin{proof}
For $\upsilon \in \Ssis^m (\BCx)$, one has $\upsilon \lp x e^{i \phi}\rp = e^{im \phi} \upsilon_m (x)$ and  $\upsilon_m \in \Ssis (\overline \BR_ +)$. Thus the first assertion follows immediately from Lemma \ref{2lem: Sis and Msis, R+} and Corollary \ref{2cor: refined decomp, R+} (2).

Now let $\varphi \in \SS (\BC)$ and $ \upsilon (z) = |z|^{-\lambda} (\log |z|)^j \varphi (z)$. Clearly, their $m$-th Fourier  coefficients are related by $\upsilon_m (x) = x^{-\lambda} (\log x)^j \varphi_{ m} (x)$.
Since $\varphi_{ m} \in  \SS_m (\overline \BR _+)$, it follows from Corollary \ref{2cor: refined decomp, R+} (2) that $H_{ m} = \EM_{-m} \upsilon = 4\pi \EM \upsilon_m$ lies in $ \Nsis^{\lambda - |m|, j}$, and therefore we are left to show \eqref{2eq: rapid decay Nsis, C, 0}. Recall that in the proof of Lemma \ref{2lem: Mellin isomorphism, Msis lambda j} we turned to verify \eqref{2eq: Mellin rapid decay} instead of \eqref{2eq: Mellin rapid decay, 0}. Likewise, it is more convenient to verify the following equivalent statement of \eqref{2eq: rapid decay Nsis, C, 0},
\begin{itemize}
\item[\noindent \refstepcounter{equation}(\theequation) \label{2eq: rapid decay Nsis, C, 1} \hskip 6 pt] for any given $\alpha,  A \in \BN, b \geq a > \Re \lambda - \alpha - A - 1$, 
\vskip 5 pt 
{\noindent $H_m (s) \lll_{ \lambda, j, \alpha, A, a, b } (|m| + 1)^{-A} (|\Im s| + 1)^{-\alpha} \ \text{ for all }  s \in \BS[a, b] $,  if $ |m| > \alpha + A$.}
\end{itemize}

According to Lemma \ref{lem: Schwartz} (3.1), $\varphi_{ m}$ satisfies the conditions (\ref{1eq: bounds for Fourier coefficients}, \ref{1eq: phi m, 2}). 
Suppose $|m| > \alpha + A$. One directly applies \eqref{1eq: bounds for Fourier coefficients} and \eqref{1eq: phi m, 2} to bound the following integral by a constant multiple of $(|m| + 1)^{- A}$,
\begin{equation*}
(-)^\alpha (s - \lambda )_{\alpha} \EM \upsilon_m (s) = \int_0^\infty \frac {d^\alpha} {d x^\alpha} \lp  (\log x)^j \varphi_{ m} (x) \rp  x^{s - \lambda + \alpha - 1} d x.
\end{equation*}
This proves \eqref{2eq: rapid decay Nsis, C, 1} for $H_{ m} = 4\pi \EM \upsilon_m$.
Therefore, the sequence $\lpp \EM_{-m} \upsilon \rpp $ belongs to $\Nsis^{\BC, \lambda, j}$. 

Conversely, let $\lpp H_m \rpp  \in \Nsis^{\BC, \lambda, j}$, and let $ 4 \pi \upsilon_m$ be the Mellin inversion of $H_{ m}$, \begin{equation*} 
\upsilon_m (x) = \frac 1 {8\pi^2 i}\int_{(\sigma)} H_{ m} (s) x^{-s} d s, \hskip 10 pt \sigma > \Re \lambda - |m|.
\end{equation*} 
Since $H_{ m} \in \Nsis^{\lambda - |m|, j}$, Corollary \ref{2cor: refined decomp, R+} (2) implies that $\upsilon_m (x) \in  x^{- \lambda} (\log x)^j  \SS_m (\overline \BR _+)$ and hence $\varphi_m (x) = x^{ \lambda} (\log x)^{-j} \upsilon_m (x)$ lies in $\SS_m (\overline \BR _+)$.
This proves \eqref{1eq: phi m, 1}.
Similar to the proof of Lemma \ref{2lem: Mellin isomorphism, Msis lambda j}, right shifting of the contour of integration combined with \eqref{2eq: rapid decay Nsis, C, 1} yields \eqref{1eq: bounds for Fourier coefficients}, whereas left shifting combined with \eqref{2eq: rapid decay Nsis, C, 1} yields \eqref{1eq: phi m, 2}\footnote{ Actually, $O_{\alpha, A} \lp (|m| + 1)^{- A} x^{A+1} \rp$ in  \eqref{1eq: phi m, 2} should be replaced by  $O_{\alpha, A,  \rho} \lp (|m| + 1)^{- A} x^{A+ \rho} \rp$, $1 >  \rho > 0$. Moreover, one observes that the left contour shift here does not cross any pole.}. 

The proof of the second assertion is completed.
\end{proof}


\subsubsection{An alternative decomposition of $\Ssis^m (\BCx)$}\label{sec: alternative decomp, C}

The following lemma follows from Corollary \ref{2cor: refined decomp, R+} (2).

\begin{lem}\label{2lem: Ssis m, M -m}
Let $m \in \BZ $. 
The Mellin transform $\EM_{-m} $ respects the following  decompositions,\begin{equation}\label{2eq: refinement Ssis, C}
\begin{split}
&\ \Ssis^m (\BCx)/ \SS_m (\BCx) \\
= &  \bigoplus_{\omega\, \in \BC \times \BZ / {\scriptscriptstyle \sim } } \bigoplus_{j \in \BN} \ \varinjlim_{(\lambda, k)\, \in \omega} \big( [z]^{-k} |z|^{-\lambda} (\log |z|)^j \SS_{m + k} ( \BC) \big) \big/\SS_m (\BCx),
\end{split}
\end{equation}
\begin{equation}\label{2eq: refinement Msis, C}
\Msis / \Hrd
= \bigoplus_{\omega\, \in \BC \times \BZ / {\scriptscriptstyle \sim } } \bigoplus_{j \in \BN} \ \varinjlim_{(\lambda, k)\, \in \omega}  \Nsis^{\lambda - |m + k|, j} \big/ \Hrd.
\end{equation}
\end{lem}

\section {Hankel transforms and their Bessel kernels} \label{sec: Hankel transforms}

This section is arranged as follows. We start with the type of Hankel transforms over $\BR_+$ whose kernels are the (fundamental) Bessel functions studied in \cite{Qi}. After this, we introduce two auxiliary Hankel transforms and Bessel kernels over $\BR_+$. Finally, we proceed to construct and study Hankel transforms and their Bessel kernels over $\BFx$, with $\BF = \BR, \BC$. 

\begin{defn}\label{3defn: ordered set}
Let $ (\BX, \preccurlyeq) $ be an ordered set satisfying the condition that
\begin{equation}\label{3eq: condition on order}
\text{``}\lambda \preccurlyeq \lambda' \text{ or } \lambda' \preccurlyeq \lambda \text{'' is an equivalence relation.}
\end{equation} 
We denote the above equivalence relation  by $ \lambda \sim \lambda'$.  Given $ \ulambda = (\lambda_1, ..., \lambda_n) \in \BX^{n }$, the set $\{1, ... n\}$ is partitioned into several pair-wise disjoint subsets $L_\alpha$, $\alpha = 1, ..., A$, such that 
$$ \lambda_{l     } \sim  \lambda_{l     '}  \text{ if and only if } l     , l     ' \text{ are in the same } L_\alpha.$$
Each $\Lambda ^{\alpha} = \left \{ \lambda_{l     } \right \}_{l      \in L_\alpha}$\footnote{Here, $\left \{ \lambda_{l     } \right \}_{l      \in L_\alpha}$ is considered as a set, namely, $\lambda_{l     }$ are counted without multiplicity.} 
is a totally ordered set. Let $B_{\alpha} = \left| \Lambda ^{\alpha} \right|$ and  label the elements of $\Lambda ^{\alpha}$  in the descending order, $ \lambda_{\alpha, 1} \succ ... \succ \lambda_{\alpha, B_\alpha}.$
For $ \lambda_{\alpha, \beta} \in \Lambda ^{\alpha}$, let $M_{\alpha, \beta}$ denote the multiplicity of $\lambda_{\alpha, \beta}$ in $\ulambda$, that is, $M_{\alpha, \beta} = \left|\left\{ l      : \lambda_l      = \lambda_{\alpha, \beta} \right \} \right|$, and define $N_{\alpha, \beta} = \sum_{\gamma =1}^{\beta} M_{\alpha, \gamma} = \left|\left\{ l      : \lambda_{\alpha, \beta} \preccurlyeq \lambda_l      \right \} \right|$.

$\ulambda$ is called generic if $\lambda_{l     } \nsim  \lambda_{l     '}$ for any $l      \neq l     '$.
\end{defn}

We recall that the ordered sets $(\BC, \preccurlyeq_1)$,  $(\BC, \preccurlyeq_2)$,  $(\BC \times \BZT, \preccurlyeq )$ and $(\BC \times \BZ, \preccurlyeq )$
defined in \S \ref{sec: all Ssis} all satisfy \eqref{3eq: condition on order}.

\begin{figure}
	\begin{center}
		\begin{tikzpicture}

		\draw [->] (-3.5,0) -- (1.5,0);
		\draw [->] (0, -3) -- (0, 3);
		\draw [-, thick] (1, -0.8) -- (1, 0.8);
		\draw [-, thick] (-0.5 , -3) -- (-0.5, -1.75);
		\draw [-, thick] (-0.5 , 1.75) -- (-0.5, 3);
		\draw [->] (-0.5 , 2) -- (-0.5, 2.5);
		\draw [->] (-0.5 , -2.6) -- (-0.5, -2.5);
		\draw [dotted] (-0.5 , -2) -- (-0.5, 2);
		\draw [-, thick] (1, -0.8) to [out=-90, in=90] (-0.5, -1.75);
		\draw [-, thick] (1, 0.8) to [out=90, in=-90] (-0.5, 1.75);
		
		\draw[fill=white] (-0.5, 0) circle [radius=0.04];
		\node [ below ] at (-0.65, 0.05) {\footnotesize $\sigma$};
		
		\draw[fill=white] (0.5, 0.3 ) circle [radius=0.04];
		\node [above ] at (0.5, 0.3 )  {\footnotesize $\lambda_{l     } - \kappa_{l     } $};
		\draw[fill=white] (-0.7, 0.3 ) circle [radius=0.04];
		\draw[fill=white] (-1.9, 0.3 ) circle [radius=0.04];
		\draw[fill=white] (-3.1, 0.3 ) circle [radius=0.04];
		
		\node [left ] at (-0.5, 2.2 )  {\footnotesize $\EC^d_{(\ulambda, \ukappa)}$};

		\draw [->] (7+-3.5,0) -- (7+1.5,0);
		\draw [->] (7+0, -1.7) -- (7+0, 1.7);
		\draw [-, thick] (7+0.8, -0.81) -- (7+0.8, 0.81);
		\draw [-, thick] (7+-3.5, 0.8) -- (7+0.8, 0.8);
		\draw [-, thick] (7+-3.5, -0.8) -- (7+0.8, -0.8);
		\draw [->] (7+-1, 0.8) -- (7+-2, 0.8);
		\draw [->] (7+-3.5, -0.8) -- (7+-2, -0.8);
		
		\draw[fill=white] (7+0.5, 0.3 ) circle [radius=0.04];
		\node [above ] at (7+0.5, 0.3 )  {\footnotesize $\lambda_{l     }$};
		\draw[fill=white] (7+-0.7, 0.3 ) circle [radius=0.04];
		\draw[fill=white] (7+-1.9, 0.3 ) circle [radius=0.04];
		\draw[fill=white] (7+-3.1, 0.3 ) circle [radius=0.04];
		
		\node [right ] at (7+0.8, 0.8)  {\footnotesize $\EC'_{\ulambda}$};
		
		\end{tikzpicture}
	\end{center}
	\caption{$\EC^d_{(\ulambda, \ukappa)}$ and $\EC'_{\ulambda}$}\label{fig: C d lambda}
\end{figure}

\begin{defn}\label{3defn: C d lambda}
Let $d = 1$ or $2$, $\ulambda \in \BC^{n}$ and $\ukappa \in \BN^n$. Put $\sigma < \frac d 2 + \frac 1 {n} (\Re |\ulambda| - 1 )$ and choose
a contour $ \EC^d _{(\ulambda, \ukappa)}$ {\rm(}see Figure {\rm \ref{fig: C d lambda}}{\rm)} such that
\begin{itemize}
\item[-] $\EC^d _{(\ulambda, \ukappa)}$ is  upward directed  from $\sigma - i \infty$ to $\sigma + i \infty$,
\item[-] all the sets $\lambda_{l     } - \kappa_{l     } - \BN$ lie on the left  side of $\EC _{(\ulambda, \ukappa)}$, and
\item[-] if $s \in \EC^d _{(\ulambda, \ukappa)}$ and $|\Im s| $ is sufficiently large, say $|\Im s| - \max  \left\{ |\Im \lambda_{l     }| \right \} \ggg 1$, then $\Re s = \sigma$.
\end{itemize}
For $\ulambda \in \BC$, we denote $\EC_{\ulambda} = \EC^1 _{(\ulambda, \boldsymbol 0)}$. For $(\umu, \udelta) \in \BC^n \times (\BZT)^n$, we denote $\EC_{(\umu, \udelta)} = \EC^1 _{(\umu, \udelta)}$. For $(\umu, \um) \in \BC^n \times \BZ^n$, we denote $\EC_{(\umu, \um)} = \frac 1 2\cdot \EC^2 _{(2 \umu, \|\um\|)}$.
\end{defn}

\begin{defn}\label{3defn: C ' lambda}
For  $\ulambda \in \BC^{n}$, choose a contour $\EC'_{\ulambda}$ illustrated in Figure {\rm \ref{fig: C d lambda}} such that 
\begin{itemize}
\item[-] $\EC'_{\ulambda}$ starts from and returns to $- \infty$ counter-clockwise,
\item[-] $\EC'_{\ulambda}$ consists two horizontal infinite half lines, 
\item[-] $\EC'_{\ulambda}$ encircles all the sets $\lambda_{l     } - \BN$, and
\item[-] $ \Im s \lll \max  \{ |\Im \lambda_l     | \} + 1$ for all $s \in \EC'_{\ulambda}$.
\end{itemize}
\end{defn}

\subsection {\texorpdfstring{The Hankel transform $ \Hsl$ and the Bessel function $J(x; \usigma, \ulambda)$}{The Hankel transform $ H_{(\varsigma, \lambda)}$ and the Bessel function $J(x; \varsigma, \lambda)$}}

\subsubsection{\texorpdfstring{The definition of  $ \Hsl$}{The definition of  $ H_{(\varsigma, \lambda)}$}}

Consider the ordered set $(\BC, \preccurlyeq_1)$. For  $\ulambda \in \BC^{n}$, let notations $\lambda_{\alpha, \beta}$, $B_\alpha$, $M_{\alpha, \beta}$ and $N_{\alpha, \beta}$ be as in Definition \ref{3defn: ordered set}.
We define the following subspace of $ \Ssis (\BR _+)$,
\begin{equation*}
\Ssis^{ \ulambda } (\BR _+) = \sum_{ \alpha = 1}^A \sum_{\beta = 1}^{B_\alpha} \sum_{j=0}^{N_{\alpha, \beta} - 1} x^{- \lambda_{\alpha, \beta}} (\log x)^{j } \SS (\overline \BR _+).
\end{equation*}
\begin{prop} Let $( \usigma, \ulambda) \in  \{ +, - \}^n \times \BC^{n}$. Suppose $\upsilon \in \SS (\BR _+)$. Then there exists a unique function $\Upsilon \in \Ssis^{ \ulambda } (\BR _+) $ satisfying the following identity,
\begin{equation}\label{3eq: Hankel transform identity 0, R+}
\EM \Upsilon (s ) = G (s; \usigma, \ulambda) \EM \upsilon ( 1 - s).
\end{equation}
We call $\Upsilon$ the Hankel transform of  $\upsilon$ over $\BR _+$ of index $( \usigma, \ulambda)$ and write  $\Hsl \upsilon  = \Upsilon$.
\end{prop}
\begin{proof}
Recall the definition  of $G (s; \usigma, \ulambda)$ given by (\ref{1def: G pm (s)}, \ref{1def: G(s; sigma; lambda)}),
\begin{equation*}
G (s; \usigma, \ulambda) = e \lp  \frac {\sum_{l      = 1}^n \varsigma_l      (s - \lambda_l     )} 4 \rp
\prod_{l      = 1}^n \Gamma \lp s - \lambda_{l     } \rp.
\end{equation*}
The product in the above expression   may be rewritten as below
\begin{equation*}
\prod_{ \alpha = 1}^A \prod_{\beta = 1}^{B_\alpha}  \Gamma \lp s - \lambda _{\alpha, \beta} \rp^{M_{\alpha, \beta} }.
\end{equation*}
Thus the singularities of $G (s; \usigma, \ulambda)$ are poles at the points in $ \lambda _{\alpha, 1} - \BN$, $\alpha =1,..., A$. More precisely, $G (s; \usigma, \ulambda)$ has a pole of pure order $N_{\alpha, \beta}$ at $\lambda \in \lambda_{\alpha, 1} - \BN$ if one let  $\beta = \max \left\{ \beta' : \lambda \preccurlyeq_1 \lambda_{\alpha, \beta' } \right \} $. Moreover, in view of \eqref{1eq: vertical bound, G (s; sigma, lambda)} in Lemma \ref{1lem: vertical bound}, $G (s; \usigma, \ulambda)$ is of uniform moderate growth on vertical strips. 

On the other hand, according to Corollary \ref{cor: Mellin, Schwartz} (1), $\EM \upsilon (1-s)$ uniformly rapidly decays on vertical strips. 

Therefore, the product $G (s; \usigma, \ulambda) \EM \upsilon ( 1 - s)$ on the right hand side of \eqref{3eq: Hankel transform identity 0, R+} is a meromorphic function in the space $\sum_{ \alpha = 1}^A \sum_{\beta = 1}^{B_\alpha} \sum_{j=0}^{N_{\alpha, \beta} - 1}   \Msis^{ \lambda_{\alpha, \beta} , j}$. We conclude from Lemma \ref{2lem: Sis and Msis, R+} that  \eqref{3eq: Hankel transform identity 0, R+} uniquely determines a function  $\Upsilon$ in $ \Ssis^{ \ulambda } (\BR _+)$.
\end{proof}

\subsubsection{The Bessel function $J(x; \usigma, \ulambda)$}\label{sec: Bessel kernel J(x; sigma, lambda)}
\
\vskip 5 pt 
{\it The integral kernel $J(x; \usigma, \ulambda)$ of $\Hsl$.}
Suppose $\upsilon \in \SS (\BR _+)$. By the Mellin inversion, we have
\begin{equation}\label{3eq: Bessel kernel, Mellin inversion, 1}
\Upsilon (x) 
= \frac { 1 } {2 \pi i} \int_{(\sigma)} G (s; \usigma, \ulambda) \EM  \upsilon (1 - s) x ^{ - s} d s, \hskip 10 pt \textstyle \sigma > \max  \left\{ \Re  \lambda_l      \right \}.
\end{equation}
It is an iterated double integral as below
\begin{equation*}
\Upsilon (x) 
= \frac { 1 } {2 \pi i} \int_{(\sigma)} \int_0^\infty \upsilon (y) y^{ - s } d y \cdot G (s; \usigma, \ulambda) x ^{ - s} d s.
\end{equation*}
We now shift the integral contour to $\EC _{\ulambda}$ defined in Definition \ref{3defn: C d lambda}. Using \eqref{1eq: vertical bound, G (s; sigma, lambda)} in Lemma \ref{1lem: vertical bound}, one shows that the above double integral becomes absolutely convergent after this contour shift.
Therefore, on changing the order of integrals, one obtains
\begin{equation}\label{2eq: Psi (x; sigma) as Hankel transform}
\Upsilon (x ) = \int_0^\infty \upsilon (y) J \big( (xy)^{\frac 1 n}; \usigma, \ulambda\big) d y.
\end{equation}
Here $J (x;  \usigma, \ulambda)$ is the {\rm({\it fundamental})} \textit{Bessel function} defined by the  Mellin-Barnes type integral
\begin{equation}\label{3eq: definition of J (x; sigma)}
J (x ; \usigma, \ulambda) = \frac 1  {2 \pi i}\int_{\EC_{\ulambda}} G(s; \usigma, \ulambda) x^{- n s} d s,
\end{equation}
which is categorized as a {\it Bessel function of the second kind}  (see \cite{Qi}).


\begin{rem}
The expression \eqref{2eq: Psi (x; sigma) as Hankel transform} of the Hankel transform together with properties of the   Bessel function $J (x ; \usigma, \ulambda)$ may also yield $\Upsilon \in \Ssis^{ \ulambda } (\BR _+)$. 

The Schwartz condition on $\Upsilon$ at infinity follows from either the rapid decay or the oscillation of $J (x ; \usigma, \ulambda)$ as well as its derivatives {\rm ({\it see \cite[\S 5, \S 9]{Qi}})}.  

As for the singularity type of $\Upsilon$ at zero, we first assume that $\ulambda$ is generic. We express  $J (x ; \usigma, \ulambda)$ as a combination of Bessel functions of the first kind {\rm ({\it see \cite[\S 7.1, 7.2]{Qi}})}. Then the type of singularities of $\Upsilon$ at zero is reflected by the leading term in the series expansions of Bessel functions of the first kind. For nongeneric $\ulambda$ the occurrence of powers of $\log x $ follows from either solving the Bessel equations using the Frobenius method or taking the limit of the above expression of  $J (x ; \usigma, \ulambda)$ with respect to the index $\ulambda$.
\end{rem}

{\it Shifting the index of $J (x ; \usigma, \ulambda)$.}
\begin{lem}\label{3lem: normalize J(x; sigma, lambda)}
Let $(\usigma, \ulambda) \in \{+, -\}^n \times \BC^{n}$ and $\lambda \in \BC$. Recall that  $\ue^n$ denotes the $n$-tuple $ ( {1,..., 1})$. Then
\begin{equation}\label{4eq: normalize J}
J (x ; \usigma, \ulambda - \lambda \ue^n) = x^{n \lambda} J (x ; \usigma, \ulambda).
\end{equation}
\end{lem}

{\it Regularity of $J(x; \usigma, \ulambda)$.}
According to \cite[\S 6, 7]{Qi}, $J (x ; \usigma, \ulambda)$ satisfies a differential equation with analytic coefficients. Therefore, $J (x ; \usigma, \ulambda)$ admits an analytic continuation from $\BR _+$ onto $\BU$, and in particular is {\it real analytic}. Here, we shall take an alternative viewpoint from \cite[Remark 7.11]{Qi},
that is the following Barnes type integral representation,
\begin{equation} \label{3eq: Barnes contour integral 1}
J (\zeta ; \usigma, \ulambda) = \frac 1  {2 \pi i}\int_{\EC_{\ulambda}'} G(s; \usigma, \ulambda) \zeta^{-n s} d s, \hskip 10 pt \zeta = x e^{i \omega} \in \BU, x \in \BR _+, \omega \in \BR,
\end{equation}
with the integral contour  given in Definition \ref{3defn: C ' lambda}.
One first rewrites $G  (s, \pm)$ using Euler's reflection formula,
\begin{equation*} 
G  (s, \pm) = \frac {\pi e \lp \pm \frac 1 4 s   \rp} {\sin (\pi s) \Gamma (1-s)},
\end{equation*}
Then Stirling's formula \eqref{1eq: Stirling's formula} yields,
\begin{equation*} 
G(- \rho + it; \usigma, \ulambda) \lll_{\ulambda, r} {e^{n \rho}} {\rho^{ - n \lp \rho + \frac 1 2 \rp - \Re |\ulambda|}},
\end{equation*}
for all $ - \rho + it \notin \bigcup_{l      = 1}^n \bigcup_{\kappa \in \BN} \BB_r (\lambda_{l     } - \kappa)$ satisfying $\rho \ggg 1$ and $ t \lll \max  \{ |\Im \lambda_{l     }| \} + 1$.
It follows that the contour integral in \eqref{3eq: Barnes contour integral 1} converges absolutely and compactly
in $\zeta$, and hence $J (\zeta ; \usigma, \ulambda)$ is analytic in $\zeta$. 

Moreover, given any bounded open subset of $\BC^{n}$, one fixes a single contour $\EC' = \EC'_{\ulambda}$ for all $\ulambda$ in this set and verifies the uniform convergence of the integral in the $\ulambda$ aspect. Then follows the analyticity of $J (\zeta ; \usigma, \ulambda)$  with respect to $\ulambda$.

\begin{lem}\label{3lem: J (x; sigma, lambda) analytic}
$J (x ; \usigma, \ulambda) $  admits an analytic continuation $J (\zeta ; \usigma, \ulambda)$ from $\BR _+$ onto $\BU$. 
In particular, $J (x ; \usigma, \ulambda) $ is a real analytic function of $x$ on $\BR _+$. Moreover,  $J (\zeta ; \usigma, \ulambda)$  is an analytic function of $\ulambda$ on $\BC^{n}$.
\end{lem}


{\it The rank-one and rank-two cases.}

\begin{example}\label{ex: fundamental Bessel, n=1, 2}
According to \cite[Proposition 2.4]{Qi}, if $n = 1$, then
\begin{equation*}
J (x; \pm, 0) = e^{\pm i x}.
\end{equation*}
For $n = 2$, from \cite[Proposition 2.7]{Qi} we have
\begin{align*}
J (x; \pm, \pm, \lambda, - \lambda)   = \pm \pi i e^{\pm \pi i \lambda} H^{(1, 2)}_{2 \lambda} (2 x), \hskip 10 pt 
J (x; \pm, \mp, \lambda, - \lambda)   = 2 e^{\mp \pi i \lambda} K_{2 \lambda} (2 x),
\end{align*}
where, for $\nu \in \BC$, $H^{(1)}_{\nu}$, $H^{(2)}_{\nu}$ are the Hankel functions, and $K_{\nu}$ is the $K$-Bessel function {\rm ({\it the modified Bessel function of the second kind})}.
\end{example}

\subsection {\texorpdfstring{The Hankel transforms $\hld$, $\hmum$ and the Bessel kernels $j_{(\umu, \udelta)} $, $j_{(\umu, \um)}$}{The Hankel transforms $h_{(\mu, \delta)} $, $h_{(\mu, m)}$ and the Bessel kernels $j_{(\mu, \delta)} $, $j_{(\mu, m)}$}}

Consider the ordered set $(\BC, \preccurlyeq_2)$ and define $ \lambda _{\alpha, \beta}$, $B_\alpha$, $M_{\alpha, \beta}$ and $N_{\alpha, \beta}$  as in Definition \ref{3defn: ordered set} corresponding to $\ulambda \in \BC^n$. We define the following subspace of $ \Ssis (\BR _+)$
\begin{equation}\label{3eq: Ssis2, R+}
\Ssiss^{ \ulambda } (\BR _+) = \sum_{ \alpha = 1}^A \sum_{\beta = 1}^{B_\alpha}  \sum_{j=0}^{N_{\alpha, \beta} - 1} x^{ - \lambda_{\alpha, \beta}} (\log x)^{j } \SS_{0} (\overline \BR _+).
\end{equation}

\subsubsection{\texorpdfstring{The definition of  $  \hld$}{The definition of  $h_{(\mu, \delta)} $}}\label{sec: h (lambda, delta)}


The following proposition provides the definition of the Hankel transform $\hld$, which maps $\Ssiss^{ - \umu - \udelta } (\BR _+)$    onto  $\Ssiss^{ \umu - \udelta } (\BR _+)$ bijectively. 

\begin{prop} \label{3prop: h (lambda, delta)}
Let $(\umu, \udelta) \in \BC^{n} \times (\BZ/2 \BZ)^n$. Suppose $\upsilon \in \Ssiss^{ - \umu - \udelta } (\BR _+)$. Then there exists a unique function $\Upsilon \in \Ssiss^{ \umu - \udelta } (\BR _+) $ satisfying the following identity,
\begin{equation}\label{3eq: Hankel transform identity 1, R+}
\EM \Upsilon (s ) = G_{(\umu, \udelta )} (s) \EM \upsilon ( 1 - s).
\end{equation}
We call $\Upsilon$ the Hankel transform of  $\upsilon$ over $\BR _+$ of index $(\umu, \udelta)$ and  write $\hld \upsilon  = \Upsilon$. Furthermore, we have the Hankel inversion formula
\begin{equation}\label{3eq: Hankel inversion, R 0}
\hld \upsilon  = \Upsilon, \hskip 10 pt \hmld \Upsilon  = \upsilon.
\end{equation}
\end{prop}

\begin{proof}
Recall the definition of $G_{(\umu, \udelta )}$ given by (\ref{1def: G delta}, \ref{1def: G (lambda, delta)}), 
\begin{equation*}
G_{(\umu, \udelta )} (s) = i^{|\udelta|} \pi^{n \lp \frac 1 2 - s \rp + |\umu|} 
  \frac { \prod_{l      = 1}^n \Gamma \lp \frac 1 2( {s - \mu_l      + \delta_l     } ) \rp  } 
{\prod_{l      = 1}^n \Gamma \lp \frac 1 2 ( {1 - s + \mu_l      + \delta_l     } ) \rp }, \end{equation*}
where $|\udelta| = \sum_{l     } \delta_l      \in \BN$, with each $\delta_l     $ viewed as a number in the set $\{0, 1\} \subset \BN$. 

We write $\umu^\pm = \pm \umu - \udelta$. Since $\mu^+_{  l     } + \mu^-_{ l     } = - 2 \delta_l      \in \{0, - 2\}$, the partition $\left\{ L_{\alpha} \right \}_{\alpha = 1}^A$ of $\{1, ..., n\}$ and $B_\alpha$ in Definition \ref{3defn: ordered set} are the same for both $\umu^+$ and $\umu^-$. Let  $ \mu^\pm _{\alpha, \beta}$, $M^\pm _{\alpha, \beta}$ and $N^\pm _{\alpha, \beta}$ be the notations  in Definition \ref{3defn: ordered set} corresponding to $\umu^\pm$.  Then the Gamma quotient above may be rewritten as follows,
\begin{equation*}
\frac { \prod_{ \alpha = 1}^A \prod_{\beta = 1}^{B_\alpha}  \Gamma \left( \frac 1 2 \big( { s - \mu^+_{\alpha, \beta}} \big) \right)^{M^+_{\alpha, \beta} } } 
{ \prod_{ \alpha = 1}^A \prod_{\beta = 1}^{B_\alpha}  \Gamma \left( \frac 1 2 \big( {1 - s - \mu^-_{\alpha, \beta}} \big) \right)^{M^-_{\alpha, \beta} } }.
\end{equation*}
Thus, at each point $\mu \in \mu^+_{\alpha, 1} - 2\BN$ the product in the numerator contributes to $G_{(\umu, \udelta )} (s)$ a pole of pure order  $N_{\alpha, \beta}^+$, with $\beta = \max \left\{ \beta' : \mu \preccurlyeq_2 \mu^+_{\alpha, \beta' } \right \} $, whereas at each point $ \mu \in - \mu^-_{\alpha, 1} + 2\BN + 1$ the denominator contributes a zero of order $N_{\alpha, \beta}^-$, with $\beta = \max \left\{ \beta' : 1 - \mu \preccurlyeq_2 \mu^-_{\alpha, \beta' } \right \} $. 
Moreover, \eqref{1eq: vertical bound, G (lambda, delta) (s)} in Lemma \ref{1lem: vertical bound} implies that $G_{(\umu, \udelta )} (s)$ is of uniform moderate growth on vertical strips. 

On the other hand, according  to Lemma \ref{2lem: refined decomp, R+},  the Mellin transform $\EM \upsilon $ lies in the space $ \sum_{ \alpha = 1}^A \sum_{\beta = 1}^{B_\alpha}  \sum_{j=0}^{N^-_{\alpha, \beta} - 1} \Nsis^{  \mu^-_{\alpha, \beta}, j}.$
In particular,  the poles of $\EM \upsilon (1-s)$ are annihilated by the zeros contributed from the denominator of the Gamma quotient. Furthermore, $\EM \upsilon (1-s)$ uniformly rapidly decays on vertical strips. 

We conclude that the product $ G_{(\umu, \udelta )} (s) \EM \upsilon ( 1 - s)$ on the right hand side of \eqref{3eq: Hankel transform identity 1, R+} lies in the space $\sum_{ \alpha = 1}^A \sum_{\beta = 1}^{B_\alpha}  \sum_{j=0}^{N^+_{\alpha, \beta} - 1} \Nsis^{  \mu^+_{\alpha, \beta}, j} $, and hence  $\Upsilon \in \Ssiss^{ \umu - \udelta } (\BR _+) $, with another application of  Lemma \ref{2lem: refined decomp, R+}.

Finally, the Hankel inversion formula \eqref{3eq: Hankel inversion, R 0} is an immediate consequence of the functional relation \eqref{1eq: G mu delta (1-s) = G - mu delta (s)} of gamma factors.
\end{proof}


\subsubsection{\texorpdfstring{The definition of  $ \hmum$}{The definition of  $h_{(\mu, m)}$}}\label{sec: h mu m}

The following proposition provides the definition of the Hankel transform $\hmum$, which maps $\Ssiss^{ - 2 \umu - \| \um \| } (\BR _+)$    onto  $\Ssiss^{ 2 \umu - \| \um \| } (\BR _+)$ bijectively. 

\begin{prop} \label{3prop: h (mu, m)}
Let $(\umu, \um) \in \BC^{n} \times \BZ ^n$. Suppose $\upsilon \in \Ssiss^{ - 2 \umu - \|\um\| } (\BR _+)$. Then there exists a unique function $\Upsilon \in \Ssiss^{ 2 \umu - \|\um\| } (\BR _+) $ satisfying the following identity,
\begin{equation}\label{3eq: Hankel transform identity 2, R+}
\EM \Upsilon (2 s ) = G_{(\umu, \um )} (s) \EM \upsilon ( 2(1-s)).
\end{equation}
We call $\Upsilon$ the Hankel transform of  $\upsilon$ over $\BR _+$ of index $(\umu, \um)$ and write $\hmum \upsilon  = \Upsilon$.
Moreover, we have the Hankel inversion formula
\begin{equation}\label{3eq: Hankel inversion, C 0}
\hmum \upsilon  = \Upsilon, \hskip 10 pt \hmmum \Upsilon  = \upsilon.
\end{equation}
\end{prop}

\begin{proof}
We first rewrite \eqref{3eq: Hankel transform identity 2, R+} as follows,
\begin{equation*} 
\EM \Upsilon ( s ) = G_{(\umu, \um )} \lp \frac s 2 \rp \EM \upsilon ( 2 - s ).
\end{equation*}
From (\ref{1def: G m (s)}, \ref{1def: G (mu, m)}), we have
\begin{equation*}
G_{(\umu, \um)} \lp \frac s 2 \rp = i^{\left|\|\um\|\right|} \pi^{n \lp 1 -  s \rp + 2 |\umu|} 
  \frac { \prod_{l      = 1}^n \Gamma \lp \frac 1 2 ( {s - 2\mu_l      + |m_l     |} ) \rp  } 
{\prod_{l      = 1}^n \Gamma \lp \frac 1 2 ( {2 - s + 2\mu_l      + |m_l     |} ) \rp },
\end{equation*}
where $\left|\|\um\|\right| = \sum_{l     =1}^n |m_l     | $ according to our notations.
We can now proceed to apply the same arguments in the proof of Proposition \ref{3prop: h (lambda, delta)}. Here, one uses \eqref{1eq: vertical bound, G (mu, m) (s)} and \eqref{1eq: G mu m (1-s) = G - mu m (s)} instead of \eqref{1eq: vertical bound, G (lambda, delta) (s)} and \eqref{1eq: G mu delta (1-s) = G - mu delta (s)} respectively.
\end{proof}

\subsubsection{\texorpdfstring{The Bessel kernel  $j_{(\umu, \udelta)} $}{The Bessel kernel  $j_{(\mu, \delta)} $}} \label{sec: Bessel kernel j lambda delta}
\ 
\vskip 5 pt

{\it The definition of  $j_{(\umu, \udelta)} $.}
For $(\umu, \udelta) \in \BC^{n} \times (\BZ/2 \BZ)^n$, we define the Bessel kernel $j_{(\umu, \udelta)}$,
\begin{equation}\label{2def: Bessel kernel, 1}
j_{(\umu, \udelta)} (x) = \frac 1  {2 \pi i} \int_{\EC_{(\umu, \udelta)}} G_{(\umu, \udelta)} (s ) x^{- s} d s.
\end{equation}
We call the integral in \eqref{2def: Bessel kernel, 1} a  Mellin-Barnes type integral. It is clear that
\begin{equation}\label{3eq: normalize j (lambda, delta)}
j_{(\umu - \mu \ue^n, \udelta)} (x) = x^{\mu} j_{(\umu, \udelta)} (x).
\end{equation}
In view of \eqref{1eq: G (lambda, delta) = G (s; sigma, lambda)}, we have
\begin{equation}\label{3eq: j (lambda, delta) and fundamental}
j_{(\umu, \udelta)} (x) = (2\pi)^{|\umu|}  \sum_{\usigma \in \{ +, - \}^n} \usigma^{\udelta} J \big(2 \pi x^{\frac 1 n}; \usigma, \umu \big).
\end{equation}


{\it Regularity of $j_{(\umu, \udelta)} $.}
It follows from \eqref{3eq: j (lambda, delta) and fundamental} and Lemma \ref{3lem: J (x; sigma, lambda) analytic} that
$j_{(\umu, \udelta)} (x)$ admits an analytic continuation $j_{(\umu, \udelta)} (\zeta)$, which is also analytic with respect to $\umu$.
Moreover,  $j_{(\umu, \udelta)} (\zeta)$ has the following Barnes type integral representation,
\begin{equation}\label{2def: Bessel kernel, analytic continuation, 1}
j_{(\umu, \udelta)} (\zeta) = \frac 1  {2 \pi i} \int_{\EC'_{\umu - \udelta}} G_{(\umu, \udelta)} (s ) \zeta^{- s} d s, \hskip  10 pt \zeta \in \BU.
\end{equation}
To see the convergence, the following formula is required
\begin{equation} 
G_\delta (s) = 
\left\{ \begin{split}
& \frac {\pi (2 \pi)^{-s}}  {\sin \left(\frac 1 2 {\pi s}   \right) \Gamma (1-s)}, \hskip 10 pt \text { if } \delta = 0,\\
& \frac {\pi i (2 \pi)^{-s}}  {\cos \left(\frac 1 2 {\pi s}   \right) \Gamma (1-s)}, \hskip 10 pt \text { if } \delta = 1.
\end{split} \right.
\end{equation}

{\it The integral kernel of $\hld$.}
Suppose $\upsilon \in \Ssiss^{ - \umu - \udelta } (\BR _+)$. In order to proceed in the same way as in \S \ref{sec: Bessel kernel J(x; sigma, lambda)},
one needs to assume that $(\umu, \udelta) $ satisfies the condition 
\begin{equation}\label{3eq: condition on (lambda, delta), 0}
\textstyle \min  \left\{ \Re \mu_l      + \delta_{l     } \right\} + 1 > \max  \left\{ \Re \mu_l      - \delta_{l     } \right\}.
\end{equation}
Then,
\begin{equation}\label{3eq: integral kernel, R 0}
\hld \upsilon  (x ) =  \int_0^\infty \upsilon (y) j_{(\umu, \udelta)} ( xy ) d y.
\end{equation}
Here, it is required for the convergence of the integral over $d y$ that the contour $\EC _{(\umu, \udelta)}$ in \eqref{2def: Bessel kernel, 1} is chosen   to lie in the left half-plane $\left\{s : \Re s < \min  \left\{ \Re \mu_l      + \delta_{l     } \right\} + 1 \right\}$. According to Definition \ref{3defn: C d lambda}, this choice of  $\EC _{(\umu, \udelta)}$ is permissible due to our assumption \eqref{3eq: condition on (lambda, delta), 0}.

If one assumes $\upsilon \in \SS (\BR _+)$, then \eqref{3eq: integral kernel, R 0} remains valid without requiring the condition \eqref{3eq: condition on (lambda, delta), 0}. 

\vskip 5 pt

{\it The rank-one and rank-two examples.}

\begin{example}\label{ex: Bessel kernel, real, 0, n=1, 2}
If $n = 1$, we have
\begin{equation}\label{3eq: j (0, delta)}
j_{(0, 0)} (x ) = 2 \cos (2\pi x), \hskip 10 pt j_{(0, 1)} (x ) = 2 i \sin (2\pi x).
\end{equation}
If $n = 2$, we are particularly interested in the following Bessel kernel,
\begin{equation}\label{3eq: j lambda delta, discrete series}
j_{\lp \frac 1 2 m , - \frac 1 2 m ,\, \delta (m) + 1, 0 \rp} (x) = j_{\lp \frac 1 2 m , - \frac 1 2 m ,\, \delta (m), 1 \rp} (x) =  2 \pi i^{m + 1} J_{m} (4 \pi \sqrt x),
\end{equation}
 with $m \in \BN $. For $\nu \in \BC$, $J_{\nu}$ is the $J$-Bessel function {\rm ({\it the Bessel function of the first kind})}.
\end{example}

\subsubsection{\texorpdfstring{The Bessel kernel $j_{(\umu, \um)} $}{The Bessel kernel $j_{(\mu, m)} $}} \label{sec: Bessel kernel j mu m}
\
\vskip 5 pt
{\it The definition of $j_{(\umu, \um)} $.}
For $(\umu, \um) \in \BC^{n} \times  \BZ ^n$ define the Bessel kernel $j_{(\umu, \um)}$ by
\begin{equation}\label{3def: Bessel kernel j mu m}
j_{(\umu, \um)} (x) = \frac 1  {2 \pi i} \int_{\EC _{ (\umu, \um)}} G_{(\umu, \um)} (s ) x^{- 2 s} d s.
\end{equation}
The integral in \eqref{3def: Bessel kernel j mu m} is called a Mellin-Barnes type integral. We have
\begin{equation}\label{3eq: normalize j (mu, m)}
j_{(\umu - \mu \ue^n, \um)} (x) = x^{2 \mu} j_{(\umu, \um)} (x).
\end{equation}
In view of Lemma \ref{1lem: complex and real gamma factors}, if $(\boldsymbol \eta, \udelta) \in \BC^{2n} \times (\BZT)^{2n}$ is related to $(\umu, \um) \in \BC^{n} \times  \BZ ^n$ via either \eqref{1eq: relation between (mu, m) and (lambda, delta), 1} or  \eqref{1eq: relation between (mu, m) and (lambda, delta), 2}, then 
\begin{equation}\label{3eq: j mu m = j lambda delta}
i^n j_{(\umu, \um)} (x) = j_{(\boldsymbol \eta, \udelta)} \big( x^2 \big).
\end{equation}

{\it Regularity of $j_{(\umu, \um)} $.}
In view of  \eqref{3eq: j mu m = j lambda delta}, the regularity of
$j_{(\umu, \um)} $ follows from that of $j_{(\boldsymbol \eta, \udelta)}$. 
Alternatively, this may be seen from
\begin{equation}\label{3eq: Bessel kernel, analytic continuation, 2}
j_{(\umu, \um)} (\zeta) = \frac 1  {2 \pi i} \int_{\EC'_{ \umu - \frac 1 2 \|\um\|}} G_{(\umu, \um)} (s ) \zeta^{- 2 s} d s, \hskip  10 pt \zeta \in \BU.
\end{equation}
To see the convergence,  the following formula is required
\begin{equation} \label{3eq: rewrite G m}
G_m (s) = 
\frac {\pi i^{|m|} (2\pi )^{1-2s} }  { \sin \lp \pi \lp s + \frac 1 2 {|m|}   \rp \rp \Gamma \lp 1-s-\frac 1 2 {|m|}   \rp  \Gamma \lp 1-s+\frac 1 2 {|m|}   \rp }.
\end{equation}

{\it The integral kernel of $\hmum$.}
Suppose $\upsilon \in \Ssiss^{ - 2 \umu - \|\um\| } (\BR _+)$. We assume that $(\umu, \um) $ satisfies the following condition
\begin{equation}\label{3eq: condition on (mu, m), 0}
\textstyle \min  \left\{ \Re \mu_l      + \frac 1 2 { |m_{l     }| }   \right\} + 1 > \max  \left\{ \Re \mu_l      - \frac 1 2 {|m_{l     }|}   \right\}.
\end{equation}
Then
\begin{equation}\label{3eq: h mu m = int of j mu m}
\hmum \upsilon (x)  =  \int_0^\infty \upsilon (y) j_{(\umu, \um)} ( xy ) \cdot 2 y d y,
\end{equation}
It is required for convergence that the integral contour $\EC _{ (\umu, \um)}$ in \eqref{3def: Bessel kernel j mu m} lies in the left half-plane $\left\{s : \Re s < \min  \left\{\Re \mu_l      + \frac 1 2 { |m_{l     }| }   \right\} + 1 \right\}$. This is guaranteed by \eqref{3eq: condition on (mu, m), 0}.

Moreover, if one assumes  $\upsilon \in \SS (\BR _+)$, then \eqref{3eq: h mu m = int of j mu m} holds true for any index $(\umu, \um)$.

\vskip 5 pt

{\it The rank-one case.}

\begin{example}\label{ex: Bessel kernel j 0 m, complex, n=1}
If $n = 1$, in view of \eqref{3eq: j lambda delta, discrete series} and \eqref{3eq: j mu m = j lambda delta}, we have for $m \in \BZ$
\begin{equation}\label{3eq: j (0, m)}
j_{(0, m)} (x) = 2 \pi i^{|m| } J_{|m|} (4 \pi x) = 2 \pi i^{m } J_{m} (4 \pi x).
\end{equation}
where the second equality follows from the identity $J_{-m} (x) = (-)^m J_m (x)$.
\end{example}


{\it Auxiliary bounds for $j_{(\umu, \um + m\ue^n)}$.}
\begin{lem}\label{3lem: bound of the Bessel kernel, C}
Let $(\umu, \um) \in \BC^{n} \times  \BZ ^n$ and $m \in \BZ$. 
Put 
\begin{align*}
& A =\textstyle  n \lp \max  \{ \Re \mu_l      \} + \frac 1 2 \max  \{ |m_l     | \} - \frac 1 2 \rp - \Re |\umu| + \frac 1 2 |\|\um\||, \\
& B_+ =\textstyle  - 2 \min  \{ \Re \mu_l      \} + \max  \{ |m_l     | \} + \max \left\{ \frac 1 n - \frac 1 2, 0 \right \}, \\
& B_- =\textstyle  - 2 \max  \{ \Re \mu_l      \} - \max  \{ |m_l     | \}.
\end{align*} 
Fix  $ \epsilon > 0$. Denote by $\ue^n$   the $n$-tuple $(1, ..., 1)$. We have the following estimate
\begin{equation}
\label{2eq: bound of the Bessel kernel, C}
\begin{split}
j_{(\umu, \um + m\ue^n)} (x) \lll_{(\umu, \um),\, \epsilon, \, n} &\ \lp \frac { 2 \pi e x^{\frac 1 { n}}} {|m| + 1} \rp^{n |m|}  (|m| + 1)^{A + n \epsilon} \max \left\{ x^{B_+ + 2\epsilon }, x^{ B_- - 2\epsilon}  \right\}.
\end{split}
\end{equation}
\end{lem}

\begin{proof} 
Let 
\begin{align*}
\rho_m &= \textstyle \max  \left\{ \Re \mu_l      - \frac 1 2 { |m_l      + m | }   \right  \}, \\
\sigma_m &= \min \left\{ \textstyle \frac 1 2 + \frac 1 n \lp \Re |\umu| - \frac 1 2 \left| \| \um + m\ue^n\| \right| - 1 \rp, \rho_m \right \}.
\end{align*}
Choose the contour $\EC_m = \EC _{(\umu, \um + m\ue^n)}$ (see Definition \ref{3defn: C d lambda}) such that
\begin{itemize}
\item[-] if $s \in \EC_m$ and $\Im s $ is sufficiently large, then $\Re s = \sigma_m - \epsilon$, and
\item[-] $\EC_m$ lies in the vertical strip $\BS[ \sigma_m - \epsilon, \rho_m +\epsilon]$.
\end{itemize} 

We first assume that $|m|$ is large enough so that
$$ \textstyle  n \lp \rho_m + \epsilon - \frac 1 2 \rp - \Re |\umu| - \frac 1 2 \left| \| \um + m\ue^n \| \right| < 0.$$
For the sake of brevity, we write $y = (2\pi)^{ n} x $. 
We first bound $\left| j_{(\umu, \um + m\ue^n)} (x) \right|$ by
\begin{equation*}
\begin{split}
 (2\pi)^{n + \Re |\umu|} \int_{\EC_m}  y^{- 2 \Re s} \prod_{l      = 1}^n \left| \frac { \Gamma \lp s - \mu_l      + \frac 1 2 { |m_l      + m| }   \rp} { \Gamma \lp 1 - s + \mu_l      + \frac 1 2 { |m_l      + m |}   \rp } \right| |d s|.
\end{split}
\end{equation*}
With the observations that for $s \in \EC_m$
\begin{itemize}
\item[-] $\Re s \in [\sigma_m - \epsilon, \rho_m +\epsilon]$,
\vskip 2 pt
\item[-] $\left| \Re s - \mu_l      + \frac 1 2 { |m_l      + m| }   \right| \lll_{ (\umu, \um)} 1$, 
\vskip 2 pt
\item[-] $\left| \lp 1 - \Re s + \mu_l      + \frac 1 2  { |m_l      + m| }   \rp - |m| \right| \lll_{ (\umu, \um)} 1$,
\end{itemize}
in conjunction with Stirling's formula \eqref{1eq: Stirling's formula}, we have the following estimate
\begin{equation*}
\begin{split}
j_{(\umu, \um + m\ue^n)} (x) & \lll_{\, (\umu, \um),\, n,\, \epsilon} \max \left\{ y^{- 2 \sigma_m + 2 \epsilon }, y^{- 2 \rho_m - 2 \epsilon } \right \} \\
& \hskip 40 pt \int_{\EC_m} \frac {(|\Im s| + 1)^{ n \lp \Re s - \frac 1 2 \rp - \Re |\umu| + \frac 1 2 \left| \| \um + m\ue^n\| \right|}} {  e^{- n |m|} \big( \sqrt {(\Im s)^2 + m^2} + 1 \big)^{n \lp \frac 1 2 - \Re s \rp + \Re |\umu| + \frac 1 2 \left| \| \um + m\ue^n\| \right| } } |d s| \\
& \leq  \max \left\{ y^{- 2 \sigma_m + 2 \epsilon }, y^{- 2 \rho_m - 2 \epsilon } \right \} e^{n |m|} 
(|m| + 1)^{n \lp \rho_m + \epsilon - \frac 1 2 \rp - \Re |\umu| - \frac 1 2 \left| \| \um + m\ue^n\| \right|} \\
& \hskip 111 pt \int_{\EC_m} (|\Im s| + 1)^{ n \lp \Re s - \frac 1 2 \rp - \Re |\umu| + \frac 1 2 \left| \| \um + m\ue^n\| \right|} |d s|.
\end{split}
\end{equation*}
For $s \in \EC_m$, we have  $\Re s = \sigma_m - \epsilon$ if $\Im s$ is sufficiently large, and our choice of $\sigma_m$ implies
$\textstyle n \lp \sigma_m - \epsilon - \frac 1 2 \rp - \Re |\umu| + \frac 1 2 \left| \| \um + m\ue^n\| \right| \leq - 1 - n \epsilon$, then it follows that the above integral converges and is of size $O_{(\umu, \um), \epsilon, n} (1)$.

Finally, note that both $ - 2\sigma_m +  2 \epsilon$ and $- 2\rho_m -  2 \epsilon$ are close to $ {|m|} $, whereas the exponent of $(|m| + 1)$, that is $ n \lp  \rho_m + \epsilon - \frac 1 2 \rp - \Re |\umu| - \frac 1 2 \left| \| \um + m\ue^n\| \right| $, is close to $- n |m|$.
Thus the following bounds yield \eqref{2eq: bound of the Bessel kernel, C},
\begin{equation*}
\begin{split}
& \textstyle {|m|} + B_-  \leq - 2 \rho_m  \leq - 2 \sigma_m  \leq  {|m|}   + B_+, \\
&  \textstyle n \lp \rho_m - \frac 1 2 \rp - \Re |\umu| - \textstyle \frac 1 2 \left| \| \um + m\ue^n\| \right| \leq - n |m| + A.
\end{split}
\end{equation*}
When $|m|$ is small, we have the following estimate that also implies \eqref{2eq: bound of the Bessel kernel, C},
\begin{equation*}
j_{(\umu, \um + m\ue^n)} (x) \lll_{ (\umu, \um),\, \epsilon,\, n}
\max \left\{ y^{- 2 \sigma_m + 2 \epsilon }, y^{- 2 \rho_m - 2 \epsilon } \right \} e^{n |m|}.
\end{equation*}
\end{proof}

Using the formula \eqref{3eq: rewrite G m} of $G_m (s)$ instead of \eqref{1def: G m (s)} and the  Barnes type integral representation \eqref{3eq: Bessel kernel, analytic continuation, 2} for $j_{(\umu, \um + m\ue^n)} (\zeta)$ instead of the  Mellin-Barnes type integral representation \eqref{3def: Bessel kernel j mu m} for  $j_{(\umu, \um + m\ue^n)} (x)$, similar arguments in the proof of Lemma \ref{3lem: bound of the Bessel kernel, C} imply the following lemma.

\begin{lem}\label{3lem: bound of the Bessel kernel, analytic continuation, C}
Let $(\umu, \um) \in \BC^{n} \times  \BZ ^n$ and $m \in \BZ$. 
Put 
\begin{align*}
& A =\textstyle  n \lp \max  \{ \Re \mu_l      \} + \frac 1 2 \max  \{ |m_l     | \} - \frac 1 2 \rp - \Re |\umu| + \frac 1 2 |\|\um\||, \\
& B =\textstyle   - 2 \max  \{ \Re \mu_l      \} - \max  \{ |m_l     | \},  \hskip 10 pt C =\textstyle 2 \max  \{ |\Im \mu_l     | \}.
\end{align*} 
Fix $X > 0$ and $ \epsilon > 0$. Then 
\begin{equation*}
\begin{split}
j_{(\umu, \um + m\ue^n)} \lp x e^{i\omega} \rp \lll_{(\umu, \um),\, X,\, \epsilon, \, n} &\ \lp \frac { 2 \pi e x^{\frac 1 { n}}} {|m| + 1} \rp^{n |m|}  (|m| + 1)^{A + n \epsilon} x^{B + 2\epsilon } e^{ |\omega| ( C + 2 \epsilon)}
\end{split}
\end{equation*}
for all $x < X$.
\end{lem}

\subsection {\texorpdfstring{The Hankel transform $\Hld$ and the Bessel kernel $J_{(\umu, \udelta)} $}{The Hankel transform $H_{(\mu, \delta)} $ and the Bessel kernel $J_{(\mu, \delta)} $}}


\subsubsection{\texorpdfstring{The definition of  $ \Hld$}{The definition of  $H_{(\mu, \delta)} $}}

Consider the ordered set $(\BC \times \BZT , \preccurlyeq )$ and define $( \mu_{\alpha, \beta}, \delta_{\alpha, \beta} ) = ( \mu, \delta )_{\alpha, \beta}$,  $B_\alpha$, $M_{\alpha, \beta}$ and $N_{\alpha, \beta}$ as in Definition \ref{3defn: ordered set} corresponding to $( \umu, \udelta ) \in ( \BC \times \BZ / 2\BZ)^n$. We define the following subspaces of $ \Ssis (\BRx)$,
\begin{equation}\label{3eq: Ssis (lambda, delta) delta, R}
\Ssis ^{( \umu, \udelta ), \delta} (\BRx) = \sum_{ \alpha = 1}^A \sum_{\beta = 1}^{B_\alpha}  \sum_{j=0}^{N_{\alpha, \beta} - 1} \sgn(x)^{ \delta_{\alpha, \beta}} |x|^{ - \mu_{\alpha, \beta}} (\log |x|)^{j } \SS_{\delta_{\alpha, \beta} + \delta } ( \BR ).
\end{equation}
\begin{equation}\label{3eq: Ssis (lambda, delta), R}
\begin{split}
\Ssis ^{( \umu, \udelta )}  (\BRx) = & \Ssis ^{( \umu, \udelta ), 0} (\BRx) \oplus \Ssis ^{( \umu, \udelta ), 1} (\BRx) \\
= & \sum_{ \alpha = 1}^A \sum_{\beta = 1}^{B_\alpha}  \sum_{j=0}^{N_{\alpha, \beta} - 1} \sgn(x)^{\delta_{\alpha, \beta}} |x|^{ - \mu_{\alpha, \beta}} (\log |x|)^{j } \SS ( \BR ).
\end{split}
\end{equation}

From the definition  of $\Ssiss^{\ulambda} (\BR _+)$ in \eqref{3eq: Ssis2, R+}, together with $\SS_{\delta} (\BR) = \sgn (x)^\delta \SS_{\delta} (\overline \BR _+)$ and $\SS_{\delta} (\overline \BR _+) = x^{\delta} \SS_{0} (\overline \BR _+) $, we have
\begin{equation}\label{3eq: Ssis (lambda, delta) delta = Ssis2}
\Ssis ^{( \umu, \udelta ), \delta} (\BRx) = \sgn(x)^{\delta} \Ssiss^{\umu - (\udelta + \delta \ue^n)} (\BR _+).
\end{equation}

The following theorem gives the definition of the Hankel transform $\Hld$, which maps $\Ssis^{ (- \umu, \udelta) } (\BRx)$    onto  $\Ssis^{( \umu, \udelta) } (\BRx)$ bijectively.

\begin{thm} \label{3prop: H (lambda, delta)}
Let $(\umu, \udelta) \in \BC^{n} \times (\BZ/2 \BZ)^n$. Suppose $ \upsilon \in \Ssis^{ (- \umu, \udelta) } (\BRx)$.  Then there exists a unique function $\Upsilon \in \Ssis^{( \umu, \udelta) } (\BRx)$ satisfying the following two identities,
\begin{equation}\label{3eq: Hankel transform identity, R}
\EM _\delta \Upsilon (s ) = G_{(\umu,  \udelta + \delta \ue^n)} (s) \EM _\delta \upsilon ( 1 - s), \hskip 10 pt \delta \in \BZ/2 \BZ.
\end{equation}
We call $\Upsilon$ the Hankel transform of  $\upsilon$ over $\BRx$ of index $(\umu, \udelta)$ and  write $\Hld \upsilon  = \Upsilon$. Moreover, we have the Hankel inversion formula
\begin{equation}\label{3eq: Hankel inversion, R}
\Hld \upsilon  = \Upsilon, \hskip 10 pt \Hmld \Upsilon  = \upsilon.
\end{equation}
\end{thm}

\begin{proof}
Recall that 
$$\EM_\delta \upsilon (s) = 2 \EM \upsilon_\delta (s).$$
In view of \eqref{3eq: Ssis (lambda, delta) delta = Ssis2}, one has $\upsilon_\delta \in \Ssiss^{- \umu - (\udelta + \delta \ue^n)} (\BR _+)$. Applying Proposition \ref{3prop: h (lambda, delta)}, there is a unique function $\Upsilon_\delta \in \Ssiss^{ \umu - (\udelta + \delta \ue^n)} (\BR _+)$ satisfying 
\begin{equation*}
\EM  \Upsilon_\delta (s ) = G_{(\umu, \udelta + \delta \ue^n)} (s) \EM  \upsilon_\delta ( 1 - s).
\end{equation*}
According to \eqref{3eq: Ssis (lambda, delta) delta = Ssis2}, $\Upsilon (x) = \Upsilon_0 (|x|) + \sgn (x)\Upsilon_1 (|x|)$ lies in $\Ssis^{( \umu, \udelta), 0} (\BRx) \oplus \Ssis^{( \umu, \udelta), 1} (\BRx) = \Ssis^{( \umu, \udelta) } (\BRx)$. Clearly, $\Upsilon$ satisfies \eqref{3eq: Hankel transform identity, R}. Moreover, \eqref{3eq: Hankel inversion, R} follows immediately from \eqref{3eq: Hankel inversion, R 0} in Proposition  \ref{3prop: h (lambda, delta)}.
\end{proof}

\begin{cor} \label{3cor: H = h, R}
Let $(\umu, \udelta) \in \BC^{n} \times (\BZT)^n$ and $\delta \in \BZT$. Suppose that $\varphi \in \Ssiss^{- \umu - (\udelta + \delta \ue^n)} (\BR _+)$ and $ \upsilon (x) = \sgn (x)^{\delta} \varphi (|x|)$.  Then
\begin{equation*}
\Hld \upsilon (\pm x) = (\pm)^{\delta } \hh_{(\umu, \udelta + \delta \ue^n)} \varphi (x), \hskip 10 pt x \in \BR _+.
\end{equation*}
\end{cor}

\subsubsection{The Bessel kernel $J_{(\umu, \udelta)} $}

Let $(\umu, \udelta) \in \BC^{n} \times (\BZ/2 \BZ)^n$.
We define
\begin{equation}\label{3def: Bessel function, R, 0}
\begin{split}
J_{(\umu, \udelta)} \lp \pm x \rp = \frac 1 2 \sum_{\delta \in \BZ/ 2\BZ} (\pm)^{ \delta} j_{(\umu, \udelta + \delta \ue^n)} (x), \hskip 10 pt x \in \BR_+, 
\end{split}
\end{equation}
or equivalently,
\begin{equation}\label{3def: Bessel function, R}
\begin{split}
J_{(\umu, \udelta)} \lp x \rp = \frac 1 2 \sum_{\delta \in \BZ/ 2\BZ} \sgn (x)^{ \delta} j_{(\umu, \udelta + \delta \ue^n)} (|x|), \hskip 10 pt x \in \BRx.
\end{split}
\end{equation}
Some properties of $J_{(\umu, \udelta)}$ are summarized as below.

\begin{prop}\label{3prop: properties of J, R}
Let $(\umu, \udelta) \in \BC^{n} \times (\BZT)^n$.

{\rm (1).} Let $(\mu, \delta) \in \BC\times \BZT$. We have
\begin{equation*}
J_{(\umu - \mu \ue^n, \udelta - \delta \ue^n)} (x) = \sgn (x)^\delta |x|^{\mu} J_{(\umu, \udelta)} (x).
\end{equation*}

{\rm(2).} $J_{(\umu, \udelta)} (x)$  is a real analytic function of $x$ on $\BRx$ as well as an analytic function of $\umu$ on $\BC^{n}$.

{\rm(3).} Assume that $\umu$ satisfies the  condition
\begin{equation}\label{3eq: condition on lambda}
\textstyle \min  \left\{ \Re \mu_l      \right\} + 1 > \max  \left\{ \Re \mu_l      \right\}.
\end{equation}
Then for $ \upsilon \in \Ssis^{ (- \umu, \udelta) } (\BRx)$ 
\begin{equation}\label{3eq: Hankel transform, with Bessel kernel, R}
\Hld \upsilon (x) = \int_{\BR ^\times} \upsilon (y) J_{(\umu, \udelta)} (xy ) d y.
\end{equation}
Moreover, if $\upsilon \in \SS (\BRx)$, then \eqref{3eq: Hankel transform, with Bessel kernel, R} remains true for any index $ \umu \in \BC^n $.
\end{prop}

\begin{example}\label{3ex: Bessel n=1, R}
For $n=1$, we have 
$$J_{(0, 0)} (x) =  e(x).$$
For $n = 2$, \eqref{3eq: j (lambda, delta) and fundamental} and \eqref{3def: Bessel function, R, 0} yield 
\begin{equation*}
J_{(\mu, - \mu, \delta, 0)} (\pm x) = J \lp 2\pi \sqrt x; +, \pm , \mu, - \mu \rp + (-)^{\delta } J \lp 2\pi \sqrt x; -, \mp, \mu, - \mu \rp\end{equation*}
for $x \in \BR _+$, $\mu \in \BC$ and $\delta \in \BZT$.
In view of Example {\rm \ref{ex: fundamental Bessel, n=1, 2}}, for $x \in \BR _+$,  we have
\begin{equation*}
J_{(\mu, - \mu, \delta, 0)} (x)  = \left\{ 
\begin{split}
& - \frac {\pi } {\sin (\pi \mu)} \lp J_{2 \mu} (4\pi \sqrt x ) - J_{-2 \mu} (4\pi \sqrt x ) \rp, \hskip 10 pt \text{ if } \delta = 0, \\
&  \frac {\pi i} {\cos (\pi \mu)} \lp J_{2 \mu} (4\pi \sqrt x) + J_{-2 \mu} (4\pi \sqrt x ) \rp , \hskip 21 pt \text{ if } \delta = 1,
\end{split}
\right.
\end{equation*}
where the right hand side is replaced by its limit if  $2 \mu \in  \delta + 2\BZ$, and
\begin{equation*} 
J_{(\mu, - \mu, \delta, 0)} (-x) = 
\left\{ 
\begin{split}
& 4 \cos (\pi \mu) K_{2 \mu} (4 \pi \sqrt x ), \hskip 24 pt \text{ if } \delta = 0,\\
& - 4 i \sin (\pi \mu) K_{2 \mu} (4 \pi  \sqrt x ), \hskip 11 pt \text{ if } \delta = 1.
\end{split}
\right.
\end{equation*}
Observe that for $m \in \BN $
\begin{equation*}
J_{\lp \frac 12 m , - \frac 12 m ,\, \delta (m) + 1, 0 \rp} (x) = 2 \pi i^{m+1} J_m (4\pi \sqrt x), \hskip 10 pt J_{\lp \frac 12  m , - \frac 12 m ,\, \delta (m) + 1, 0 \rp} (-x) = 0.
\end{equation*}
\end{example}

\subsection {\texorpdfstring{The Hankel transform $ \Hmum$ and the Bessel kernel $J_{(\umu, \um)} $}{The Hankel transform $ H_{(\mu, m)}$ and the Bessel kernel $J_{(\mu, m)} $}}

\subsubsection{\texorpdfstring{The definition of  $ \Hmum$}{The definition of $ H_{(\mu, m)}$}}

Consider now the ordered set $(\BC \times \BZ , \preccurlyeq )$ and define $( 2 \mu_{\alpha, \beta}, m_{\alpha, \beta} ) = ( 2 \mu, m )_{\alpha, \beta}$,  $B_\alpha$, $M_{\alpha, \beta}$ and $N_{\alpha, \beta}$ as in Definition \ref{3defn: ordered set} corresponding to $( 2 \umu, \um ) \in ( \BC \times \BZ )^n$. We define the following subspace of $ \Ssis (\BCx)$,
\begin{equation}\label{3eq: Ssis (mu, m), C}
\begin{split}
\Ssis ^{( \umu, \um )}  (\BCx) = \sum_{ \alpha = 1}^A \sum_{\beta = 1}^{B_\alpha}  \sum_{j=0}^{N_{\alpha, \beta} - 1} [z]^{- m_{\alpha, \beta}} \|z\|^{ -  \mu_{\alpha, \beta}} (\log |z|)^{j } \SS ( \BC ).
\end{split}
\end{equation}
The projection via the $m$-th Fourier coefficient maps $\Ssis ^{( \umu, \um )}  (\BCx) $ onto the space
\begin{equation}\label{3eq: Ssis (mu, m) m, C}
\Ssis ^{( \umu, \um ), m} (\BCx) = \sum_{ \alpha = 1}^A \sum_{\beta = 1}^{B_\alpha}  \sum_{j=0}^{N_{\alpha, \beta} - 1} [z]^{ - m_{\alpha, \beta}} \|z\|^{ - \mu_{\alpha, \beta}} (\log |z|)^{j } \SS_{m_{\alpha, \beta} + m} ( \BC ).
\end{equation}
From the definition of $\Ssiss^{\ulambda} (\BR _+)$ in \eqref{3eq: Ssis2, R+}, along with $\SS_{m} (\BC) = [z]^m \SS_{m} (\overline \BR _+)$  and $\SS_{m} (\overline \BR _+) = x^{|m|} \SS_{0} (\overline \BR _+) $, we have
\begin{equation}\label{3eq: Ssis (mu, m) m = Ssis2}
\Ssis ^{( \umu, \um ), m} (\BCx) = [z]^{m} \Ssiss^{2 \umu - \| \um + m\ue^n \|  } (\BR _+).
\end{equation}

The following theorem gives the definition of the Hankel transform $\Hmum$, which maps $\Ssis^{ (- \umu, - \um) } (\BCx)$    onto  $\Ssis^{( \umu, \um) } (\BCx)$ bijectively.

\begin{thm} \label{3prop: H (mu, m)}
Let $(\umu, \um) \in \BC^{n} \times \BZ ^n$. Suppose $ \upsilon \in \Ssis^{ (- \umu, - \um) } (\BCx)$.  Then there exists a unique function $\Upsilon \in \Ssis^{( \umu, \um) } (\BCx)$ satisfying the following sequence of identities,
\begin{equation}\label{3eq: Hankel transform identity, C}
\EM _{-m} \Upsilon (2 s ) = G_{(\umu,  \um + m\ue^n)} (s) \EM _m \upsilon ( 2 (1-s) ), \hskip 10 pt m \in \BZ .
\end{equation}
We call $\Upsilon$ the Hankel transform of  $\upsilon$ over $\BCx$ of index $(\umu, \um)$ and  write $\Hmum \upsilon  = \Upsilon$. Moreover, we have the Hankel inversion formula
\begin{equation}\label{3eq: Hankel inversion, C}
\Hmum \upsilon  = \Upsilon, \hskip 10 pt \Hmmum \Upsilon  = \upsilon.
\end{equation}
\end{thm}

\begin{proof}
Recall that 
$$\EM_{m} \upsilon (s) = 4 \pi \EM \upsilon_{-m} (s).$$
In view of \eqref{3eq: Ssis (mu, m) m = Ssis2}, we have $\upsilon_{- m} \in \Ssiss^{- 2 \umu - \| \um + m\ue^n \|} (\BR _+)$. Applying Proposition \ref{3prop: h (mu, m)}, we infer that there is a unique function $\Upsilon_m \in \Ssiss^{ 2 \umu - \| \um + m\ue^n \|} (\BR _+)$ satisfying 
\begin{equation*} 
\EM  \Upsilon_{m} (2 s ) = G_{(\umu,  \um + m\ue^n)} (s) \EM  \upsilon_{-m} ( 2(1-s) ).
\end{equation*}
According to Lemma \ref{2lem: Ssis to Msis, C}, in order to show that the Fourier series $\Upsilon \lp x e^{i \phi}\rp = \sum  \Upsilon_m (x) e^{im \phi}$ lies in $\Ssis^{( \umu, \um) } (\BCx)$, it suffices to verify that  $G_{(\umu,  \um + m\ue^n)} (s) \EM  \upsilon_{-m} ( 2(1-s) )$   rapidly decays with respect to $m$, uniformly on vertical strips. 
This however follows from the uniform rapid decay of $\EM  \upsilon_{-m} ( 2(1-s) )$ along with the uniform moderate growth of $G_{(\umu,  \um + m\ue^n)} (s)$ (\eqref{1eq: vertical bound, G (mu, m) (s)} in Lemma \ref{1lem: vertical bound}) in the $m$ aspect on vertical strips. 

Finally,  \eqref{3eq: Hankel inversion, C 0} in Proposition  \ref{3prop: h (mu, m)} implies \eqref{3eq: Hankel inversion, C}.
\end{proof}

\begin{cor} \label{3cor: H = h, C}
Let $(\umu, \um) \in \BC^{n} \times \BZ^n$ and $m \in \BZ $. Suppose $\varphi \in \Ssiss^{- 2 \umu - \| \um + m\ue^n \|} (\BR _+)$ and $ \upsilon (z) = [z]^{- m} \varphi (|z|)$.  Then
\begin{equation*}
\Hmum \upsilon \lp x e^{i\phi} \rp = e^{i m \phi} \hh_{(\umu,  \um + m\ue^n)} \varphi (x), \hskip 10 pt x \in \BR _+, \phi \in \BR / 2 \pi \BZ.
\end{equation*}
\end{cor}

\subsubsection{\texorpdfstring{The Bessel kernel $J_{(\umu, \um)} $}{The Bessel kernel $J_{(\mu, m)} $}}

For $(\umu, \um) \in \BC^{n} \times \BZ^n$, we define
\begin{equation}\label{2eq: Bessel kernel over C, polar}
 J_{(\umu, \um)} \lp x e^{i \phi} \rp =  \frac 1 {2 \pi} \sum_{m \in \BZ} j_{(\umu, \um + m\ue^n)} (x ) e^{ i m \phi}, 
\end{equation}
or equivalently,
\begin{equation}\label{2eq: Bessel kernel over C}
J_{(\umu, \um)} \lp z \rp =  \frac 1 {2 \pi} \sum_{m \in \BZ} j_{(\umu, \um + m\ue^n)} (|z|) [z]^m.
\end{equation}
Lemma \ref{3lem: bound of the Bessel kernel, C} secures the absolute convergence of this series.

\begin{prop}\label{3prop: properties of J, C}
Let $(\umu, \um) \in \BC^{n} \times \BZ^n$.

{\rm (1).} Let $(\mu, m) \in \BC \times \BZ$. We have
\begin{equation*}
J_{(\umu - \mu \ue^n, \um - m \ue^n)} (z) = [z]^m \|z\|^{\mu} J_{(\umu, \um)} (z).
\end{equation*}

{\rm (2).} $ J_{(\umu, \um)} (z)$  is a real analytic function of $z$ on $\BCx$ as well as an analytic function of $\umu$ on $\BC^{n}$.

{\rm (3).} Assume that $\umu$ satisfies the following condition
\begin{equation}\label{3eq: condition on mu}
\textstyle  \min  \left\{ \Re \mu_l      \right\} + 1 > \max  \left\{ \Re \mu_l      \right\}.
\end{equation}
Suppose $ \upsilon \in \Ssis^{ (- \umu, - \um) } (\BCx)$. Then 
\begin{equation} \label{3eq: Hankel transform, C, polar}
\Upsilon \lp x e^{i \phi} \rp  =  \int_0^\infty \int_0^{2\pi} \upsilon \lp y e^{i \theta}\rp  J_{(\umu, \um)} \lp x e^{ i \phi } y e^{ i \theta } \rp \cdot 2 y d \theta dy,
\end{equation}
or equivalently,
\begin{equation} \label{3eq: Hankel transform, with Bessel kernel, C}
\Upsilon (z)  =  \int_{\BCx} \upsilon (u) J_{(\umu, \um)} ( zu ) d u.
\end{equation}
Moreover, \eqref{3eq: Hankel transform, C, polar} and \eqref{3eq: Hankel transform, with Bessel kernel, C} still hold true for any index $ \umu \in \BC$ if $\upsilon \in \SS (\BCx)$.
\end{prop}

\begin{proof}

(1). This is clear.

(2). In \eqref{2eq: Bessel kernel over C, polar}, with abuse of notation, we view $x$ and $\phi$ as complex variables on $\BU$ and $\BC/2 \pi \BZ$ respectively, $j_{(\umu, \um + m\ue^n)} (x )$ and $ e^{ i m \phi}$ as analytic functions. Then Lemma \ref{3lem: bound of the Bessel kernel, analytic continuation, C} implies that the series in \eqref{2eq: Bessel kernel over C, polar} is absolutely convergent, compactly with respect to both $x$ and $\phi$, and therefore $ J_{(\umu, \um)} \lp x e^{i \phi} \rp$ is an analytic function of $x$ and $\phi$. In particular, $ J_{(\umu, \um)} (z)$  is a \textit{real analytic} function of $z$ on $\BCx$.

Moreover, in  Lemma \ref{3lem: bound of the Bessel kernel, C}, we may allow $\umu$ to vary in an $\epsilon$-ball in $\BC^{n}$ and choose the implied constant in the estimate to be uniformly bounded with respect to $\umu$. This implies that the series in \eqref{2eq: Bessel kernel over C, polar} is convergent compactly in the $\umu$ aspect.
Therefore, $ J_{(\umu, \um)} (z)$ is an analytic function of $\umu$ on $\BC^{n}$.

(3). It follows from  \eqref{3eq: Ssis (mu, m) m = Ssis2} that $\upsilon_{- m} \in \Ssiss^{- 2 \umu - \| \um + m\ue^n\|} (\BR _+)$. Moreover, one observes that  $(\umu, \um + m \ue^n)$ satisfies the condition \eqref{3eq: condition on (mu, m), 0} due to \eqref{3eq: condition on mu}. Therefore, in conjunction with  Proposition \ref{3prop: h (mu, m)}, \eqref{3eq: h mu m = int of j mu m} implies 
\begin{equation*} 
\Upsilon_{m} (x ) = 2 \int_0^\infty \upsilon_{- m} (y) j_{(\umu, \um + m\ue^n)} \lp x y \rp y d y.
\end{equation*}
Hence
\begin{equation*} 
\begin{split}
\Upsilon \lp x e^{i \phi} \rp =   \sum_{m \in \BZ} \Upsilon_{ m} (x ) e^{i m \phi}  
=   \sum_{m \in \BZ} \frac 1 { \pi} \int_0^\infty \int_0^{2\pi} \upsilon \lp y e^{i \theta}\rp j_{(\umu, \um + m\ue^n)} \lp x y \rp e^{ i m (\phi + \theta)} y d \theta d y.
\end{split}
\end{equation*}
The estimate of $j_{(\umu, \um + m\ue^n)}$ in Lemma \ref{3lem: bound of the Bessel kernel, C} implies that
the above series of integrals converges absolutely. On interchanging the order of summation and integration, one obtains \eqref{3eq: Hankel transform, C, polar}  in view of the definition of $J_{(\umu, \um)}$ in \eqref{2eq: Bessel kernel over C, polar}.

Note that in the case $\upsilon \in \SS (\BCx)$, one has $\upsilon_{- m}  \in \SS (\BR _+)$, and therefore \eqref{3eq: h mu m = int of j mu m} can be applied unconditionally. 
\end{proof}


\begin{example}\label{3ex: Bessel n=1, C}
Let $n = 1$. 
From \eqref{3eq: j (0, m)}, we have
\begin{equation*}
j_{(0, m)} (x ) = 
\left \{
\begin{split}
& (-)^d 2\pi J_{2 d } (4 \pi x ), \hskip 23 pt \text { if } |m| = 2 d,\\
& (-)^{d } 2\pi i J_{2 d + 1} (4 \pi x ), \hskip 10 pt \text { if } |m| = 2 d + 1.
\end{split}
\right.
\end{equation*}
The following expansions {\rm (\cite[2.22 (3, 4)]{Watson})}
\begin{align*}
\cos (x \cos \phi) &= J_0(x) + 2 \sum_{d = 1}^\infty (-)^d J_{2 d} (x) \cos (2 d \phi),\\
\sin (x \cos \phi) &=   2 \sum_{d = 0}^\infty (-)^d J_{2 d + 1} (x) \cos ((2 d + 1) \phi),
\end{align*}
imply
\begin{equation*}
J_{(0, 0)} \lp x e^{i \phi} \rp = \cos (4 \pi x \cos \phi) + i \sin (4 \pi x \cos \phi) = e(2 x \cos \phi ),
\end{equation*}
or equivalently,
\begin{equation*}
J_{(0, 0)} (z) = e(z + \overline z).
\end{equation*}
We remark that the two expansions \cite[2.22 (3, 4)]{Watson} can be incorporated into 
\begin{equation*}
e^{i x \cos \phi} = \sum_{m = -\infty}^{\infty}  i^{m}  J_m (x) e^{im \phi}.
\end{equation*}
\end{example}

	\subsection{Concluding remarks}
	
	\subsubsection{Connection formulae}\label{sec: connection formulae}

	From the various connection formulae (\ref{3eq: j (lambda, delta) and fundamental}, \ref{3eq: j mu m = j lambda delta}, \ref{2eq: Bessel kernel over C, polar}, \ref{2eq: Bessel kernel over C}) which have been derived so far, 
	one can connect  the Bessel kernel $J_{(\umu, \um)} (z)$   to the Bessel functions  $J(x; \usigma, \ulambda)$ of doubled rank $2n$. However, in contrast to the expression of $J_{(\umu, \udelta)} (\pm x)$ by a {\it finite} sum of $J \big(2 \pi x^{\frac 1 n}; \usigma, \umu \big)$ (see  (\ref{3eq: j (lambda, delta) and fundamental},  \ref{3def: Bessel function, R, 0}, \ref{3def: Bessel function, R})), which enables us to reduce  the study of $J_{(\umu, \udelta)} (x) $  to that of  $J(x; \usigma, \ulambda)$ given in \cite{Qi}, these connection formulae yield an expression of $J_{(\umu, \um)} \lp x e^{i \phi} \rp$ in terms of an {\it infinite} series involving the Bessel functions $J \big(2 \pi x^{\frac 1 n}; \usigma, \ulambda \big)$ of rank $2n$,  so a similar reduction for $J_{(\umu, \um)} (z)$ does not  exist from this approach. 
	
	In  \S \ref{sec: two connection formulae for J mu m}, we shall prove two alternative connection formulae that relate $J_{(\umu, \um)} (z)$ to  the two kinds of Bessel functions of rank $n$ and positive sign. These kinds of Bessel functions arise in \cite[\S 7]{Qi} as solutions of the Bessel equation of positive sign.

	

	\subsubsection{Asymptotics of Bessel kernels}\label{sec: sub asymptotic}
	
	Using the connection formulae between the Bessel kernel $J_{(\umu, \udelta)} ( x)$ and Bessel functions $J(x; \usigma, \ulambda)$ along with the asymptotics of the latter, the asymptotic of $J_{(\umu, \udelta)} (x)$ is readily established in \cite[Theorem 5.13, 9.3]{Qi}. With the help of the second connection formula  for $J_{(\umu, \um)} (z)$ in \S \ref{sec: second connection}, we shall present  in \S \ref{sec: asymptotics} the asymptotic of $J_{(\umu, \um)} (z) $ as an application of the asymptotic expansions of Bessel functions of the second kind \cite[Theorem 7.27]{Qi}.

	\subsubsection{Normalizations of indices} 
	Usually, it is convenient to normalize the indices in $J(x; \usigma, \ulambda)$,  $j_{(\umu, \udelta)} (x)$, $j_{(\umu, \um)} (x)$, $J_{(\umu, \udelta)} (x)$ and $J_{(\umu, \um)} (z)$ so that $\ulambda, \umu \in \BL^{n-1}$. Furthermore, without loss of generality, the assumptions $\delta_n = 0$ and $m_n = 0$ may also be imposed for $J_{(\umu, \udelta)} (x)$ and $J_{(\umu, \um)} (z)$ respectively.
	These normalizations are justified by Lemma {\rm \ref{3lem: normalize J(x; sigma, lambda)}}, \eqref{3eq: normalize j (lambda, delta)}, \eqref{3eq: normalize j (mu, m)}, Proposition \ref{3prop: properties of J, R} (1) and \ref{3prop: properties of J, C} (1).

\section{Fourier type integral transforms} 
\label{sec: Fourier type transforms}

In this section, we shall introduce an alternative perspective of Hankel transforms. 
We shall first show how to construct Hankel transforms from the Fourier transform and   Miller-Schmid transforms. From this, we shall express the Hankel transforms  $\Hld$ and $\Hmum$ in terms of certain Fourier type integral transforms, assuming that the components of $\Re \umu$ are {\it strictly} decreasing.


\subsection{The Fourier transform and rank-one Hankel transforms}

For either $\BF = \BR$ or $\BF = \BC$, we have seen in Example \ref{3ex: Bessel n=1, R} and \ref{3ex: Bessel n=1, C} that $J_{(0, 0)}$ is exactly the inverse Fourier kernel, namely
$$J_{(0, 0)} (x) = e (\Lambda (x)), \hskip 10 pt x \in \BF,$$ with $\Lambda (x)$ defined by \eqref{1eq: Lambda (x)}. Therefore, in view of Proposition \ref{3prop: properties of J, R} (3) and \ref{3prop: properties of J, C} (3), $\EH_{(0, 0)} $ is precisely the inverse Fourier transform over the Schwartz space $\Ssis^{(0, 0)} (\BFx) = \SS (\BF)$. The following lemma is a consequence of Theorem  \ref{3prop: H (lambda, delta)} and   \ref{3prop: H (mu, m)}.

\begin{lem}\label{5lem: Fourier}
Let $\upsilon \in \SS (\BF)$. If $\BF = \BR$, then the Fourier transform $\widehat \upsilon $ of $\upsilon$ can be determined by the following two identities
\begin{equation*}
\EM_{\delta} \widehat \upsilon (s) = (-)^\delta G_{\delta} (s) \EM_{\delta} \upsilon (1-s), \hskip 10 pt \delta \in \BZT.
\end{equation*}
If $\BF = \BC$, then the Fourier transform $\widehat \upsilon $ of $\upsilon$ can be determined by the following sequence of identities
\begin{equation*}
\EM_{- m} \widehat \upsilon (2s) = (-)^m G_{m} (s) \EM_{m} \upsilon (2 (1-s)), \hskip 10 pt m \in \BZ.
\end{equation*}
\end{lem}

It is convenient for our purpose to introduce the renormalized rank-one Hankel transforms $\ES_{(\mu, \epsilon)}$ and $\ES_{(\mu, k)}$ as follows.

\begin{lem}\label{5cor: renormalize Hankel}
	Let $(\mu, \epsilon) \in \BC \times \BZT$ and $(\mu, k) \in \BC \times \BZ$. 
	
	{\rm (1).} For $\upsilon (x) \in \sgn (x)^{\epsilon} |x|^{\mu} \SS (\BR) $, define $\ES_{(\mu, \epsilon)} \upsilon (x) =    |x|^{ \mu}  \EH_{(\mu, \epsilon)} \upsilon (x)$. Then 
	\begin{equation}\label{5eq: Mellin S, R}
	\EM_{\delta} \ES_{(\mu, \epsilon)} \upsilon (s) = G_{\epsilon + \delta} (s) \EM_{\delta} \upsilon (1-s-\mu), \hskip 10 pt \delta \in \BZT,
	\end{equation}
	and $\ES_{(\mu, \epsilon)}$ sends $\sgn (x)^{\epsilon} |x|^{\mu} \SS (\BR)$   onto $\sgn (x)^{\epsilon} \SS (\BR)$ bijectively. Furthermore, 
	\begin{equation*}
	\ES_{(\mu, \epsilon)} \upsilon (x) =  \sgn(x)^{\epsilon} \int_{\BRx} \sgn (y)^\epsilon |y|^{-\mu} \upsilon (y) e (x y) d y   
	= \sgn (x)^\epsilon  \EF \varphi (-x),
	\end{equation*}
	with $\varphi (x) = \sgn(x)^\epsilon |x |^{-\mu} \upsilon (x) \in \SS (\BR).$
	
	{\rm(2).} For $\upsilon (z) \in [z]^{ k} \|z\|^{\mu} \SS (\BC) $, define $\ES_{(\mu, k)} \upsilon (z) =    \|z\|^{ \mu}  \EH_{(\mu, k)} \upsilon (z)$. Then 
	\begin{equation}\label{5eq: Mellin S, C}
	\EM_{-m} \ES_{(\mu, k)} \upsilon (2 s) = G_{k+m} (s) \EM_{m} \upsilon (2(1-s-\mu)), \hskip 10 pt m \in \BZ.
	\end{equation} 
	and $\ES_{(\mu, k)}$ sends $[z]^{ k} \|z\|^{\mu} \SS (\BC) $   onto $ [z]^{- k} \SS (\BC)$ bijectively. Furthermore,
	\begin{equation*}
	\begin{split}
	\ES_{(\mu, k)} \upsilon (z) =   [z]^{-k} \int_{\BCx} [u]^{-k} \|u\|^{-\mu} \upsilon (u) e (z u + \overline {z u}) d u
	= [z]^{-k}  \EF \varphi (-z),
	\end{split}
	\end{equation*}
	with $ \varphi (z) =  [z]^{-k} \| z \|^{-\mu} \upsilon (z) \in \SS (\BC) $.
	
\end{lem}

\delete{According to \eqref{3def: Bessel function, R, 0} and \eqref{2eq: Bessel kernel over C, polar}, one has the decomposition
\begin{equation*}
J_{(0, 0)} \lp \pm x \rp = \frac 1 2 \lp j_{(0, 0)} (x)  \pm j_{(0, 1)} (x) \rp, \hskip 10 pt x\in \BR _+, 
\end{equation*}
and the series expansion
\begin{equation*}
J_{(0, 0)} \lp x  e^{i \phi} \rp = \frac 1 {2 \pi} \sum_{m \in \BZ} j_{(0, m)} (x)  e^{i m \phi}, \hskip 10 pt x\in \BR _+, \ \phi \in \BR / 2 \pi \BZ.
\end{equation*}
Recall the formulae of $j_{(0, \delta)}$, $\delta \in \BZT$, and $j_{(0, m)}$, $m \in \BZ$, that is, \eqref{3eq: j (0, delta)} in Example \ref{ex: Bessel kernel, real, 0, n=1, 2} and \eqref{3eq: j (0, m)} in Example \ref{ex: Bessel kernel j 0 m, complex, n=1}.  Some calculations yield the following lemma.
}

\begin{lem}\label{5lem: S, R and C}
Let $(\mu, \epsilon) \in \BC \times \BZT$ and $(\mu, k) \in \BC \times \BZ$. 

{\rm (1).} Let $\delta \in \BZT$. Suppose that  $\varphi (x) \in x^{ \mu} \SS_{\delta + \epsilon} (\overline \BR_+)$ and $\upsilon (x) =  \sgn(x)^{\delta}  \varphi (|x|) $. Then
\begin{equation*}
\begin{split}
\ES_{(\mu, \epsilon)} \upsilon (\pm x) &=  (\pm)^{\delta } \int_{\BR_+} y^{- \mu} \varphi (y) j_{(0, \delta + \epsilon)} (x y) d y \\
& = 
\begin{cases}
\ds (\pm)^{\epsilon} 2 \int_{\BR_+} y^{- \mu} \varphi (y) \cos (x y) d y, &  \text { if } \delta = \epsilon,\\
\ds (\pm)^{\epsilon + 1} 2 i \int_{\BR_+} y^{- \mu} \varphi (y) \sin (x y) d y, &  \text { if } \delta = \epsilon + 1.
\end{cases}
\end{split}
\end{equation*}
The transform $\ES_{(\mu, \epsilon)}$ is a bijective map from  $\sgn (x)^{\epsilon} |x|^{\mu} \SS_{\delta} (\BR)$ onto $\sgn (x)^{\epsilon} \SS_{\delta} (\BR)$.

{\rm (2).} Let $m \in \BZ$.  Suppose that  $\varphi (x) \in x^{2 \mu} \SS_{-m-k} (\overline \BR_+)$ and $\upsilon (z) =  [z]^{- m }  \varphi (|z|) $. Then
\begin{equation*}
\begin{split}
\ES_{(\mu, k)} \upsilon \lp x e^{i \phi} \rp & = 2 e^{ i m \phi} \int_{\BR_+} y^{1 - 2 \mu} \varphi (y) j_{(0, m + k)} (x y)   d y \\
 & =  4 \pi i^{m+k} e^{ i m \phi} \int_{\BR_+}  y^{1 - 2 \mu} \varphi (y) J_{m+k} (4 \pi x y) d y.
\end{split}
\end{equation*}
The transform  $\ES_{(\mu, k)}$ is a bijective map from  $[z]^{ k} \|z\|^{\mu} \SS_m (\BC) $ onto $ [z]^{- k} \SS_{-m} (\BC)$.
\end{lem}

\subsection{Miller-Schmid transforms}

In \cite[\S 6]{Miller-Schmid-2004}, certain transforms over $\BR$, which play an important role in the proof of the \Voronoi summation formula in their subsequent work \cite{Miller-Schmid-2006, Miller-Schmid-2009}, are introduced by Miller and Schmid. Here, we shall first recollect their construction of these transforms with slight modifications, and then define similar transforms over $\BC$ in a parallel way.

\subsubsection{The Miller-Schmid transform $\ET _{(\mu, \epsilon)}$}

\begin{lem} \label{5lem: Miller-Schmid transforms, R}
Let $(\mu, \epsilon) \in \BC \times \BZ/2 \BZ $.

{\rm (1).}
For any $ \upsilon \in \Ssis (\BRx)$ there is a unique function $\Upsilon \in \Ssis (\BRx)$ satisfying the following two identities,
\begin{equation}\label{5eq: Miller-Schmid, R}
\EM _\delta \Upsilon (s ) = G_{ \epsilon + \delta } (s) \EM _\delta \upsilon (s + \mu), \hskip 10 pt \delta \in \BZ/2 \BZ.
\end{equation}
We  write $ \Upsilon = \ET_{(\mu, \epsilon)} \upsilon $ and call $\ET_{(\mu, \epsilon)}$ the Miller-Schmid transform over $\BR$ of index $(\mu, \epsilon)$.

{\rm (2).}
Let $\lambda \in \BC$. Suppose  $\upsilon (x) \in \sgn (x)^{\delta } |x|^{-\lambda} (\log |x|)^j \SS (\BR)$. If  $\Re \lambda <  \Re \mu - \frac 1 2$, then
\begin{equation}\label{5eq: Miller-Schmid and Fourier, R}
\begin{split}
\ET_{(\mu, \epsilon)} \upsilon (x)
& = \sgn (x)^{\epsilon} \int_{\BRx} \sgn (y)^{\epsilon} |y|^{-\mu} \upsilon \lp y\-\rp  e (x y) d^\times y   = \sgn (x)^{\epsilon} \EF \varphi \, (- x),
\end{split}
\end{equation}
with $\varphi (x) = \sgn(x)^\epsilon |x |^{-\mu-1}  \upsilon (x\-)$.

{\rm (3).} Suppose that $\Re \lambda < \Re \mu$. Then the integral in \eqref{5eq: Miller-Schmid and Fourier, R} is absolutely convergent and \eqref{5eq: Miller-Schmid and Fourier, R} remains valid for any  $\upsilon (x) \in \sgn (x)^{\delta } |x|^{-\lambda} (\log |x|)^j \SS (\BR)$.

{\rm (4).}
Suppose that $\Re \mu > 0$. Define the function space
\begin{equation*}
\Ssiss (\BRx) =  \sum_{\delta \in \BZT} \ \sum_{\Re \lambda \leq 0 } \ \sum_{j \in \BN}  \sgn(x)^{\delta} |x|^{-\lambda} (\log |x|)^j \SS (\BR).
\end{equation*}
Then the transform $\ET_{(\mu, \epsilon)}$ sends $\Ssiss (\BRx)$ into itself. Moreover, \eqref{5eq: Miller-Schmid and Fourier, R} also holds true for any $\upsilon \in \Ssiss (\BRx)$, wherein the integral absolutely converges.
\end{lem}

\begin{proof}
Following the ideas in the proofs of Proposition \ref{3prop: h (lambda, delta)} and Theorem \ref{3prop: H (lambda, delta)}, one may prove (1).
Actually, the case here is much easier!

As for (2), we have
\begin{equation*}
\begin{split}
\ET_{(\mu, \epsilon)} \upsilon (x)
& = \frac 1 {4 \pi i} \sum_{\delta \in \BZT} \int_{\BRx} \upsilon (y) |y|^{\mu} \cdot \sgn (x y)^{\delta} \int_{ \EC _{ (0, \delta + \epsilon)} } G_{\delta + \epsilon} (s) |y|^{s } |x|^{-s} d s d^\times y\\
& = \int_{\BRx} \upsilon (y) |y|^{\mu} J_{(0, \epsilon)} \lp x y\-\rp d^\times y,
\end{split}
\end{equation*}
 provided that the double integral is absolutely convergent. In order to guarantee the convergence of the integral over $d^\times  y$, the integral contour $\EC _{ (0, \delta + \epsilon)}$ is required to lie in the right half-plane $\left\{ s : \Re s > \Re (\lambda - \mu) \right \}$. In view of Definition \ref{3defn: C d lambda}, such a choice of  $\EC _{(0, \delta + \epsilon)}$ is permissible since $\Re (\lambda - \mu) <  - \frac 1 2$ according to our assumption. Finally, the change of variables from $y$ to $y\-$, along with the formula $J_{(0, \epsilon)} (x) = \sgn (x)^{\epsilon} e (x)$, yields \eqref{5eq: Miller-Schmid and Fourier, R}. 

For the case  $\Re \lambda < \Re \mu$ in (3), the absolute convergence of the integral in \eqref{5eq: Miller-Schmid and Fourier, R} is obvious. The validity of \eqref{5eq: Miller-Schmid and Fourier, R} follows from the analyticity with respect to $\mu$.

Observe that, under the isomorphism established by $\EM_{\BR}$ in Lemma \ref{2lem: Ssis to Msis, R}, $\Ssiss (\BRx)$ corresponds to the subspace of $\Msis^{\BR}$ consisting of pairs of meromorphic functions $(H_0, H_1)$ such that the poles of both $H_0$ and $ H_1$ lie in the left half-plane $\left\{ s : \Re s \leq 0 \right\}$ (see Lemma \ref{2lem: Ssis delta, M delta}).
Then the first assertion in (4) is clear, since  the map that corresponds to  $\ET _{(\mu, \epsilon)}$ is given by $ (H_0 (s), H_1 (s)) \mapsto (G_{ \epsilon } (s) H_0 (s + \mu), G_{ \epsilon + 1 } (s) H_1 (s + \mu)) $ and sends the  subspace of $\Msis^{\BR}$ described above into itself. The second assertion  in (4) immediately follows from (3).
\end{proof}

Similar to Lemma \ref{5lem: S, R and C} (1), we have the following lemma.

\begin{lem}\label{5lem: T, R}
Let $(\mu, \epsilon)  \in \BC \times \BZT$ be such that $\Re \mu > 0$. For $\delta \in \BZT$ define $\Ssiss^{ \delta} (\BRx)$ to be the space of functions in $ \Ssiss (\BRx)$  satisfying the condition \eqref{1eq: delta condition, R}. For  $\upsilon \in \Ssiss^{ \delta} (\BRx)$, we write $\upsilon (x) =  \sgn(x)^{\delta }  \varphi (|x|) $. Then
\begin{equation*}
\begin{split}
\ET_{(\mu, \epsilon)} \upsilon (\pm x) &=  (\pm)^{\delta } \int_{\BR_+} y^{- \mu} \varphi \lp y\- \rp j_{(0, \delta + \epsilon)} (x y) d^\times y \\
& = 
\begin{cases}
\ds (\pm)^{\delta} 2 \int_{\BR_+} y^{- \mu} \varphi \lp y\- \rp \cos (x y) d^\times y, &  \text { if } \delta = \epsilon,\\
\ds (\pm)^{\delta} 2 i \int_{\BR_+} y^{- \mu} \varphi \lp y\- \rp \sin (x y) d^\times y, &  \text { if } \delta = \epsilon + 1.
\end{cases}
\end{split}
\end{equation*}
The transform $\ET_{(\mu, \epsilon)}$ sends $\Ssiss^{ \delta} (\BRx)$ into itself.
\end{lem}

\subsubsection{The Miller-Schmid transform $\ET _{(\mu, k)}$}
In parallel to Lemma \ref{5lem: Miller-Schmid transforms, R}, the following lemma defines the Miller-Schmid transform $\ET _{(\mu, k)}$ over $\BC$ and gives its connection to the Fourier transform over $\BC$.

\begin{lem} \label{5lem: Miller-Schmid transforms, C}
Let $(\mu, k) \in \BC \times \BZ $.

{\rm (1).}
For any $ \upsilon \in \Ssis (\BCx)$ there is a unique function $\Upsilon \in \Ssis (\BCx)$ satisfying the following sequence of identities,
\begin{equation}\label{5eq: Miller-Schmid, C}
\EM _{-m} \Upsilon (2s ) = G_{ m + k } (s) \EM _{-m} \upsilon (2( s + \mu) ), \hskip 10 pt m \in \BZ.
\end{equation}
We  write $ \Upsilon = \ET_{(\mu, k)} \upsilon $ and call $\ET_{(\mu, k)}$ the Miller-Schmid transform  over $\BC$ of index $(\mu, k)$.

{\rm (2).}
Let $\lambda \in \BC$. If $\Re \lambda < 2\, \Re \mu$,  then for any $\upsilon (z) \in [z]^{m } |z|^{-\lambda} (\log |z|)^j \SS (\BC)$ we have
\begin{equation}\label{5eq: Miller-Schmid and Fourier, C}
\begin{split}
\ET_{(\mu, k)} \upsilon (z)
& = [z]^{k} \int_{\BCx} [u]^{k} \|u\|^{-\mu} \upsilon \lp u\-\rp  e (z u + \overline {z u}) d^\times u  = [z]^{k} \EF \varphi \, (- z),
\end{split}
\end{equation}
with $\varphi (z ) = [ z ]^k \| z \|^{-\mu-1}    \upsilon \lp z\-\rp$.

{\rm (3).} When $\Re \lambda < 2 \Re \mu$, the integral in \eqref{5eq: Miller-Schmid and Fourier, C} is absolutely convergent for any  $\upsilon (z) \in [z]^{m } |z|^{-\lambda} (\log |z|)^j \SS (\BC)$.

{\rm (4).}
Suppose that $\Re \mu > 0$. Define the function space 
\begin{equation*}
\Ssiss (\BCx) =  \sum_{m \in \BZ} \ \sum_{\Re \lambda \leq 0 } \ \sum_{j \in \BN}  [z]^{m} |z|^{-\lambda} (\log |z|)^j \SS (\BC).
\end{equation*}
Then the transform $\ET_{(\mu, k)}$ sends $\Ssiss  (\BCx)$ into itself. Moreover, \eqref{5eq: Miller-Schmid and Fourier, R} also holds true for any $\upsilon \in \Ssiss  (\BCx)$, wherein the integral absolutely converges.
\end{lem}

\begin{proof}
Following literally the same ideas in the proof of Lemma \ref{5lem: Miller-Schmid transforms, R}, one may show this lemma without any difficulty. We only remark that, via the isomorphism $\EM_{\BC}$ in Lemma \ref{2lem: Ssis to Msis, C}, $\Ssiss  (\BCx)$ corresponds to the subspace of $\Msis^{\BC}$ consisting of sequences $\lpp H_m\rpp $ such that the poles of each $H_{m}$ lie in the left half-plane $\big\{ s : \Re s \leq \min \{ M - |m|, 0 \} \big\}$ for some $M \in \BN$ (see Lemma \ref{2lem: Ssis m, M -m}).
\end{proof}

\begin{lem}\label{5lem: T, C}
Let $(\mu, k)  \in \BC \times \BZ$ be such that $\Re \mu > 0$. For $m \in \BZ $ define $\Ssiss^{  m} (\BCx)$ to be the space of functions in $ \Ssiss  (\BCx)$  satisfying the condition \eqref{1eq: m condition, C}. For  $\upsilon \in \Ssiss^{   m} (\BCx)$, we write $\upsilon (z) =  [z]^{m }  \varphi (|z|) $. Then
\begin{align*}
\ET_{(\mu, k)} \upsilon \lp x e^{i \phi} \rp & = 2 e^{ i m \phi} \int_{\BR_+} y^{- 2 \mu} \varphi \lp y\- \rp j_{(0, m + k)} (x y)  d^\times y \\
& =  4 \pi i^{m+k} e^{ i m \phi} \int_{\BR_+}  y^{- 2 \mu} \varphi \lp y\- \rp J_{m+k} (4 \pi x y)    d^\times y.
\end{align*}
The transform $\ET_{(\mu, k)}$ sends $\Ssiss^{  m} (\BCx)$ into itself.
\end{lem}

\subsection{Fourier type integral transforms}

In the following, we shall derive the Fourier type integral transform expressions for $\Hld$ and $\Hmum$ from the Fourier transform (more precisely, the renormalized rank-one Hankel transforms) and the Miller-Schmid transforms.

\subsubsection{The Fourier type transform expression for $\Hld$}

Let $(\umu, \udelta) \in \BC^n\times (\BZT)^n$. Following \cite[(6.51)]{Miller-Schmid-2006}, for $\upsilon (x) \in \sgn (x)^{\delta_n} |x|^{\mu_n} \SS (\BR) $, we consider
\begin{equation}\label{5eq: composite T S, R}
\Upsilon (x) = |x|^{-\mu_1} \ET_{(\mu_1-\mu_2, \delta_1)}\, {\sstyle \circ} \, ...\, {\sstyle \circ}\, \ET_{(\mu_{n-1}-\mu_{n}, \delta_{n-1})}\, {\sstyle \circ}\, \ES_{(\mu_n, \delta_n)} \upsilon (x).
\end{equation}
According to Lemma \ref{5cor: renormalize Hankel} (1) and Lemma \ref{5lem: Miller-Schmid transforms, R} (1), $\ES_{(\mu_n, \delta_n)} \upsilon (x)$ lies in the space $\sgn(x)^{\delta_n} \SS (\BR) $ ($\subset \Ssis(\BRx)$), whereas each Miller-Schmid transform sends $\Ssis(\BRx)$ into itself.
Thus, one can apply the Mellin transform $\EM_\delta$ to both sides of \eqref{5eq: composite T S, R}. Using \eqref{5eq: Mellin S, R} and \eqref{5eq: Miller-Schmid, R}, some calculations show that the application of $\EM_\delta$ converts \eqref{5eq: composite T S, R} exactly  into the identities in \eqref{3eq: Hankel transform identity, R} which defines $\Hld$. Therefore, $\Upsilon = \Hld \upsilon$.

\begin{thm}\label{5thm: Fourier type transform, R}
\cite[(1.3)]{Miller-Schmid-2009}. 
Let $(\umu, \udelta) \in \BC^n\times (\BZT)^n$ be such that $\Re \mu_1  > ... > \Re \mu_{n-1} > \Re \mu_n$. Suppose  $\upsilon (x) \in \sgn (x)^{\delta_n} |x|^{\mu_n} \SS (\BR) $. Then
\begin{equation} \label{5eq: Fourier type integral, R}
\Hld \upsilon (x) = \frac 1 {|x|} \int_{\BR^{\times n}} \upsilon \lp \frac {x_1 ... x_n} x \rp \lp \prod_{l      = 1}^{n }  \sgn (x_l     )^{ \delta_{l     }} |x_{l     }|^{ - \mu_{l     }} e \lp x_{l     } \rp \rp dx_n ... d x_1,
\end{equation}
where the integral converges when performed as iterated integral in the indicated order $d x_n d x_{n-1} ... d x_1$, starting from $d x_n$, then $d x_{n-1}$, ..., and finally $d x_1$.
\end{thm}

\begin{proof}
We first observe that $\ES_{(\mu_n, \delta_n)} \upsilon (x) \in \sgn(x)^{\delta_n} \SS (\BR) \subset \Ssiss (\BRx)$. For each $ l      = 1 ,..., n-1$, since $\Re (\mu_{l     } -\mu_{l     +1}) > 0$, Lemma \ref{5lem: Miller-Schmid transforms, R} (4) implies that the transform $\ET_{(\mu_{l     } -\mu_{l     +1}, \delta_{l     })}$ sends the space $\Ssiss (\BRx)$ into itself. According to Lemma \ref{5cor: renormalize Hankel} (1) and Lemma \ref{5lem: Miller-Schmid transforms, R} (3), $\ES_{(\mu_n, \delta_n)} $ and all the $\ET_{(\mu_{l     } -\mu_{l     +1}, \delta_{l     })}$ in \eqref{5eq: composite T S, R}  may be expressed as integral transforms, which are absolutely convergent. From these, the right hand side of \eqref{5eq: composite T S, R} turns into the integral,
\begin{equation*}
\begin{split}
\int_{\BR^{\times n}}  \sgn (x)^{\delta_1} |x|^{-\mu_1} e \lp x y_1 \rp 
\lp \prod_{l      = 1}^{n-1}  \sgn (y_l     )^{\delta_{l      + 1} + \delta_{l     }} |y_{l     }|^{\mu_{l      + 1} - \mu_{l     } - 1} e \lp y_{l     }\- y_{l     +1} \rp \rp & \\
 \sgn (y_n)^{\delta_n} |y_n|^{-\mu_n} \upsilon (y_n) & d y_{n}  ...  d y_{1},
\end{split}
\end{equation*}
which converges as iterated integral.
Our proof is completed upon making the change of variables $x_1 = x y_1$, $x_{l      + 1} = y_l     \- y_{l      + 1}$, $ l      = 1, ..., n-1$.
\end{proof}

We have the following corollary to Theorem \ref{5thm: Fourier type transform, R}, which can also be seen from Lemma \ref{5lem: S, R and C} (1) and Lemma \ref{5lem: T, R}.

\begin{cor}\label{5cor: H, R}
Let $(\umu, \udelta) \in \BC^n\times (\BZT)^n$ and  $\delta \in \BZT$. Assume that $\Re \mu_1 > ... > \Re \mu_{n-1} > \Re \mu_n$. Let $\varphi (x) \in x^{  \mu_n} \SS_{ \delta + \delta_n} (\overline \BR_+)$ and $\upsilon (x) =  \sgn (x)^{ \delta }  \varphi (|x|) $. Then
\begin{equation} \label{5eq: Fourier type integral, R, 1}
\Hmum \upsilon \lp \pm x  \rp = \frac { (\pm)^{\delta}} {x} \int_{\BR _+ ^n} \varphi \lp \frac {x_1 ... x_n} x \rp \lp \prod_{l      = 1}^{n }   x_{l     } ^{ -  \mu_{l     } } j_{(0, \delta_{l     } + \delta)} (x_{l     }) \rp d x_n ... d x_1,
\end{equation}
with $x \in \BR _+$. Here the iterated integration is performed in the  indicated order.
\end{cor}

\subsubsection{The Fourier type transform expression for $\Hmum$}

Let $(\umu, \um) \in \BC^n \times \BZ^n$. Using Lemma \ref{5cor: renormalize Hankel} (2) and Lemma \ref{5lem: Miller-Schmid transforms, C} (1), especially \eqref{5eq: Mellin S, C} and \eqref{5eq: Miller-Schmid, C}, one may show that
\begin{equation}\label{5eq: composite T S, C}
\Hmum \upsilon (z) = \|z\|^{-\mu_1} \ET_{(\mu_1-\mu_2, m_1)} \, {\sstyle \circ} \,  ... \, {\sstyle \circ} \,  \ET_{(\mu_{n-1}-\mu_{n}, m_{n-1})} \, {\sstyle \circ} \,  \ES_{(\mu_n, m_n)} \upsilon (z).
\end{equation}

\begin{thm}\label{5thm: Fourier type transform, C}
Let $(\umu, \um) \in \BC^n\times \BZ^n$ be such that $\Re \mu_1 > ... > \Re \mu_{n-1} > \Re \mu_n$. Suppose  $\upsilon (z) \in [z]^{m_n} \|z\|^{\mu_n} \SS (\BC) $. Then
\begin{equation} \label{5eq: Fourier type integral, C}
\Hmum \upsilon (z) = \frac 1 {\|z\|} \int_{\BC^{\times n}} \lp \prod_{l      = 1}^{n }  [z_l     ]^{- m_{l     }} \|z_{l     }\|^{ - \mu_{l     }} e \lp z_{l     } + \overline {z_l     } \rp \rp \upsilon \lp \frac {z_1 ... z_n} z \rp dz_n ... d z_1,
\end{equation}
where the integral converges when performed as iterated integral in the indicated order. 
\end{thm}
\begin{proof}
One applies the same arguments in the proof of Theorem \ref{5thm: Fourier type transform, R} using Lemma \ref{5cor: renormalize Hankel} (2) and Lemma \ref{5lem: Miller-Schmid transforms, C} (3, 4). 
\end{proof}

Lemma \ref{5lem: S, R and C} (2) and Lemma \ref{5lem: T, C} yield the following corollary.

\begin{cor}\label{5cor: H, C}
Let $(\umu, \um) \in \BC^n\times \BZ^n$ and  $m \in \BZ$. Assume that $\Re \mu_1 > ... > \Re \mu_{n-1} > \Re \mu_n$. Let  $\varphi (x) \in x^{2 \mu_n} \SS_{-m-m_n} (\overline \BR_+)$ and $\upsilon (z) =  [z]^{- m }  \varphi (|z|) $. Then
\begin{equation} \label{5eq: Fourier type integral, C, 1}
\Hmum \upsilon \lp x e^{i \phi} \rp = 2^n \frac  {e^{i m \phi}} {x^{2}} \int_{\BR _+ ^n} \varphi \lp \frac {x_1 ... x_n} x \rp \lp \prod_{l      = 1}^{n }   x_{l     } ^{ - 2 \mu_{l     } + 1} j_{(0, m_{l     } + m)} (x_{l     }) \rp d x_n ... d x_1,
\end{equation}
with $x \in \BR _+$ and $\phi \in \BR/ 2\pi \BZ$. Here the iterated integration is performed in the indicated order.
\end{cor}

\section{Integral representations of Bessel kernels}\label{sec: integral representations}

In the previous article  \cite{Qi}, when  $n \geq 2$, the formal integral representation of the Bessel function $J(x; \usigma, \umu)$ is obtained in symbolic manner from the Fourier type integral in Theorem \ref{5thm: Fourier type transform, R}, where the assumption $\Re\mu_1 > ... > \Re \mu_n$ is simply ignored. It is however more straightforward to derive the formal integral representation of the Bessel kernel $J_{(\umu, \udelta)} (x)$ from  Theorem \ref{5thm: Fourier type transform, R}. This should be  well understood,  since $J_{(\umu, \udelta)} (x)$ is a finite combination of $J\big(2\pi |x|^{\frac 1 n}; \usigma, \umu \big)$.

Similarly, Theorem \ref{5thm: Fourier type transform, C}  also yields a  formal integral representation of $J_{(\umu, \um)} (z)$. It turns out that one can naturally transform this formal integral into an integral that is absolutely convergent, given that the index $\umu$ satisfies certain conditions. The main reason for the absolute convergence is that $j_{(0, m)} (x) = 2 \pi i^{m } J_{m} (4 \pi x)$ (see  \eqref{3eq: j (0, m)}) decays proportionally to  $\frac 1 {\sqrt x}$ at infinity (in comparison, $j_{(0, \delta)} (x)$ is equal to either  $2 \cos (2\pi x)$ or $ 2 i \sin (2\pi x)$). 


\addtocontents{toc}{\protect\setcounter{tocdepth}{1}}

\subsection*{Assumptions and notations} 

Let  $n \geq 2$. Assume that $\umu \in \BL^{n-1}$. 
\begin{notation}\label{not: d, nu}
	Let $d = n-1$. Let the pairs of tuples, $\umu \in \BL^d $ and $\unu \in \BC^d$, $\udelta \in (\BZT)^{d+1}$ and $\boldsymbol \epsilon \in (\BZT)^d$, $\um \in \BZ^{d+1}$  and $\uk \in \BZ^d$, be subjected to the following relations 
	\begin{equation*}
	\nu_{l     } = \mu_{l     } - \mu_{d+1}, \hskip 10 pt \epsilon_{l     } = \delta_{l     } + \delta_{d+1}, \hskip 10 pt  k_{l     } = m_{l     } - m_{d+1},
	\end{equation*}
	for $l      = 1, ..., d.$
\end{notation}

Instead of Hankel transforms, we shall be interested in their Bessel kernels. Therefore, it is convenient to further  assume that the weight functions are Schwartz, namely,  $\varphi \in \SS (\BR_+)$ and $\upsilon \in \SS (\BFx)$. 
According to (\ref{3eq: integral kernel, R 0}, \ref{3eq: h mu m = int of j mu m}), Proposition \ref{3prop: properties of J, R} (3) and \ref{3prop: properties of J, C} (3), for such Schwartz functions $\varphi$ and $\upsilon$, 
\begin{align} 
\label{6eq: Hankel = integral of j} & \hld \varphi  (x ) =  \int_{\BR_+} \varphi (y) j_{(\umu, \udelta)} ( xy ) d y, \hskip 18 pt \hmum \varphi (x)  = 2 \int_{\BR_+}  \varphi (y) j_{(\umu, \um)} ( xy )  y d y, \\
\label{6eq: Hankel = integral of Bessel kernel} &  \Hld \upsilon (x) = \int_{\BR ^\times} \upsilon (y) J_{(\umu, \udelta)} (xy ) d y, \hskip 15 pt \Hmum  \upsilon (z) = \int_{\BC^\times} \upsilon (u) J_{(\umu, \um)} (z u ) d u,
\end{align}
with the index $(\umu, \udelta) \in \BC^n \times (\BZT)^n$ or $(\umu, \um) \in \BC^n \times \BZ^n$ being arbitrary.



\addtocontents{toc}{\protect\setcounter{tocdepth}{2}}

\subsection{\texorpdfstring{The formal integral  $J_{\unu, \boldsymbol \epsilon} ( x,\pm )$}{The formal integral  $J_{\nu, \epsilon} ( x,\pm )$}}\label{sec: formal integral, R}

To motivate the definition of  $J_{\unu, \boldsymbol \epsilon} ( x,\pm )$, we shall do certain operations on the Fourier type integral \eqref{5eq: Fourier type integral, R}  in Theorem \ref{5thm: Fourier type transform, R}. In the meanwhile, we shall forget the assumption $\Re\mu_1 > ... > \Re \mu_n$, which is required for convergence.

Upon making the change of variables, $x_{n} = (x_1 ... x_{n-1})\- x y$, $x_{l     } = |xy|^{\frac 1 {n}} y_{l     }\-$, $ l      = 1, ..., n-1$, one converts \eqref{5eq: Fourier type integral, R}   into
\begin{align*}
\Hld \upsilon (x) = \int_{\BR^{\times n}} \upsilon \lp y \rp &\ \sgn (xy)^{\delta_{n}} 
\lp \prod_{l      = 1}^{n-1 }  \sgn (y_l     )^{ \delta_{l     } + \delta_n} |y_{l     }|^{ \mu_{l     } - \mu_{n} - 1} \rp \\
&  e \lp |xy|^{\frac 1 n} \lp \sgn(x y) \cdot y_1 ... y_{n-1} + \sum_{l      = 1}^{n-1} y_{l     }\-  \rp \rp d y d y_{n-1} ... d y_1.
\end{align*}
In symbolic notation, moving the integral over $d y$ to the outermost place and comparing the resulted integral with the right hand side of the first formula in \eqref{6eq: Hankel = integral of Bessel kernel}, the Bessel kernel $J_{(\umu,\udelta)} (x)$ is then represented by the following formal integral over $d y_{n-1} ... d y_{1}$,
\begin{align*}
\sgn (x)^{\delta_{n}} 
\int_{\BR^{\times \, n-1}}  &  \lp \prod_{l      = 1}^{n-1 }  \sgn (y_l     )^{ \delta_{l     } + \delta_n} |y_{l     }|^{ \mu_{l     } - \mu_{n} - 1} \rp  \\
&\hskip 20 pt e \lp |x|^{\frac 1 n} \lp \sgn(x) \cdot y_1 ... y_{n-1} + \sum_{l      = 1}^{n-1} y_{l     }\-  \rp \rp d y_{n-1} ... d y_1.
\end{align*}
We define the formal integral
\begin{equation}
\begin{split}
J_{\unu, \boldsymbol \epsilon} (x, \pm) = \int_{\BR^{\times d}} 
  \lp \prod_{l      = 1}^{d}  \sgn (y_l     )^{ \epsilon_{l     } } |y_{l     }|^{ \nu_{l     } - 1} \rp e^{ i x \lp \pm  y_1 ... y_{d} + \sum_{l      = 1}^{d} y_{l     }\-  \rp} d y_{d} ... d y_1, \hskip 5 pt x \in \BR_+.
\end{split}
\end{equation}
Thus, in view of Notation \ref{not: d, nu}, we have $ J_{(\umu,\udelta)} (\pm  x  ) = (\pm)^{\delta_{d+1}}  
J_{\unu, \boldsymbol \epsilon} \big( 2 \pi x^{\frac 1 {d+1}}, \pm \big)$ in symbolic notation.

\subsection{\texorpdfstring{The formal integral  $j_{\unu, \udelta} (x)$}{The formal integral  $j_{\nu, \delta} (x)$}} 
For $\unu \in \BC^d$ and $ \udelta \in (\BZT)^{d+1}$, we define the formal integral
\begin{equation}
\begin{split}
j_{\unu, \udelta} (x ) = \int_{ \BR_+ ^{d}} &  j_{(0, \delta_{d+1})} \lp x y_{1} ... y_d \rp \prod_{l      = 1}^{d} y_{l     }^{ \nu_{l     } - 1} j_{(0, \delta_{l     })} \lp x y_{l     } \- \rp   
  d y_{d} ... d y_1, \hskip 10 pt x \in \BR_+.
\end{split}
\end{equation}
We may derive the symbolic identity  $j_{(\umu, \udelta)} (x) =  
j_{\unu, \udelta} \big( x^{\frac 1 {d+1}} \big)$ from Corollary \ref{3cor: H = h, R} and  \ref{5cor: H, R}, combined with the first formula in \eqref{6eq: Hankel = integral of j}.



\subsection{\texorpdfstring{The integral  $J_{\unu, \uk} ( x, u )$}{The integral  $J_{\nu, k} ( x, u )$}}

First of all, proceeding in the same way as in \S \ref{sec: formal integral, R}, from the Fourier type integral \eqref{5eq: Fourier type integral, C} in Theorem \ref{5thm: Fourier type transform, C}, we can deduce the symbolic equality $ J_{(\umu,\um)} \lp  x  e^{i\phi} \rp = e^{- i m_{d+1} \phi} 
J_{\unu, \uk} \big( 2 \pi x^{\frac 1 {d+1}}, e^{i\phi} \big)$, with the definition of the formal integral,
\begin{equation}\label{5eq: formal integral, C, 0}
	\begin{split}
		J_{\unu, \uk} ( x, u ) =    \int_{\BC^{\times d}}
		 \lp \prod_{l      = 1}^{d}  [u_{l     }]^{k_{l     } } \|u_{l     }\|^{ \nu_{l     } - 1} \rp  &  e^{ i x \Lambda \lp u u_1 ... u_{d} + \sum_{l      = 1}^{d} u_{l     }\-  \rp  }  d u_{d} ... d u_1, \\ 
		 & \hskip 60 pt x \in \BR_+,\, u \in \BC, |u| = 1.
	\end{split}
\end{equation}
Here, we recall that $\Lambda (z) = z + \overline z$. 

In the polar coordinates, we write $u_{l     } = y_{l     } e^{i \theta_{l     }}$ and $u = e^{i \phi}$. Moving the integral over the torus $(\BR/ 2 \pi \BZ)^d$ inside, in symbolic manner, the integral above turns into
\begin{equation*}
\begin{split}
2^d \int_{\BR_+^d} \hskip -3 pt \int_{(\BR/ 2 \pi \BZ)^d} \hskip -3 pt
\lp \prod_{l      = 1}^{d} y_{l     }^{ 2 \nu_{l     } - 1} \hskip -2 pt \rp  \hskip -2 pt e^{i \sum_{l      = 1}^d k_{l     } \theta_l    + 2 i  x  \lp y_1 ... y_{d} \cos \lp  \sum_{l      = 1}^d \theta_{l     } + \phi \rp + \sum_{l      = 1}^{d} y_{l     }\- \cos \theta_l        \rp } d \theta_{d} ... d \theta_1 d y_{d} ... d y_1.
\end{split}
\end{equation*}
Let us introduce the following definitions
\begin{align}
\label{6eq: Theta (theta, y; x, phi)}
\Theta_{\uk} (\utheta, \uy; x, \phi) = 2 x  y_1 ... y_{d} & \cos \lp \textstyle \sum_{l      = 1}^d \theta_{l     } + \phi \rp + \sum_{l      = 1}^d \lp k_{l     } \theta_l      + 2 x y_{l     }\- \cos \theta_l      \rp ,\\
	\label{6eq: Jk (y;x, phi)}
	 J_{\uk}(\uy; x, \phi ) & = 
	 \int_{(\BR/ 2 \pi \BZ)^d} e^{i \Theta_{\uk} (\utheta, \uy; x, \phi) } d\utheta, \\
p_{2 \unu} & (\uy) = \prod_{l      = 1}^d   y_{l     }^{ 2 \nu_l      - 1},
\end{align}
with  $\uy = (y_1, ..., y_d)$, $\utheta = (\theta_1, ..., \theta_d)$.
Then \eqref{5eq: formal integral, C, 0} can be symbolically rewritten as
\begin{equation}
\label{6eq: formal integral, C}
J_{\unu, \uk} \lp x, e^{  i \phi} \rp = 2^d \int_{\BR_+^d} p_{2 \unu} (\uy) J_{\uk}(\uy; x, \phi ) d \uy, \hskip 10 pt x\in \BR_+, \, \phi \in \BR/2\pi \BZ.
\end{equation}

\begin{thm}\label{6thm: formal integral, C}
Let  $(\umu, \um) \in \BL^d \times \BZ^{d+1}$ and $(\unu, \uk) \in \BC^d \times \BZ^d$  satisfy the relations given in Notation {\rm \ref{not: d, nu}}. Suppose  $\unu \in \bigcup_{a \in \left[-\frac 1 2, 0 \right]} \left\{ \unu \in \BC^d : - \frac 1 2 < 2 \Re \nu_{l     } + a < 0 \text{ for all } l      = 1,..., d \right\}$. 

{\rm(1).}  The integral in \eqref{6eq: formal integral, C} converges absolutely. Subsequently, we shall therefore use \eqref{6eq: formal integral, C} as the definition of $J_{\unu, \uk} \lp x, e^{  i \phi} \rp$. 

{\rm(2).}  We have the {\rm ({\it genuine})} identity
$$ J_{(\umu,\um)} \lp  x  e^{i\phi} \rp = e^{- i m_{d+1} \phi} 
J_{\unu, \uk} \big( 2 \pi x^{\frac 1 {d+1}}, e^{i\phi} \big).$$
\end{thm}

\subsection{\texorpdfstring{The integral  $j_{\unu, \um} (x)$}{The integral  $j_{\nu, m} (x)$}}
Let us consider the integral $j_{\unu, \um} (x)$ defined by
\begin{equation}\label{6def: integral j nu m}
	\begin{split}
	j_{\unu, \um} (x) =	2^d \int_{\BR_+^{ d}} j_{(0, m_{d+1})} \left(  x  y_1 ... y_{d} \right) 
		  \prod_{l      = 1}^{d} y_l      ^{2\nu_l      - 1} j_{(0, m_l     )} \left( x y_{l     }\- \right)   
		    d y_{d} ... d y_1,
	\end{split}
\end{equation}
with $\unu \in \BC^d $ and $\um \in \BZ^{d+1}$.

\subsubsection{Absolute convergence of  $j_{\unu, \um} (x)$}
In contrast to the real case, where the integral $j_{\unu, \udelta} (x)$ never absolutely converges, $j_{\unu, \um} (x)$ is actually absolutely convergent, if each component of $\unu$ lies in   certain vertical strips of width at least $\frac 1 4$. 

\begin{defn}\label{6def: hyper-strip Sd}
	For $\boldsymbol{a}, \boldsymbol{b} \in \BR^d$ such that $a_{l     } < b_{l     }$ for all $l      = 1, ..., d$, we define the open hyper-strip $\BS^d (\boldsymbol a,  \boldsymbol{b}) = \left\{ \unu \in \BC^d : \Re \nu_{l     } \in (a_{l     }, b_{l     }) \right\}$. We write $ \BS^d (a ,  b ) = \BS^d (a \ue^d,  b \ue^d)$ for simplicity.
\end{defn}

\begin{prop}\label{6prop: convergence of j nu m}
	Let $(\unu, \um) \in \BC^d \times \BZ^{d+1}$. The integral $j_{\unu, \um} (x)$ defined above by \eqref{6def: integral j nu m} absolutely converges if $\unu \in \bigcup_{a \in \left[- \frac 1 2 , |m_{d+1}| \right]} \BS^d \left(\frac 1 2 \left(- \frac 1 2 - a\right) \ue^d, \frac 1 2 \lp \left\| \um^d \right\| - a \ue^d \rp \right)$, with $\um^d = (m_1, ..., m_d)$ and $\left\| \um^d \right\| = (|m_1|, ..., |m_d|)$.
\end{prop}

To show this, we first recollect some well-known facts concerning $J_m (x)$, as $j_{(0, m)} (x) = 2 \pi i^{m } J_{m} (4 \pi x)$ in view of  \eqref{3eq: j (0, m)}. 

Firstly, for $m \in \BN$, we have the Poisson-Lommel integral representation (see \cite[3.3 (1)]{Watson})
\begin{equation}
J_{m} (x) = \frac {\lp \frac 1 2x \rp^m } {\Gamma \left( m + \frac 1 2\right) \Gamma \left(\frac 1 2\right)} \int_0^{\pi} \cos (x \cos \theta) \sin^{2 m} \theta d \theta.
\end{equation}
This yields the bound 
\begin{equation}\label{6eq: bound < x m}
|J_{m} (x) | \leq \frac {\sqrt \pi \lp \frac 1 2x \rp^{|m|} } {\Gamma \left( |m| + \frac 1 2\right) },
\end{equation}
for $m \in \BZ$.
Secondly, the asymptotic expansion of $J_m (x)$ (see \cite[7.21 (1)]{Watson}) 
provides the estimate
\begin{equation}\label{6eq: bound < x -1/2}
J_m (x) \lll_m x^{- \frac 1 2}.
\end{equation}
Combining these, we then arrive at the following lemma. 
\begin{lem}\label{6lem: crude bound for j m}
Let $m$ be an integer.

{\rm(1).} We have 
the estimates 
$$j_{(0, m)} (x) \lll_m x^{|m|}, \hskip 10 pt j_{(0, m)} (x) \lll_m x^{- \frac 1 2}.$$

{\rm (2).} More generally, for any $a \in \left[- \frac 1 2, |m|\right]$, we have the estimate
$$j_{(0, m)} (x) \lll_m x^{a}.$$
\end{lem}

\begin{proof}[Proof of Proposition \ref{6prop: convergence of j nu m}]
	
	We divide $\BR_+ = (0, \infty)$ into the union of two intervals, $ I_- \cup I_+ = (0, 1] \cup [1, \infty)$. Accordingly,  the integral in \eqref{6def: integral j nu m}  is partitioned into $2^d$ many integrals, each of which is supported on some hyper-cube $I_{\urho} = I_{\varrho_1 } \times ... \times I_{\varrho_d }$ for $\urho 
	\in \{+ , -\}^d$. For each such integral, we estimate  $j_{(0, m_l     )} \left(x y_{l     }\-\right)$ using the first or the second estimate in Lemma \ref{6lem: crude bound for j m} (1) according as $\varrho_{l     } = +$ or  $\varrho_{l     } = -$ and apply  the bound in Lemma \ref{6lem: crude bound for j m} (2) for $j_{(0, m_{d+1})} (x y_1 ... y_d) $. In this way, for any $a \in \left[- \frac 1 2, |m_{d+1} | \right]$, one has
	\begin{align*}
	&  2^d \int_{\BR_+^{ d}} 
		\left|j_{(0, m_{d+1})} \left(  x  y_1 ... y_{d} \right) \right| 
	 \prod_{l      = 1}^{d} \left|y_l      ^{2\nu_l      - 1} j_{(0, m_l     )} \left( x y_{l     }\- \right) \right| 
    d y_{d} ... d y_1 
	\\
	\lll &\ \sum_{ \urho \in \{+, -\}^d } x^{ \sum_{ l       \in L_+ (\urho)} |m_l     | - \frac 1 2 \left|L_-(\urho)\right| + a} 
	I_{2 \unu +   a \ue^d, \um^d} ( \urho ),
	\end{align*}
with the auxiliary definition
\begin{equation*} 
 I_{\ulambda, \uk} ( \urho ) = \int_{ I_{\urho} } \lp \prod_{ l       \in L_+ (\urho)}  y_l      ^{ \Re \lambda_l      - |k_{l     } | - 1} \rp \lp  \prod_{ l      \in L_- (\urho)}  y_l      ^{  \Re \lambda_l      - \frac 1 2} \rp d y_{d} ... d y_1, \hskip 5 pt (\ulambda, \uk) \in \BC^d \times \BZ^d,
\end{equation*}
and $L_\pm (\urho) = \left\{ l      : \varrho_{l     } = \pm \right\}$. The implied constant depends only on $\um$ and $ d$. It is clear that all the integrals $I_{2 \unu +   a \ue^d, \um^d} ( \urho )$ absolutely converge if  $ - \frac 1 2   < 2 \, \Re \nu_{l     } + a < |m_l      | $ for all $l      = 1,..., d$. The proof is then completed. 
\end{proof}

\begin{rem}\label{6rem: d=1, j}
When $d = 1$, one may apply the two estimates in Lemma  {\rm \ref{6lem: crude bound for j m} (1)} to $j_{(0, m_{2})} (x y )$ in the similar fashion as $j_{(0, m_1)} \left( x y \- \right)$. Then
	\begin{equation*}
\begin{split}
2 \int_{0 }^{\infty} 
\left|y ^{2\nu - 1} j_{(0, m_1)} \left( x y \- \right)  j_{(0, m_{2})} \left(  x  y \right) \right|  &  d y  \\
\lll_{\, m_1, m_2}  \ x^{|m_1| -\frac 1 2} &
\int_{1}^\infty   y ^{2\Re \nu - |m_1| - \frac 3 2} d y + x^{|m_2| -\frac 1 2} \int_{0}^1  y ^{2 \Re \nu + |m_2| - \frac 1 2} d y .
\end{split}
\end{equation*}
Since both integrals above absolutely converge if $\, - |m_2| - \frac 1 2 < 2 \Re \nu <  |m_1| + \frac 1 2$, this also proves Proposition {\rm \ref{6prop: convergence of j nu m}} in the case $d=1$.
\end{rem}

\subsubsection{Equality between $j_{(\umu, \um)} (x) $ and $ j_{\unu, \um} \big (x^{\frac 1 {d+1}} \big)$}

\begin{prop}\label{6prop: j mu = j nu}
Let $(\unu, \um) \in \BC^d \times \BZ^{d+1}$ be as in Proposition {\rm \ref{6prop: convergence of j nu m}} so that the integral $j_{\unu, \um} (x)$ absolutely converges. Suppose that $\umu \in \BL^{d}$ and $\unu \in \BC^d$ satisfy the relations given in Notation {\rm \ref{not: d, nu}}.  Then we have the identity
$$j_{(\umu, \um)} (x) = j_{\unu, \um} \big (x^{\frac 1 {d+1}} \big).$$
\end{prop}

\begin{proof}
Some change of variables turns the integral in Corollary \ref{5cor: H, C} into 
\begin{equation*}
	2^{d+1} e^{i m \phi} \int_{\BR_+^{ d+1}} \varphi (y)  \,  
	 j_{(0, m_{d+1})} \big(  (xy)^{\frac 1 {d+1}} y_1 ... y_{d} \big)
	  \prod_{l      = 1}^{d} y_l      ^{2\nu_l        - 1} j_{(0, m_l     )} \big( (xy)^{\frac 1 {d+1}} y_{l     }\- \big)   
	   y  d y d y_{d} ... d y_1. 
\end{equation*}
Corollary   \ref{3cor: H = h, C} and \ref{5cor: H, C}, along with the second formula in \eqref{6eq: Hankel = integral of j}, yield 
\begin{equation*}
	\begin{split} 
	&2 \int_{\BR_+} \varphi (y) j_{(\umu, \um)} ( xy )   y d y = \\
	& 2^{d+1} \int_{\BR_+^{ d+1}} \varphi (y)  
	 j_{(0, m_{d+1})} \big(  (xy)^{\frac 1 {d+1}}  y_1 ... y_{d} \big)
	  \prod_{l      = 1}^{d} y_l      ^{2\nu_l        - 1} j_{(0, m_l     )} \big( (xy)^{\frac 1 {d+1}} y_{l     }\- \big)   
	   y d y d y_{d} ... d y_1,
	 \end{split}
\end{equation*}
for any $\varphi \in \SS (\BR_+)$, provided that $\Re \mu_1 > ... > \Re \mu_{d+1} $ or equivalently $ \Re \nu_1  > ... > \Re \nu_d > 0$. In view of Proposition {\rm \ref{6prop: convergence of j nu m}}, the integral on the right hand side is absolute convergent at least when $\frac 1 4 > \Re \nu_1  > ... > \Re \nu_d > 0$. Therefore, the asserted equality holds on the domain $\left\{ \unu\in \BC^d : \frac 1 4 > \Re \nu_1  > ... > \Re \nu_d > 0 \right\}$ and remains valid on the whole domain of convergence for $j_{\unu, \um} (x) $ given in Proposition {\rm \ref{6prop: convergence of j nu m}} due to the principle of analytic continuation.
\end{proof}

\subsubsection{An auxiliary lemma}

\begin{lem}\label{6lem: bound for j nu m + m e}
	Let $(\unu, \um) \in \BC^d \times \BZ^{d+1}$ and $m \in \BZ$. Set $A = \max_{l     =1,..., d+1} \left\{  |m_{l     } | \right\}$.
Suppose $\unu \in \bigcup_{a \in \left[-\frac 1 2, 0 \right]} \BS^d \left(- \frac 1 4 -\frac 1 2 a, - \frac 1 2 a \right)$. We have the estimate 
\begin{align*}
& \hskip 27 pt 2^d \int_{\BR_+^{ d}} 
\left|j_{(0, m_{d+1} + m)} \left(  x  y_1 ... y_{d} \right) \right| 
\prod_{l      = 1}^{d} \left|y_l      ^{2\nu_l      - 1} j_{(0, m_l      + m)} \left( x y_{l     }\- \right) \right|   
d y_{d} ... d y_1  \\
& \lll_{\, \um,\, d}   \sum_{  {\scriptstyle \urho \,\neq\, \urho_- } } \lp \frac {2 \pi e x} {|m| + 1} \rp^{\left| L _+ (\urho)\right|  |m|} (|m| + 1)^{2 \left|L_-(\urho)\right| + A  \left|L_+ (\urho)\right| } \\
& \hskip  126 pt x^{ - \frac 1 2 \left|L_-(\urho)\right|} \max \left\{ x^{\left|L_+ (\urho)\right| A}, x^{- \left|L_+ (\urho)\right| A - \frac 1 2} \right\} \\
& \hskip 27 pt + \lp \frac {2 \pi e x} {|m| + 1} \rp^{ |m|} (|m| + 1)^{ A } x^{ - \frac d 2} \max \left\{ x^{ A}, x^{-A} \right\},
\end{align*}
	where $\urho \in \{+, -\}^d$, $\urho_- = (-, ..., -)$ and $L_\pm (\urho) = \left\{ l      : \varrho_{l     } = \pm \right\}$.  
\end{lem}

Firstly, we require the bound \eqref{6eq: bound < x m}  for $J_m (x)$.
Secondly, we observe that when $x \geq (|m| + 1)^2$ the bound \eqref{6eq: bound < x -1/2} for $J_m (x)$ can be improved  so that the implied constant becomes absolute. This follows from the asymptotic expansion of $J_m (x)$ given in \cite[\S 7.13.1]{Olver}.
Moreover, we have Bessel's integral representation (see \cite[2.2 (1)]{Watson})
\begin{equation}\label{6eq: Bessel integral}
J_m (x) = \frac 1 {2 \pi} \int_0^{2 \pi} \cos \lp m \theta - x \sin \theta \rp d \theta,
\end{equation}
which yields the bound
\begin{equation}\label{6eq: bound < 1}
|J_m (x)| \leq 1.
\end{equation}
We then have the following lemma (compare \cite[Proposition 8]{Harcos-Michel}).

\begin{lem}\label{6lem: crude bound for j m, with m}
Let $m$ be an integer.

{\rm (1).} The following two estimates hold
$$j_{(0, m)} (x) \lll \frac { \lp  2 \pi x \rp^{|m|} } {\Gamma \left( |m| + \frac 1 2\right) }, \hskip 10 pt j_{(0, m)} (x) \lll \frac {|m| + 1} {\sqrt x}, $$
with absolute implied constants.

{\rm (2).} For any $a \in \left[- \frac 1 2, 0 \right]$ we have the estimate
$$j_{(0, m)} (x) \lll   \lp {(|m|+1)^{- 2 } } {x} \rp^{ a},$$
with absolute implied constant.
\end{lem}

\begin{proof}[Proof of Lemma \ref{6lem: bound for j nu m + m e}] Our proof here is similar to that of Proposition  \ref{6prop: convergence of j nu m}, except that
\begin{itemize}
\item [-]   Lemma \ref{6lem: crude bound for j m, with m} (1) and (2) are applied in place of Lemma  \ref{6lem: crude bound for j m}  (1) and (2) respectively to bound $j_{(0, m_l      + m)} \left( x y_{l     }\- \right) $ and $j_{(0, m_{d+1} + m)} \left(  x  y_1 ... y_{d} \right) $, and
\item [-]      the first estimate in Lemma \ref{6lem: crude bound for j m, with m} (1) is used for $j_{(0, m_{d+1} + m)} \left(  x  y_1 ... y_{d} \right)$  in the case $\urho = \urho_- $. 
\end{itemize}
In this way, one obtains the following estimate 
	\begin{align*}
	&  \ 2^d \int_{\BR_+^{ d}} 
	\left|j_{(0, m_{d+1} + m)} \left(  x  y_1 ... y_{d} \right) \right|    
	\prod_{l      = 1}^{d} \left|y_l      ^{2\nu_l      - 1} j_{(0, m_l      + m)} \left( x y_{l     }\- \right) \right|
	d y_{d} ... d y_1 
	\\
	\lll  \ & \sum_{ {\scriptstyle \urho \neq \urho_- } } \frac {\prod_{l       \in L_- (\urho)}  (|m_l      + m| + 1)^{ 1 - 2 a } } {\prod_{l       \in L_+ (\urho)} \Gamma \left( |m_l      + m| + \frac 1 2\right) (|m_l      + m| + 1)^{2a} } \\
	& \hskip 77 pt 
	 (2 \pi x)^{  \sum_{ l       \in L_+ (\urho)} |m_l      + m | - \frac 1 2 \left|L_-(\urho)\right| + a } I_{2 \unu +   a \ue^d, \um^{d } + m \ue^{d }} (\urho) \\
	& + \frac {\prod_{l     =1}^d  (|m_{l     } + m| + 1) } { \Gamma \left( |m_{d+1} + m| + \frac 1 2\right) } 
	 (2 \pi x)^{ - \frac {d} 2 +  |m_{d+1} + m | } I_{2 \unu +   |m_{d+1} + m| \ue^d } (\urho_-),
	\end{align*}
    \footnote{When $\urho = \urho_-$, $\uk$ does not occur in the definition of $I_{\ulambda, \uk} ( \urho_- )$ and   is therefore suppressed from the subscript.}with $a \in \left[- \frac 1 2, 0\right] $. Now the implied constant above depends only on $d$.
	Suppose that $- \frac 1 2 - a < 2 \, \Re \nu_{l     } < - a $ for all $l      = 1, ..., d$, then   the integrals $I_{2 \unu +  a \ue^d, \um^{d } + m \ue^{d }} (\urho)$ and $I_{2 \unu +   |m_{d+1} + m| \ue^d } (\urho_-)$   are absolutely convergent and   of size $O_d \lp \prod_{l      \in L_+ (\urho)}( |m_{l     } + m| +1 )\- \rp$ and $O_d \lp  (|m_{d+1} + m| +1 )^{- d} \rp$ respectively. A final estimation using  Stirling's asymptotic formula yields our asserted bound.
\end{proof}

\begin{rem}\label{6rem: d=1, j m + m e}
	In the case $d=1$, modifying over the ideas in Remark {\rm \ref{6rem: d=1, j}}, one may show the  slightly improved estimate
	\begin{align*}
	2 \int_{0 }^{\infty} 
	\left|y ^{2\nu - 1} j_{(0, m_1 + m)} \left( x y \- \right) j_{(0, m_{2} + m)} \left(  x  y \right) \right| &    d y   \\
	\lll_{m_1, m_2} \lp \frac {2 \pi e x} {|m| + 1} \rp^{ |m|} & (|m| + 1)^{ A }  x^{ - \frac 1 2} \max \left\{ x^{A}, x^{- A} \right\},
	\end{align*}
	given that $|\Re \nu | < \frac 1 4$, with $A = \max \left\{ |m_1|, |m_2| \right\}$.
\end{rem}

\subsection{\texorpdfstring{The series of integrals $J_{\unu, \um} (x, u)$}{The series of integrals $J_{\nu, m} (x, u)$}}
We define the following series of integrals,
\begin{equation}\label{6def: J nu m (x, u)}
\begin{split}
& J_{\unu, \um} (x, u) = \frac 1 {2 \pi} \sum_{m \in \BZ}  u^m  j_{ \unu, \um + m\ue^n } (x)   \\
= \ & \frac {2^{d-1}} { \pi} \sum_{m \in \BZ}  u^m \int_{\BR_+^{ d}}  
j_{(0, m_{d+1} + m)} \left(  x  y_1 ... y_{d} \right)    
\prod_{l      = 1}^{d} y_l      ^{2\nu_l      - 1} j_{(0, m_l      + m)} \left( x y_{l     }\- \right)
 d y_{d} ... d y_1,
\end{split}
\end{equation}
with $x \in \BR_+$ and $u \in \BC$, $|u| = 1$.

\subsubsection{Absolute convergence of $J_{\unu, \um} (x, u)$}

We have the following  direct consequence of Lemma \ref{6lem: bound for j nu m + m e}.

\begin{prop}\label{6prop: J nu m (x, u)}
	Let $(\unu, \um) \in \BC^d \times \BZ^{d+1}$. The series of integrals $J_{\unu, \um} (x, u)$ defined by \eqref{6def: J nu m (x, u)} is absolutely convergent if  $\unu \in \bigcup_{a \in \left[-\frac 1 2, 0 \right]} \BS^d \left(- \frac 1 4 -\frac 1 2 a, - \frac 1 2 a \right)$. 
\end{prop}

\subsubsection{Equality between $J_{(\umu, \um)} \lp x e^{i \phi} \rp $ and $J_{\unu, \um} \big(x^{\frac 1 {d+1}}, e^{i \phi} \big)$}

In view of Proposition \ref{6prop: j mu = j nu} along with \eqref{2eq: Bessel kernel over C, polar} and \eqref{6def: J nu m (x, u)}, the following proposition is readily established.

\begin{prop}\label{6prop: J mu m = J nu m}
	Let $(\unu, \um) \in \BC^d \times \BZ^{d+1}$. Suppose that $ \unu $ satisfies the condition in Proposition {\rm \ref{6prop: J nu m (x, u)}} so that $J_{\unu, \um} (x, u)$ is absolutely convergent. Then, given that $\umu $ and $\unu$ satisfy the relations   in Notation {\rm \ref{not: d, nu}}, we have the identity
	$$J_{(\umu, \um)} \lp x e^{i \phi} \rp = J_{\unu, \um} \big(x^{\frac 1 {d+1}}, e^{i \phi} \big), $$
	with $x \in \BR_+$ and $\phi \in \BR/2\pi \BZ$.
\end{prop}

\subsection{Proof of Theorem \ref{6thm: formal integral, C}}

\begin{lem}\label{6lem: series expansion of Jk}
Let $ \uk \in \BZ^d $ and recall the integral $J_{\uk}(\uy; x, \phi )$ defined by {\rm (\ref{6eq: Theta (theta, y; x, phi)}, \ref{6eq: Jk (y;x, phi)})}. We have  the following absolutely convergent  series expansion of $J_{\uk}(\uy; x, \phi )$
\begin{equation} \label{6eq: series expansion of J k}
J_{\uk}(\uy; 2 \pi x, \phi ) =  \frac {1} { 2 \pi} \sum_{m \in \BZ}  e^{i m\phi} 
j_{(0,  m)} \left(  x  y_1 ... y_{d} \right) \prod_{l      = 1}^{d}  j_{(0, k_l      + m)} \left( x y_{l     }\- \right).
\end{equation}
\end{lem}

\begin{proof}
In view of Example \ref{3ex: Bessel n=1, C}, we have the integral representation 
$$j_{(0, m)} (x) = \int_{\BR/2\pi\BZ} e^{i m \theta + 4\pi i x \cos \theta } d \theta $$
as well as the Fourier series expansion
$$e^{4 \pi i x \cos \phi} = \frac 1 {2 \pi} \sum_{m \in \BZ} j_{(0, m)} (x) e^{i m \phi}.$$
Therefore
\begin{align*}
& \frac {1} { 2 \pi} \sum_{m \in \BZ}  e^{i m\phi} 
j_{(0,  m)} \left(  x  y_1 ... y_{d} \right) \prod_{l      = 1}^{d}  j_{(0, k_l      + m)} \left( x y_{l     }\- \right)\\
= \, & \frac {1} { 2 \pi} \sum_{m \in \BZ}  e^{i m\phi} j_{(0,  m)} \left(  x  y_1 ... y_{d} \right)
\int_{(\BR/2\pi\BZ)^d} e^{i m \sum_{l     =1}^d \theta_{l     } } e^{i \sum_{l      = 1}^d \lp i k_{l     } \theta + 4 \pi i x y_{l     }\- \cos \theta_{l     } \rp  }  d \theta_d ... d \theta_1 \\
= \ & \int_{(\BR/2\pi\BZ)^d} \lp \frac {1} { 2 \pi} \sum_{m \in \BZ}  e^{i m \lp \sum_{l     =1}^d \theta_{l     } + \phi \rp} j_{(0,  m)} \left(  x  y_1 ... y_{d} \right) \rp e^{i \sum_{l      = 1}^d \lp i k_{l     } \theta + 4 \pi i x y_{l     }\- \cos \theta_{l     } \rp  }  d \theta_d ... d \theta_1\\
= \ & \int_{(\BR/2\pi\BZ)^d} e^{4 \pi i x  y_1 ... y_{d} \cos \lp \sum_{l     =1}^d \theta_{l     } + \phi \rp } e^{i \sum_{l      = 1}^d \lp i k_{l     } \theta + 4 \pi i x y_{l     }\- \cos \theta_{l     } \rp  }  d \theta_d ... d \theta_1.
\end{align*}
The absolute convergence required for the validity of each equality above is justified by the first estimate of $j_{(0, m)} (x)$ in Lemma \ref{6lem: crude bound for j m, with m} (1). The proof is completed, since the last line is exactly the definition of  $J_{\uk}(\uy; 2 \pi x, \phi )$.
\end{proof}

Inserting the series expansion of $J_{\uk}(\uy; 2 \pi x, \phi ) $ in Lemma \ref{6lem: series expansion of Jk} into the integral in \eqref{6eq: formal integral, C} and interchanging the order of integration and summation, one arrives exactly at the series of integrals $J_{\unu, (\uk, 0)} \lp x, e^{i \phi} \rp = e^{ - i m_{d+1} \phi}  J_{\unu, \um} \lp x, e^{i \phi} \rp$.  The first assertion on   absolute convergence in Theorem \ref{6thm: formal integral, C}   follows immediately from Proposition \ref{6prop: J nu m (x, u)}, whereas the  identity in the second assertion is a direct consequence of Proposition \ref{6prop: J mu m = J nu m}.

\subsection{The rank-two case ($d=1$)}

\subsubsection{The real case}
The formal integral representation $J_{\nu, \epsilon} ( 2 \pi \sqrt x, \pm )$ of the Bessel kernel $J_{\left(\frac 1 2 \nu, -\frac 1 2 \nu \right), (\epsilon, 0)} (\pm x)$ is reduced to the following integral representations of classical Bessel functions
\begin{align*}
\pm \pi i e^{\pm \frac 1 2 \pi i \nu } H^{(1, 2)}_\nu (2 x)  = \int_0^\infty y^{\nu - 1} e^{\pm  i x (y + y\-)} d y, \hskip 5 pt 
2 e^{\pm \frac 1 2 \pi i \nu } K_{\nu} (2 x)  = \int_0^\infty y^{\nu - 1} e^{\pm i x (y - y\-)} d y,
\end{align*}
which are only (conditionally) convergent when $|\Re \nu| < 1$ (see \cite[\S 2.3.2]{Qi}).

\subsubsection{The complex case}

\begin{lem}\label{6lem: d = 1, Jk}
	Let  $k \in \BZ$. Recall from {\rm (\ref{6eq: Theta (theta, y; x, phi)}, \ref{6eq: Jk (y;x, phi)})} the definition $$J_k (y; x, \phi) = \int_0^{2 \pi}  e^{i k \theta + 2 i x y\- \cos \theta + 2 i x y \cos (\theta + \phi)} d \theta, \hskip 10 pt x, y \in (0, \infty), \, \phi \in [0, 2 \pi) .$$
	Define $Y(y, \phi) = \left|y\-  +  y e^{i \phi}\right| = \sqrt {y^{-2} + 2 \cos \phi + y^2 } $, $\Phi (y, \phi) = \arg (y\-  +  y e^{i \phi})$ and $E(y, \phi) = e^{i \Phi(y, \phi)}$. 
	Then  
	\begin{equation}\label{6eq: d=1, Jk}
	J_k (y; x, \phi) = 2 \pi i^k E(y, \phi)^{-k} J_k \lp 2 x Y(y, \phi) \rp .
	\end{equation}
\end{lem}

\begin{proof}
	\eqref{6eq: d=1, Jk} follows immediately from the identity
	$$2 \pi i^k J_k ( x) = \int_{0}^{2 \pi} e^{i k \theta + i x \cos \theta } d \theta, $$
	along with the observation $$y\- \cos \theta + y \cos (\theta + \phi) = \Re \lp y\- e^{i\theta} + y e^{i (\theta + \phi)} \rp = Y(y, \phi) \cos \lp \theta + \Phi (y, \phi) \rp.$$
\end{proof}

\begin{prop}\label{6prop: d=1, J v k}
	Let $\nu \in \BC$ and $k \in \BZ$. Recall the definition of $J_{\nu, k} \lp x, e^{  i \phi} \rp$ given by \eqref{6eq: formal integral, C}. Then 
	\begin{equation}\label{6eq: d = 1 J nu k (x, e i phi)}
	J_{\nu, k} \lp x, e^{  i \phi} \rp = 4 \pi i^k \int_0^\infty y^{2 \nu - 1} \left[ y\-  +  y e^{i \phi} \right]^{-k} J_k \lp 2 x \left|y\-  +  y e^{i \phi}\right| \rp d y,
	\end{equation}
	with $x \in (0, \infty)$ and $\phi \in [0, 2 \pi)$. Here, we recall the notation $[z] =   {z} / {|z|}$. The integral in \eqref{6eq: d = 1 J nu k (x, e i phi)} converges when $|\Re \nu| < \frac 3 4$ and the  convergence is absolute if and only if $|\Re \nu| < \frac 1 4$. Moreover, it is analytic with respect to $\nu $ on the open vertical strip $\BS\lp - \frac 3 4, \frac 3 4 \rp$.
\end{prop}

\begin{proof}
	 \eqref{6eq: d = 1 J nu k (x, e i phi)} follows immediately from Lemma \ref{6lem: d = 1, Jk}. 
	 
	 As for the convergence, since one arrives at an integral of the same form with $\nu, \phi$ replaced by $- \nu, - \phi$ if    the variable is changed from $y$ to $y\-$, it suffices to consider the integral 
	 $$\int_2^\infty y^{2 \nu - 1} e^{-i k \Phi (y, \phi)} J_k \lp 2 x Y(y, \phi) \rp d y,$$ 
	 for $\Re \nu < \frac 3 4$. 
	 We have the following  asymptotic of $J_k (x)$ (see \cite[7.21 (1)]{Watson})
	 $$ J_k (x) = \lp \frac  2 {\pi x } \rp^{\frac 1 2} \cos \lp x - \tfrac 12 k \pi - \tfrac 1 4 \pi \rp + O_k \big( {x^{- \frac 3 2} } \big).$$
	 The error term contributes an absolutely convergent integral when $\Re \nu < \frac 3 4$, whereas the integral coming from the main term absolutely converges if and only if  $\Re \nu < \frac 1 4$. We are now reduced to the integral
	 \begin{align*}
	 \int_2^\infty  y^{2 \nu - 1} e^{-i k \Phi (y, \phi)} \lp x Y(y, \phi) \rp^{- \frac 1 2} e^{\pm  2 i x Y(y, \phi) } d y.
	 \end{align*}
	In order to see the convergence, we split out $e^{\pm 2 i x y}$ from $e^{\pm  2 i x Y(y, \phi) }$ and put $f_{\nu, k} (y; x, \phi) = y^{2 \nu - 1} e^{-i k \Phi (y, \phi)} \lp xY(y, \phi) \rp^{- \frac 1 2} e^{\pm  2 i x (Y(y, \phi) - y) }$. Partial integration turns the above  integral into
	$$\mp \frac{1} {2 i x^{\frac 3 2}} \lp 2^{2 \nu - 1} e^{-i k \Phi (2, \phi)} Y(2, \phi)^{- \frac 1 2} e^{\pm  2 i x  Y(2, \phi)  } + \int_2^\infty \lp \partial f_{\nu, k}/ \partial y \rp (y; x, \phi) e^{\pm 2 i x y}	d y \rp.$$
Some calculations show that $\lp \partial f_{\nu, k}/ \partial y \rp (y; x, \phi) \lll_{\nu, k, x} y^{2\Re \nu - \frac 5 2}$ for $y \geq 2$, and hence the integral in the second term is absolutely convergent when $\Re \nu < \frac 3 4$. With the above arguments, the analyticity with respect to $\nu$ is  obvious.
\end{proof}
\begin{cor}\label{6cor: d=1, J mu m}
Let $\mu \in \BS \lp - \frac 3 8, \frac 3 8 \rp$ and $m \in \BZ$. We have
	\begin{equation}\label{6eq: d = 1 J mu m (x, e i phi)}
	J_{\lp  \mu, - \mu, m, 0\rp } \lp x e^{  i \phi} \rp = 4 \pi i^m \int_0^\infty y^{4 \mu - 1} \left[ y\-  +  y e^{i \phi} \right]^{-m} J_m \lp 4 \pi \sqrt x \left|y\-  +  y e^{i \phi}\right| \rp d y,
	\end{equation}
	with $x \in (0, \infty)$ and $\phi \in [0, 2 \pi)$. The integral in \eqref{6eq: d = 1 J mu m (x, e i phi)} converges if $|\Re \mu| < \frac 3 8$ and absolutely converges if and only if $|\Re \mu| < \frac 1 8$. 
\end{cor}
\begin{proof}
From Theorem \ref{6thm: formal integral, C}, we see that \eqref{6eq: d = 1 J mu m (x, e i phi)} holds for $\BS \lp - \frac 1 8, \frac 1 8 \rp$. In view of Proposition \ref{6prop: d=1, J v k}, the right hand side of \eqref{6eq: d = 1 J mu m (x, e i phi)} is analytic in $\mu$ on $\BS \lp - \frac 3 8, \frac 3 8 \rp$, and therefore it is allowed to extend the domain of equality from $\BS \lp - \frac 1 8, \frac 1 8 \rp$ onto  $\BS \lp - \frac 3 8, \frac 3 8 \rp$.
\end{proof}

\section{\texorpdfstring{Two connection formulae for $J_{(\umu, \um)} (z) $}{Two connection formulae for $J_{(\mu, m)} (z)$}}\label{sec: two connection formulae for J mu m}

In this section, we shall prove two formulae for $J_{(\umu, \um)} (z)$ in connection with the two kinds of Bessel functions of rank $n$ and {positive sign}. These Bessel functions arise as solutions of Bessel equations in \cite[\S 7]{Qi} and their relations have been unraveled in \cite[\S 8.2]{Qi}. Our motivation is based on the following self-evident identity for the rank-one example
$$ e(z + \overline z) = e(z) e(\overline z). $$


\subsection{The first connection formula}

For  $\varsigma \in \{+, -\}$, $\ulambda \in \BC^n$ and $l      = 1,..., n$, we define the following series of ascending powers of $z$ (see \cite[\S 7.1]{Qi})
\begin{equation}\label{4eq: Bessel of the first kind}
J_{l     } (z; \varsigma, \ulambda) = \sum_{m=0}^\infty \frac { (\varsigma i^n)^m  z^{ n (- \lambda_{l      } + m)} } { \prod_{k = 1}^n \Gamma \lp  { \lambda_{ k } - \lambda_{l     }}  + m + 1 \rp}, \hskip 10 pt z \in \BU.
\end{equation}
$J_{l     } (z; \varsigma, \ulambda)$ is called a {\it Bessel function of the first kind}, $n$ , $\varsigma$  and $\ulambda$ its   rank, sign and index, respectively. Since the definition \eqref{4eq: Bessel of the first kind} is valid for any $\ulambda \in \BC^n$, the assumption $\ulambda \in \BL^{n-1}$ that we imposed in \cite{Qi} is rather superfluous. Also, we have the following formula in the same fashion as \eqref{4eq: normalize J} in Lemma \ref{3lem: normalize J(x; sigma, lambda)},
	\begin{equation}
	J_{l     } \left(z; \varsigma, \ulambda - \lambda \ue^n \right) = z^{  n \lambda} J_{l     } (z; \varsigma, \ulambda).
	\end{equation}

\begin{thm}\label{4thm: connection formula}
	Let $(\umu, \um) \in \BL^{n-1} \times \BZ^n$. We have
	\begin{equation}\label{4eq: connection formula}
	\begin{split}
	J_{(\umu, \um)} (z) 
	= \left(2 \pi^2\right)^{n-1}   \sum_{l      = 1}^n S_{l     } (\umu, \um) { J_l      \big( 2\pi z^{\frac 1 n}  ; +, \umu + \tfrac 1 2 \um \big) 
		J_l      \big( 2\pi \overline z^{\frac 1 n}  ; +, \umu - \tfrac 1 2 \um \big)  },
	\end{split}
	\end{equation}
	with $S_{l     } (\umu, \um) = \prod_{k \neq l     } (\pm i)^{m_l      - m_k} / \sin \lp \pi \lp   \mu_{l     } - \mu_k \pm \frac 1 2 (m_l      - m_k) \rp \rp $.
	Here, $z^{\frac 1 n}$ is the principal  $n$-th root of $z$, that is $\lp {x e^{i \phi}} \rp^{\frac 1 n} = x^{\frac 1 n} e^{\frac 1 n i \phi}$. The expression on the right hand side of \eqref{4eq: connection formula} is independent on the choice of the argument of $z$ modulo $2 \pi$. It is understood that the right hand side should be replaced by its limit if $ (\umu, \um)$ is not generic with respect to the order $\preccurlyeq$ on $\BC \times \BZ$ in the sense of Definition {\rm \ref{3defn: ordered set}}.
\end{thm}

\begin{proof}
	Recall from (\ref{1def: G m (s)}, \ref{1def: G (mu, m)}, \ref{3def: Bessel kernel j mu m}, \ref{2eq: Bessel kernel over C, polar}) that 
	\begin{equation*}
	\begin{split}
	J_{(\umu, \um)} \lp x e^{i \phi} \rp = &  \, (2 \pi)^{n-1} \sum_{m = - \infty}^\infty i^{\sum_{k = 1}^n |m_k + m| } e^{i m \phi} \\
	& \frac 1 {2\pi i} \int_{\EC_{\left(\umu, \um + m \ue^n\right)}} \lp \prod_{l     =1}^n  \frac {\Gamma \lp s - \mu_l      + \frac 1 2 {|m_l     +m|}   \rp} {\Gamma \lp 1 - s + \mu_l      + \frac 1 2 {|m_l     +m|}   \rp}  \rp 
	\lp (2 \pi)^n x \rp^{-2s} d s.
	\end{split}
	\end{equation*}
	Assume first that $(\umu, \um)$ is generic  with respect to the order $\preccurlyeq$ on $\BC \times \BZ$. The sets  of poles of the gamma factors in the above integral are $\left\{ \mu_l      - \frac 1 2 {|m_l     +m|}  - \alpha \right \}_{\alpha \in \BN}$, $l      = 1,..., n$. With the generic assumption, the integrand has only \textit{simple} poles. We left shift the integral contour of each integral in the series and pick up the residues from these poles. The contribution from the residues at the poles of the $l     $-th gamma factor is the following absolutely convergent double series,
	\begin{equation*}
	\begin{split}
	(2 \pi)^{n-1} \sum_{m = -\infty}^\infty  i^{\sum_{k = 1}^n |m_k + m| } e^{i m \phi}  \sum_{\alpha = 0}^\infty &  \frac { (-)^\alpha  \lp (2 \pi)^n x \rp^{- 2 \mu_l      + |m_l     +m| + 2 \alpha}} {\alpha ! (\alpha + |m_l     +m|)!  } \\
	&   \prod_{k \neq l     }  \frac {\Gamma \lp \mu_l      - \mu_k - \frac 1 2 ( {|m_l     +m| - |m_k + m|} ) - \alpha \rp} {\Gamma \lp 1 - \mu_l      + \mu_k + \frac 1 2 ( {|m_l     +m| + |m_k + m| } ) + \alpha \rp}    .
	\end{split}
	\end{equation*}
	Euler's reflection formula of the Gamma function turns this into
	\begin{align*}
	& \frac {\left(2 \pi^2\right)^{n-1}  } {\prod_{k \neq l     } i^{  m_l      - m_k} \sin \lp \pi \lp   \mu_{l     } - \mu_k - \frac 1 2 (m_l      - m_k) \rp \rp} \\
	& \hskip 10 pt \sum_{m = -\infty}^\infty  i^{n |m_l     +m|}  e^{i m \phi} \sum_{\alpha = 0}^\infty  \frac { (-)^{n \alpha}  \lp (2 \pi)^n x \rp^{- 2 \mu_l      + |m_l     +m| + 2 \alpha}}   {\prod_{k = 1}^n \prod_{\pm} \Gamma \lp 1 - \mu_l      + \mu_k + \frac 1 2 ( {|m_l     +m| \pm |m_k + m| } ) + \alpha \rp } .
	\end{align*}
	We now interchange the order of summations, truncate the sum over $m$ between $- m_l     $ and $- m_l      + 1$ and make the change of indices $\beta = \alpha + |m_l     +m|$. With the observation that, no matter what $m_k$ is, 
	one of $\frac 1 2 ( {|m_l     +m| + |m_k + m| } )$ and $\frac 1 2 ( {|m_l     +m| - |m_k + m| } )$ is equal to $ \frac 1 2 (m_l      - m_k)$ and the other to $|m_l      + m | - \frac 1 2 (m_l      - m_k)$ if $m \geq - m_l      + 1$, whereas the signs in front of the two $\frac 1 2 (m_l      - m_k)$ are changed if  $m \leq - m_l      $,
	the double series in the expression above turns into
	{  \begin{align*}
		\sum_{\alpha = 0}^\infty \sum_{\beta = \alpha + 1}^\infty  \frac
		{ i^{n (\alpha + \beta)} e^{i (\beta - \alpha -m_l     ) \phi} \lp (2 \pi)^n x \rp^{- 2 \mu_l      + \alpha + \beta}} 
		{ \prod_{k =1}^n   \Gamma \lp 1 -  \mu_l      + \mu_k + \frac 1 2 (m_l      - m_k) + \alpha   \rp  \Gamma \lp 1 -  \mu_l      + \mu_k - \frac 1 2 (m_l      - m_k) + \beta   \rp   } &  \\
		+ \sum_{\alpha = 0}^\infty \sum_{\beta = \alpha }^\infty  \frac
		{ i^{n (\alpha + \beta)} e^{i (\alpha - \beta - m_l     ) \phi} \lp (2 \pi)^n x \rp^{- 2 \mu_l      + \alpha + \beta}} 
		{ \prod_{k = 1}^n   \Gamma \lp 1 -  \mu_l      + \mu_k - \frac 1 2 (m_l      - m_k) + \alpha   \rp  \Gamma \lp 1 -  \mu_l      + \mu_k + \frac 1 2 (m_l      - m_k) + \beta   \rp } &, 
		\end{align*} }
	which is then equal to
	$$\sum_{\alpha = 0}^\infty \sum_{\beta = 0}^\infty  \frac
	{ i^{n (\alpha + \beta)} e^{i (\beta - \alpha -m_l     ) \phi} \lp (2 \pi)^n x \rp^{- 2 \mu_l      + \alpha + \beta}} 
	{ \prod_{k  = 1}^n   \Gamma \lp 1 -  \mu_l      + \mu_k + \frac 1 2 (m_l      - m_k) + \alpha   \rp  \Gamma \lp 1 -  \mu_l      + \mu_k - \frac 1 2 (m_l      - m_k) + \beta   \rp   }. $$
	This double series is clearly independent on the choice of $\phi$ modulo $2 \pi$, and splits exactly as the product  $$J_l      \big( 2\pi x^{\frac 1 n} e^{\frac 1 n i \phi}; +, \umu + \tfrac 1 2 \um \big) 
	J_l      \big( 2\pi x^{\frac 1 n} e^{- \frac 1 n i \phi}; +, \umu - \tfrac 1 2 \um \big) .$$ 
	This proves \eqref{4eq: connection formula} in the case when $(\umu, \um)$ is generic. As for the nongeneric case, one just passes  to the limit.
\end{proof}

\subsection{The second connection formula} \label{sec: second connection}

According to \cite[\S 7.3.2]{Qi}, {\it Bessel functions of the second kind} are solutions of Bessel equations defined according to their asymptotics at infinity. 
To remove the restriction $\ulambda \in \BL^{n-1}$ on the definition of $J(z; \ulambda ; \xi) $, with $ \xi $ a $2 n$-th root of unity, we simply impose the additional condition
\begin{equation}
	J \left(z;  \ulambda - \lambda \ue^n ; \xi \right) = z^{  n \lambda} J(z; \ulambda ; \xi).  
\end{equation}

\begin{rem}
Let $ \xi $ be an $n$-th root of $\varsigma 1$. We may also use the following formula as an alternative definition of $J \left(z;  \ulambda  ; \xi \right)$ {\rm({\it compare \cite[Corollary 8.5]{Qi}})}
	\begin{equation} \label{5eq: connection 1st and 2nd, 1}
J(z; \ulambda ; \xi) =  \sqrt n \lp \frac {\pi  } {2 }  \rp^{\frac {n-1} 2}   (- i\xi)^{\frac {n-1} 2 + |\ulambda|} \\
 \sum_{l      = 1}^{n}  \big( i  \overline \xi \big)^{n \lambda_l     } S_l      (\ulambda) J_l      ( z; \varsigma , \ulambda).
\end{equation}
where $\lp   - i \xi  \rp^{\frac {n-1} 2 + |\ulambda|} = e^{ \lp \frac {n-1} 2 + |\ulambda| \rp \lp - \frac 1 2 \pi i + i \arg \xi \rp }  $ and $\big( i  \overline \xi \big)^{n \lambda_l     } = e^{ \frac 1 2 \pi i n \lambda_l      - i n \lambda_l      \arg \xi }$ by convention, and $S_l      (\ulambda) = \prod_{k \neq l     } 1/ \sin \lp \pi (\lambda_l      - \lambda_{k} ) \rp\-$.
\end{rem}

Given an integer $a$, define $\xi _{a, j} = e^{ 2 \pi i \frac { {  j  + a  -  1}  } n}$, $j = 1, ..., n$. Let $\sigma_{l     , d} (\ulambda) $, $d = 0, 1, ..., n-1$, $ l      = 1, ..., n$, denote the elementary symmetric polynomial in $e^{- 2 \pi i \lambda_1}, ..., \widehat {e^{- 2 \pi i \lambda_l     }}, ..., e^{- 2 \pi i \lambda_n}$ of degree $d$. It follows from \cite[Corollary 8.7]{Qi} that
\begin{equation}\label{5eq: connection 1st and 2nd, 2}
\begin{split}
J_l      ( z; + , \ulambda ) = \frac {e^{\frac 3 4 \pi i \lp  (n-1) +  2 |\ulambda| \rp}} {\sqrt n (2\pi)^{\frac {n-1} 2}}  & e^{\pi i \lp \frac 1 2 n + 2 a - 2 \rp \lambda_l     } \\
& \sum_{j=1}^n (-)^{n-j} \xi_{a, j} ^{- \frac {n-1} 2 - |\ulambda|}  \sigma_{l     , n-j} (\ulambda) J \lp z; \ulambda; \xi_{a, j} \rp.
\end{split}
\end{equation} 
In addition, 
we shall require the definition 
$$\tau_{l     } (\ulambda) = \prod_{k \neq l     }    \lp e^{-2\pi i \lambda_m} - e^{-2\pi i \lambda_k} \rp = (- 2i)^{n-1} e^{-\pi i |\ulambda|}  e^{- \pi i (n-2) \lambda_l      }  \prod_{k\neq l     }  \sin \lp \pi (\lambda_l      - \lambda_{k} ) \rp.$$ 

We introduce the column vectors of the two kinds of Bessel functions 
\begin{align*}
X (z; \ulambda) =  \big(  J_l      (z; +, \ulambda)  \big)_{l     =1}^n, \hskip 10 pt   Y_a (z; \ulambda) =  \big(  J (z;   \ulambda; \xi_{a, j})  \big)_{j=1}^n,
\end{align*}
and the matrices
\begin{align*}
\varSigma (\ulambda) & = \big( \sigma_{l     , n-j} (\ulambda) \big)_{l     , j=1}^n, \\
E_a (\ulambda) & = \mathrm{diag}\lp e^{\pi i \lp \frac 1 2 n + 2 a - 2 \rp \lambda_l     } \rp_{l      = 1}^n, \hskip 10 pt D_a (\ulambda) = \mathrm{diag}\Big(  (-)^{n-j} \xi_{a, j}^{- \frac {n-1} 2 - |\ulambda| } \Big)_{j = 1}^n.
\end{align*}
Then the formula \eqref{5eq: connection 1st and 2nd, 2} may be written as
\begin{equation}\label{5eq: connection 1st and 2nd, matrix form}
X  (z; \ulambda) = \frac {e^{\frac 3 4 \pi i \lp  (n-1) +  2 |\ulambda| \rp}} {\sqrt n (2\pi)^{\frac {n-1} 2}} \cdot E_a(\ulambda) \varSigma (\ulambda) D_a(\ulambda) Y_a  (z; \ulambda).
\end{equation}

We now formulate \eqref{4eq: connection formula} as  
\begin{align} \label{4eq: connection formula, matrix}
J_{(\umu, \um)} (z)  = (-)^{ |\um|} e^{- \frac 1 2 \pi i (n-1)} \left(  4 \pi^2\right)^{n-1} \cdot 
{^t X\left( 2 \pi z^{\frac 1 n}  ; \ulambda_{(\umu, \um)}^+ \right)}  S_{(\umu, \um)}
X \left( 2\pi \overline z^{\frac 1 n}  ; \ulambda_{(\umu, \um)}^- \right),
\end{align}
with $\ulambda^\pm_{(\umu, \um)} = \umu \pm \tfrac 1 2 \um$ and
$$S_{(\umu, \um)} =  \mathrm{diag} \lp  \tau_{l     } \lp  \ulambda_{(\umu, \um)}^\pm \rp \- e^{- \pi i \lp  (n-2) \mu_l      \mp m_l     \rp } \rp_{l     =1}^n  .$$
We insert into \eqref{4eq: connection formula, matrix} the formulae of  $X\left( 2 \pi z^{\frac 1 n}  ; \ulambda_{(\umu, \um)}^+ \right)$ and $X \left( 2\pi \overline z^{\frac 1 n}  ; \ulambda_{(\umu, \um)}^- \right)$ given by \eqref{5eq: connection 1st and 2nd, matrix form}, with $\ulambda = \ulambda_{(\umu, \um)}^+$, $a=0$ in the former and $\ulambda = \ulambda_{(\umu, \um)}^-$, $a= 1 - r $, for $r = 0, 1, ..., n$, in the latter. 
Then follows the formula
\begin{equation}\label{7eq: second matrix form}
\begin{split}
& J_{(\umu, \um)}  (z) = (-)^{(n - 1) + |\um|} \frac  {\lp 2 \pi  \rp^{n-1}}  n    \\
& {^t Y_0 \left(\hskip -1 pt 2 \pi z^{\frac 1 n}  ; \ulambda_{(\umu, \um)}^+ \hskip -1 pt  \right) } \hskip -1 pt  D_0\left( \hskip -1 pt  \ulambda_{(\umu, \um)}^+ \hskip -1 pt  \right) \hskip -1 pt 
{^{t \hskip -1 pt } \varSigma_{(\umu, \um)}  } R_{(\umu, \um)}  \varSigma_{(\umu, \um)} 
D_{1-r} \hskip -1 pt  
\left( \hskip -1 pt  \ulambda_{(\umu, \um)}^- \hskip -1 pt  \right) Y_{1-r} \left( \hskip -1 pt  2 \pi \overline z^{\frac 1 n}  ; \ulambda_{(\umu, \um)}^- \hskip -1 pt  \right) \hskip -1 pt ,
\end{split}
\end{equation}
where 
\begin{align*}
\varSigma_{(\umu, \um)} & = \varSigma \lp \ulambda_{(\umu, \um)}^+ \rp = \varSigma \lp \ulambda_{(\umu, \um)}^- \rp, \\
R_{(\umu, \um)} & =  E_0 \lp \ulambda_{(\umu, \um)}^+  \rp S_{(\umu, \um)} E_{1-r} \lp \ulambda_{(\umu, \um)}^- \rp  = \mathrm{diag} \lp  \tau_{l     } \lp \ulambda_{(\umu, \um)}^\pm  \rp\- e^{- 2 \pi i r \lambda_{(\umu, \um), l     }^\pm } \rp_{l     =1}^n.
\end{align*}
We are therefore reduced to computing the matrix ${^t \varSigma_{(\umu, \um)}  } R_{(\umu, \um)}  \varSigma_{(\umu, \um)} $. For this, we have the following lemma.

\begin{lem}\label{7lem: matrix}
	Let  $\ux = (x_1, ..., x_n) \in \BC^n$ be a generic $n$-tuple  in the sense that all its components are distinct.
	Let $\sigma_{l     , d} $, respectively $\sigma_d$, denote the elementary symmetric polynomial in $x_1, ..., \widehat {x_l     }, ..., x_n$,  respectively  $x_1, ... , x_n$, of degree $d$, and let $\tau_l      = \prod_{h \neq l     } (x_l      - x_h) $. Define the matrices $\varSigma = \big( \sigma_{l     , n-j} \big)_{l     , j=1}^n$, $X = \mathrm{diag} \lp  x_l      \rp_{l     =1}^n$ and $T = \mathrm{diag} \lp  \tau_l     \-  \rp_{l     =1}^n$. 
	Then, for any $r = 0, 1, ..., n$, the matrix $ {^t \varSigma  } X^r T \varSigma$ can be written as
	\begin{equation*}
	\begin{pmatrix}
		(-)^{n - r } A & 0 \\
		0 &  (-)^{n-r+1} B
	\end{pmatrix},
	\end{equation*}
	where
	\begin{equation*}
	A = \begin{pmatrix}
		0 & \cdots & 0 & \sigma_{n} \\
		\vdots & \iddots & \iddots & \vdots\\
		0 & \iddots & \iddots & \sigma_{n-r+2} \\
		\sigma_{n} & \cdots & \sigma_{n-r+2} & \sigma_{n-r+1 }
	\end{pmatrix}, \hskip 10 pt
	B = \begin{pmatrix}
		 \sigma_{n-r-1} &  \sigma_{n-r-2} & \cdots & \sigma_0 \\
		  \sigma_{n-r-2} & \iddots & \iddots & 0  \\
		\vdots & \iddots &  \iddots & \vdots   \\
		 \sigma_0 & 0 & \cdots & 0
	\end{pmatrix}.
	\end{equation*}
	More precisely, the $(k, j)$-th entry $a_{k, j}$, $k, j = 1,..., r$, of $A$ is given by
	\begin{equation*}
	a_{k, j} = \left\{
	\begin{split}
	& \sigma_{n+r - k - j +1} \hskip 5 pt \text{ if } k + j \geq  r + 1, \\ 
	& 0 \hskip 51 pt \text{if otherwise},
	\end{split} \right.
	\end{equation*}
	whereas the $(k, j)$-th entry $b_{k, j}$, $k, j = 1,..., n-r$, of $B$ is given by
	\begin{equation*}
	b_{k, j} = \left\{
	\begin{split}
	& \sigma_{ n- r - k - j + 1 }  \hskip 5 pt \text{ if } k + j \leq n-r+1, \\ 
	& 0  \hskip 51 pt \text{if otherwise}.
	\end{split}\right.
	\end{equation*}

	\end{lem}

\begin{proof}[Proof of Lemma \ref{7lem: matrix}]

	Appealing to the Lagrange interpolation formula, we find in \cite[Lemma 8.6]{Qi}  that the inverse of $T \varSigma$ is equal to the matrix $U = \lp (-)^{n - j} x_{l     }^{j-1} \rp_{j, l      = 1}^n $. Therefore, it suffices to show that
	\begin{equation*}
	 {^t \varSigma  } X^r = \begin{pmatrix}
		(-)^{n-r } A & 0 \\
		0 &  (-)^{n-r+1} B
	\end{pmatrix} U.
	\end{equation*}
	This is equivalent to the following two collections of identities,
	\begin{align*}
	   \sum_{j=r-k+1}^{r} (-)^{r+j} \sigma_{n+r - k - j +1}  x_l     ^{j-1}  &= \sigma_{l     , n-k } x_l     ^r, \hskip 10 pt k = 1, ..., r, \\
	   \sum_{j=1}^{n - r - k +1} (-)^{j-1} \sigma_{ n- r - k - j + 1}  x_l     ^{r+j-1}  &= \sigma_{l     , n-r-k} x_l     ^r , 
	 \hskip 10 pt k = 1, ..., n-r,
	\end{align*} 
	 which are further equivalent to
	 \begin{align*}
	 \sum_{j=1}^{k } (-)^{k+j } \sigma_{ n - j+1}  x_l     ^{j - k -1}  & = \sigma_{l     , n-k},   \hskip 10 pt  k = 1, ..., r, \\
	 \sum_{j=1}^{k } (-)^{ j - 1} \sigma_{k - j}  x_l     ^{j-1}  & = \sigma_{l     , k - 1}, \hskip 10 pt  k = 1, ..., n-r.
	\end{align*}
The last two identities can be easily seen, actually for all $k = 1,..., n$, from computing the coefficients of $x^{k-1}$ and $x^{2n - k}$ on the two sides of  
	\begin{align*}
	\prod_{h \neq l     } (x - x_h) & = \lp  \sum_{p=0}^\infty x_{l     }^p x^{-p-1} \rp \prod_{h = 1}^{n} (x - x_h), \\
	\lp x^{n} - x_l     ^{n} \rp \prod_{h \neq l     } (x - x_h) & = \lp \sum_{p=1}^{n } x_{l     }^{ p - 1} x^{n-p}  \rp \prod_{h = 1}^{n} (x - x_h),
	\end{align*}
	respectively.
\end{proof}

Applying Lemma \ref{7lem: matrix} with $x_{l     } = e^{- 2\pi i \lambda_{(\umu, \um), l     }^\pm  } = (-)^{m_l     } e^{- 2\pi i  \mu_{l     } }$ to the formula  \eqref{7eq: second matrix form},
we arrive at the following theorem.

\begin{thm}
	\label{7thm: second formula}
	Let $(\umu, \um) \in \BL^{n-1} \times \BZ^n$ and $r \in \{ 0, 1, .., n \}$. 
	Define $\xi _{ j} = e^{ 2 \pi i \frac { {  j -  1}  } n}$, $\zeta_j = e^{ 2 \pi i \frac { {  j -  r}  } n}$, and  denote by $\sigma_{(\umu, \um)}^d $ the elementary symmetric polynomial in $(-)^{m_1} e^{- 2 \pi i  \mu_1 }$, $...$, $(-)^{m_n} e^{- 2 \pi i  \mu_n }$ of degree $d$, with $j = 1, ..., n$ and $d = 0, 1, ..., n$. Then we have
	\begin{equation}\label{7eq: J (z)}
	\begin{split}
	J _{(\umu, \um)} (z) & =  (-)^{ |\um|} \frac {(2 \pi)^{n-1}} n \mathop{\sum\sum}_{\substack{ k, j = 1, ..., r \\ k+j \geq r+1 }}  
	C_{k, j} (\umu, \um)  \\ 
	& \hskip 92 pt
	J \big( 2 \pi z^{\frac 1 n} ; \umu + \tfrac 1 2 \um; \xi_k \big) J \big( 2 \pi \overline z^{\frac 1 n} ; \umu - \tfrac 1 2 \um; \zeta_j \big) \\
	& + (-)^{ |\um|} \frac {(2 \pi)^{n-1}} n \mathop{\sum\sum}_{\substack{ k, j = 1, ..., n-r \\ k+j \leq n-r+1 }}  D_{k, j} (\umu, \um)   \\
	& \hskip 74 pt
	J \big( 2 \pi z^{\frac 1 n} ; \umu + \tfrac 1 2 \um; \xi_{r+k} \big) J \big( 2 \pi \overline z^{\frac 1 n} ; \umu - \tfrac 1 2 \um; \zeta_{r+j} \big).
	\end{split}
	\end{equation}
\delete{	\begin{equation}\label{7eq: J -}
	\begin{split}
	J^-_{(\umu, \um)} (z) = &\, (-)^{ |\um|} \frac {(2 \pi)^{n-1}} n \mathop{\sum\sum}_{\substack{ k, j = 1, ..., r \\ k+j \geq r+1 }}  C^-_{k, j} (\umu, \um)  \\ 
	& \hskip 82 pt
	J \big( 2 \pi z^{\frac 1 n} ; \umu + \tfrac 1 2 \um; \xi_k \big) J \big( 2 \pi \overline z^{\frac 1 n} ; \umu - \tfrac 1 2 \um; \zeta_j \big)  ,
	\end{split}
	\end{equation}
	\begin{equation}\label{7eq: J +}
	\begin{split}
	J^+_{(\umu, \um)} (z) = &\, (-)^{ |\um|} \frac {(2 \pi)^{n-1}} n \mathop{\sum\sum}_{\substack{ k, j = 1, ..., n-r \\ k+j \leq n-r+1 }}  D_{k, j} (\umu, \um)   \\
	& \hskip 64 pt
	J \big( 2 \pi z^{\frac 1 n} ; \umu - \tfrac 1 2 \um; \xi_{r+k} \big) J \big( 2 \pi \overline z^{\frac 1 n} ; \umu - \tfrac 1 2 \um; \zeta_{r+j} \big),
	\end{split}
	\end{equation}	
}
	with
	\begin{align}
	\label{7eq: C-} C _{k, j} (\umu, \um) & =  (-)^{ r + k+j + 1} \xi_k^{-\frac {n-1} 2 - \frac 1 2 |\um|} \zeta_j^{-\frac {n-1} 2 + \frac 1 2 |\um|} \sigma_{(\umu, \um) }^{n+r-k-j+1}, \\
	\label{7eq: C+} D _{k, j} (\umu, \um) & = (-)^{ r+k+j } \xi_{r+k}^{-\frac {n-1} 2 - \frac 1 2 |\um|} \zeta_{r+j}^{-\frac {n-1} 2 + \frac 1 2 |\um|} \sigma_{(\umu, \um) }^{n-r-k-j+1}.
	\end{align}
\end{thm}


\begin{lem}\label{7lem: C and D}
	We retain the notations in Theorem {\rm \ref{7thm: second formula}}. Moreover, we define  $\mathfrak I (\umu) = \max \left\{ \left| \Im \mu_l      \right| \right\}$.
	
	{\rm (1.1).} For $k=1, ..., r$, we have 
	$ C_{k, r-k+1} (\umu, \um) = (- \overline\xi_k)^{ |\um|}.$
	
	{\rm (1.2).} Let $k, j = 1, ... , r$ be such that $k + j \geq r + 2$. Denote $p = k+j-r-1$. We have the estimate
	\begin{equation*}
	|C_{k, j} (\umu, \um)| \leq {n \choose p} \exp\big(  2 \pi \min \left\{  n-p, p \right\} \mathfrak I (\umu) \big).
	\end{equation*}
	
	{\rm (2.1).} For $k=1, ..., n-r$, we have 
	$D_{k, n-r-k+1} (\umu, \um) =  (- \overline\xi_{k+r})^{ |\um|}.$
	
	{\rm (2.2).} Let $k, j = 1, ... , n-r$ be such that $k + j \leq n-r  $. Denote $p = n - r - k - j + 1$. We have the estimate
	\begin{equation*}
	|D_{k, j} (\umu, \um)| \leq {n \choose p} \exp\big(   2 \pi \min \left\{  n-p, p \right\} \mathfrak I (\umu) \big).
	\end{equation*}	
	
\end{lem}

\subsection{The rank-two case}

\begin{example}\label{3prop: n=2, C}
	Let $\mu \in \BC$ and $m \in \BZ$. 
	
	If we define 
	\begin{equation}\label{7def: J mu m (z), n=2, C}
	J_{\mu, m} (z) = J_{- 2\mu - \frac 12 m } \lp  z \rp J_{- 2\mu + \frac 12 m  } \lp  {\overline z} \rp,
	\end{equation}
	then
	\begin{equation}\label{3eq: n=2, C}
	J_{(\mu, - \mu, m, 0)} \lp z \rp \hskip -1 pt =  \hskip -3 pt
	\left\{ 
	\begin{split}
	& \hskip -3 pt \frac {2 \pi^2} {\sin (2\pi \mu)} [\sqrt z]^{-m} \lp J_{\mu, m} (4 \pi \sqrt z) \hskip -1 pt - \hskip -1 pt  J_{-\mu, -m} (4 \pi \sqrt z) \rp \hskip 2 pt \text {if } m \text{ is even},\\
	& \hskip -3 pt \frac {2 \pi^2 i} {\cos (2\pi \mu)} [\sqrt z]^{-m} \lp J_{\mu, m} (4 \pi \sqrt z) \hskip -1 pt + \hskip -1 pt J_{-\mu, -m} (4 \pi \sqrt z) \rp \hskip 2 pt \text {if }  m \text{ is odd},
	\end{split}
	\right.
	\end{equation} 
	which should be interpreted in the way as in Theorem {\rm \ref{4thm: connection formula}}. We remark that the generic case is when  $4 \mu \notin 2\BZ + m$.
	
	On the other hand, using the connection formulae {\rm (\cite[3.61 (1, 2)]{Watson})}
	\begin{equation*}
	 J_\nu (z) = \frac {H_\nu^{(1)} (z) + H_\nu^{(2)} (z)} 2, \hskip 10 pt 
	 J_{-\nu} (z) =  \frac {e^{\pi i \nu} H_\nu^{(1)} (z) + e^{-\pi i \nu} H_\nu^{(2)} (z) } 2,
	\end{equation*}
	one obtains
	\begin{equation}
	 J_{(\mu, - \mu, m, 0)} (z) = \pi^2 i [\sqrt z]^{-m}  \lp e^{2 \pi i \mu} H^{(1)}_{\mu, m} \lp 4 \pi \sqrt z \rp + (-)^{m+1} e^{- 2 \pi i \mu} H^{(2)}_{\mu, m} \lp 4 \pi \sqrt z \rp \rp,
	\end{equation}
	with the definition
	\begin{equation}\label{7def: H (1, 2) mu m (z), n=2, C}
	H^{(1, 2)}_{\mu, m} (z) = H^{(1, 2)}_{2 \mu + \frac 1 2 m} \lp   z \rp  H^{(1, 2)}_{2 \mu - \frac 1 2 m} \lp  { \overline z} \rp.
	\end{equation}
\end{example}

\section{\texorpdfstring{The asymptotic expansion of $J_{(\umu, \um)} (z) $}{The asymptotic expansion of $J_{(\mu, m)} (z) $}}\label{sec: asymptotics}



The asymptotic of  $J_{(\umu, \udelta)} (x)$ has  already been established in  \cite[Theorem 5.13, 9.4]{Qi}. In the following, we shall present the asymptotic expansion of $J_{(\umu, \um)} (z) $. 

\vskip 5 pt

First of all, we have the following proposition on the asymptotic expansion of $J (z; \ulambda; \xi)$, which is in substance \cite[Theorem 7.27]{Qi}. 


\begin{prop}\label{8prop: asymptotic}
	Let $\ulambda \in \BC^n$ and define $\mathfrak C (\ulambda) = \max \left\{ \left| \lambda_l      - \frac 1 n |\ulambda| \right| + 1 \right\}$. Let $\xi $ be a  $2 n$-th root of unity.  For a small positive constant $ \vartheta $, say $0 < \vartheta < \frac 1 2 \pi $, we define the sector
	\begin{equation*}
	\BS'_{\xi } (\vartheta) = \left\{ z : \left| \arg z - \arg ( i \overline \xi) \right|  < \pi + \frac {\pi} n - \vartheta \right \}.
	\end{equation*} 
	For a positive integer $A $, we have  the asymptotic expansion
	\begin{equation*}
	J (z; \ulambda; \xi)   =    e^{i n \xi z} z^{- \frac { n-1} 2 - |\ulambda| }   \lp     \sum_{\alpha=0}^{A-1}   (i \xi)^{-\alpha} B_\alpha   \left(  \ulambda - \tfrac 1 n |\ulambda| \ue^n   \right) z^{-\alpha}   + O_{A, \,\vartheta,\, n} \lp \fC (\ulambda)^{2 A} |z|^{- A} \rp   \rp
	\end{equation*}
	for all $z \in \BS'_{\xi} (\vartheta)$ such that $ |z| \ggg_{A, \vartheta, n} \fC (\ulambda)^2 $. Here $B_{\alpha } (\ulambda) $  is a certain symmetric polynomial function in $\ulambda \in \BL^{n-1}$ of degree $2 \alpha$, with $B_0 (\ulambda) = 1$.
\end{prop}

\begin{lem}\label{8lem: asymptotic}
	Let $r$ be a positive integer. Suppose that either $ n = 2r$ or $n = 2r-1$. 
	Put $\vartheta_n = \frac 1 {\, n\,} \pi$ if $n = 2 r$ and $\vartheta_n = \frac 1 {2 n} \pi$ if $n = 2r-1$.
	 For a given constant $0 < \vartheta < \vartheta_n$ define the sector
	\begin{equation*}
	\BS_n (\vartheta) =  \left\{ \begin{split}
	\left\{ z : - \frac \pi 2 - \frac \pi {n} + \vartheta < \arg z < - \frac \pi 2 + \frac {3 \pi} {n} - \vartheta \right\} & \hskip 10 pt \text{ if } n = 2 r, \\
	\left\{ z : - \frac \pi 2 - \frac \pi {n} + \vartheta < \arg z < - \frac \pi 2 + \frac {2 \pi} {n} - \vartheta \right\} & \hskip 10 pt \text{ if } n = 2 r - 1, 
	\end{split} \right.
	\end{equation*}
Let $(\umu, \um) \in \BL^{n-1} \times \BZ^n$ and define $\mathfrak C (\umu, \um) = \max \left\{ |\mu_l     | + 1, \left|m_l      - \frac 1 n |\um| \right| + 1 \right\}$. Define $\xi _{ j} = e^{ 2 \pi i \frac { {  j -  1}  } n}$ and $\zeta_j = e^{ 2 \pi i \frac { {  j -  r}  } n}$ for $j=1, ..., n$.  
Then, for any $z \in \BS_n (\vartheta) $ such that $|z| \ggg_{A, \vartheta, n} \fC (\umu, \um) ^2$, we have
\begin{align*}
& J  \lp 2 \pi z  ; \umu + \tfrac 1 2 \um; \xi_k \rp J \lp 2 \pi \overline z  ; \umu - \tfrac 1 2 \um; \zeta_j \rp =  \frac { e \lp n \lp  \xi_k z + \zeta_j \overline z \rp  \rp }   { (2\pi)^{n-1} |z|^{n-1} [z]^{|\um|} }  \\
& \hskip 50 pt
\lp \mathop{\sum\sum}_{\substack{ \alpha,\, \beta = 0, ..., A-1 \\ \alpha + \beta \leq A-1 }} (i \xi_k)^{-\alpha} (i \zeta_j)^{-\beta}  B_{\alpha, \beta} (\umu, \um) z^{-\alpha} \overline z^{-\beta} + O_{A, \vartheta, n} \lp\fC (\umu, \um)^{2 A} |z|^{-A} \rp \rp,
\end{align*}
with 
\begin{equation*}
B_{\alpha, \beta} (\umu, \um) = B_\alpha \lp \umu + \tfrac 1 2 \um - \tfrac 1 {2 n} |\um| \ue^n \rp B_\beta \lp \umu - \tfrac 1 2 \um + \tfrac 1 {2 n} |\um| \ue^n \rp, \hskip 10 pt \alpha, \beta \in \BN,
\end{equation*}
where $B_\alpha (\ulambda)$ 
is the polynomial function in $\ulambda$ of degree $2 \alpha$ given in Proposition {\rm \ref{8prop: asymptotic}}.
\end{lem}

\begin{proof}


Recall that, for an integer $a$, we defined $\xi _{a, j} = e^{ 2 \pi i \frac { {  j  + a  -  1}  } n}$. Note that $\xi_j = \xi _{0, j}$ and $\zeta_j = \xi _{1-r, j}$. It is clear that
\begin{equation*}
\bigcap_{j=1}^n \BS'_{\xi_{a, j} } (\vartheta) =  \left\{ z : 
- \frac \pi 2  -   \frac { 2a+1 } n \pi + \vartheta <  \arg z < - \frac \pi 2  -    \frac { 2a-3 } n \pi - \vartheta \right \}.
\end{equation*}
We denote this sector by $\BS' _{ a } (\vartheta)$. 
Observe that, when $n = 2 r$ or $2r-1$, the intersection $\BS' _{ 0 } (\vartheta) \cap \overline {\BS' _{ 1-r } (\vartheta) } $ is exactly the sector $\BS_n (\vartheta)$. In other words,  for all $j = 1, ..., n$, $z \in \BS'_{\xi_{ j} } (\vartheta) $ and $ \overline z \in \BS'_{\zeta_{ j} } (\vartheta) $ both hold if $z \in \BS_n (\vartheta)$. Therefore, Proposition \ref{8prop: asymptotic} can be applied to yield the asymptotic expansion of  $J \lp 2 \pi z  ; \umu + \tfrac 1 2 \um; \xi_k \rp J \lp 2 \pi \overline z  ; \umu - \tfrac 1 2 \um; \zeta_j \rp$ as above. 
\end{proof}

\begin{rem}  In view of our choice of $\vartheta$, the sector $\BS_{n} (\vartheta) $ is of angle at least $ \frac 2 {\,n\,} \pi$, and therefore the sector $\BS_{n} (\vartheta)^n = \left\{ z^n : z \in \BS_{n} (\vartheta) \right\}$ covers  the whole  $\BC \smallsetminus \{0\}$. 
	
\end{rem}
\begin{lem} \label{8lem: positive im}
	
		Let notations be as in Lemma {\rm \ref{8lem: asymptotic}}. 
	
	{\rm (1.1).} For $k = 1,..., r$, we have 
	\begin{equation*}
	\Im \lp \xi_{k} z + \zeta_{r-k+1} \overline z \rp = 0.
	\end{equation*}

	{\rm (1.2).} Let $k, j = 1, ... , r$ be such that $k + j \geq r + 2$. For any $z \in \BS_n (\vartheta)$, we have 
	\begin{equation*}
	\Im \lp \xi_{k} z + \zeta_j \overline z \rp \geq 2 \sin \lp\frac {k+j-r-1} n \pi\rp \sin \vartheta \cdot |z|.
	\end{equation*}
	
	{\rm (2.1).} For $k = 1,..., n - r$, we have 
	\begin{equation*}
	\Im \lp \xi_{k+r} z + \zeta_{n-k+1} \overline z \rp = 0.
	\end{equation*}
	
	{\rm (2.2).} Let $k, j = 1, ... , n-r$ be such that $k + j \leq n-r  $. For any $z \in \BS_n (\vartheta)$, we have 
	\begin{equation*}
	\Im \lp \xi_{k + r} z + \zeta_{j+r} \overline z \rp \geq 
	\left\{ \begin{split}
	& 2\sin \lp\frac {n-r -k-j +1} n \pi \rp \sin \vartheta \cdot |z|, \hskip 43 pt \text{ if } n=2r, \\
	& 2\sin \lp\frac {n-r-k-j +1} n \pi \rp \sin \lp \frac \pi n + \vartheta \rp \cdot |z|,  \hskip 10 pt \text{ if } n=2r -1.
	\end{split} \right.
	\end{equation*}
	
\end{lem}

\begin{proof}
	We shall only prove (1.1) and (1.2) in the case $n = 2 r$. The other cases follow in exactly the same way.
	
	Write $z = x e^{i \phi}$.  
	Since
	\begin{align*}
	\xi_{k} z + \zeta_j \overline z & = xe^{2 \pi i \frac {k-1} {2r} + i \phi} + xe^{2 \pi i \frac {j-r} {2r} -i \phi} \\
	& = xe^{\pi i \frac {k+j - r - 1} {2r} } \lp e^{ \pi i \frac {k-j+r-1} {2r} + i \phi}  + e^{ - \pi i \frac {k-j+r-1} {2r} - i \phi} \rp,
	\end{align*}
	(1.1) is then obvious (we also note that $\zeta_{r-k+1} = \overline \xi_k $), whereas (1.2) is equivalent to
	\begin{equation}\label{8eq: sin cos}
	\cos \lp  \frac {k-j+r-1} {2 r} \pi + \phi \rp \geq  \sin \vartheta.
	\end{equation}
Observe that the condition $z \in \BS_{2r} (\vartheta)$ amounts to 
\begin{equation*}
\left|\phi + \frac \pi 2  -\frac \pi {2r} \right| < \frac {\pi} r  - \vartheta.
\end{equation*} Moreover, under our assumptions on $k$ and $j$ in (1.2), one has $|k-j| \leq r-2$.    Consequently, these yield  the following estimate
\begin{align*}
\left|\frac {k-j+r-1} {2 r} \pi + \phi \right| \leq \frac {r-2} {2 r} \pi + \frac \pi r  - \vartheta = \frac \pi 2 - \vartheta.
\end{align*}
Thus \eqref{8eq: sin cos} is proven.
\end{proof}

\begin{rem}
	In cases other than those listed in Lemma {\rm \ref{8lem: positive im}}, $\Im \lp \xi_{k} z + \zeta_{j} \overline z \rp  $ can not always be nonnegative for all $z \in \BS_n (\vartheta)$. 
	Fortunately, these cases are excluded from the second connection formula for $J_{(\umu, \um)} (z)$ in Theorem {\rm \ref{7thm: second formula}}.
\end{rem}

Now the asymptotic expansion of $J _{(\umu, \um)} \left(z \right)$ can   be readily established using Theorem \ref{7thm: second formula} along with Lemma \ref{7lem: C and D}, \ref{8lem: asymptotic} and \ref{8lem: positive im}.

\begin{thm}\label{8thm: asymptotic}
Denote by $\BX_n  $ the set of $n$-th roots of unity. Let $(\umu, \um) \in \BL^{n-1} \times \BZ^n$ and define $\mathfrak C (\umu, \um) = \max \left\{ |\mu_l     | + 1, \left|m_l      - \frac 1 n |\um| \right| + 1 \right\}$. Let $A$ be a positive integer. Then
\begin{equation*}
\begin{split}
J _{(\umu, \um)} \left(z^n \right)   = \sum_{\xi \in \BX_n } \frac { e \big( n \big( \xi z +   \overline {\xi z } \big) \big) }   { n |z|^{n-1} [\xi z]^{|\um|} }   
\lp \mathop{\sum\sum}_{\substack{ \alpha,\, \beta = 0, ..., A-1 \\ \alpha + \beta \leq A-1 }} i^{- \alpha - \beta} \xi^{ -\alpha + \beta} B_{\alpha, \beta} (\umu, \um) z^{-\alpha} \overline z^{-\beta} \rp & \\
  \hskip 75 pt + O_{A,  n} \lp\fC (\umu, \um)^{2 A} |z|^{-A-n+1} \rp   & ,
\end{split}
\end{equation*}
if $|z| \ggg_{A, n}  \fC (\umu, \um)^2$, with the coefficient $B_{\alpha, \beta} (\umu, \um)$ given 
in Lemma {\rm \ref{8lem: asymptotic}}. 

\end{thm}

We may also prove the following elaborate version of Theorem \ref{8thm: asymptotic}.

\begin{thm}\label{8thm: asymptotic, complex, 2}
	Let notations be as in  Lemma {\rm \ref{8lem: asymptotic}} and Theorem {\rm \ref{8thm: asymptotic}}.
	Let $\mathfrak I (\umu) = \max \left\{ \left| \Im \mu_l      \right| \right\}$. Then we may write 
	 \begin{align*}
	 J _{(\umu, \um)} \left(z^n \right) = \sum_{\xi \in \BX_n } \frac { e \big( n \big( \xi z +   \overline {\xi z } \big) \big) }   { n |z|^{n-1} [\xi z]^{|\um|} } W_{(\umu, \um)} \lp  z , \xi \rp + E_{(\umu, \um)} (z),
	 \end{align*} 
	 such that
	 \begin{align*}
	 W_{(\umu, \um)} \lp z , \xi \rp =    
	 \mathop{\sum\sum}_{\substack{ \alpha,\, \beta = 0, ..., A-1 \\ \alpha + \beta \leqslant A-1 }} i^{- \alpha - \beta} \xi^{- \alpha - \beta}  B_{\alpha, \beta} (\umu, \um) z^{-\alpha } \overline z^{-\beta }  + O_{A,  n} \lp\fC (\umu, \um)^{2 A} |z|^{-A } \rp ,
	 \end{align*} 
	 and
	 \begin{align*}
	 E_{ (\umu, \um) } (z) = O_{ n} \big( |z|^{- n + 1} \exp \lp 2 \pi \mathfrak I (\umu)   - 4 \pi n \sin \lp \tfrac 1 {\,n\,} \pi \rp \sin   \vartheta   |z| \rp \big),
	 \end{align*} 
	for $z \in  \BS_n (\vartheta) $ with $|z| \ggg_{A, n}  \fC (\umu, \um)^2$.  Moreover, $ E_{ (\umu, \um) } (z) \equiv 0 $ when $n = 1, 2$.
\end{thm}

\delete{
In Theorem \ref{8thm: asymptotic, complex, 2}, $W_{(\umu, \um)} (z) = J  \lp 2 \pi z  ; \umu + \tfrac 1 2 \um; 1 \rp J \lp 2 \pi \overline z  ; \umu - \tfrac 1 2 \um; 1 \rp$. It should be noted that
\begin{align*}
J  \lp 2 \pi z  ; \umu + \tfrac 1 2 \um; \xi_k \rp 
J \big( 2 \pi \overline z  ; \umu - \tfrac 1 2 \um; \overline \xi_k \big) = J  \lp 2 \pi \xi_k z  ; \umu + \tfrac 1 2 \um; 1 \rp 
J \big( 2 \pi \overline \xi_k\overline z  ; \umu - \tfrac 1 2 \um; 1 \big),
\end{align*}
which is due to \cite[Lemma 7.22]{Qi}.
}


\section{Hankel transforms from the representation theoretic viewpoint}\label{sec: Hankel, Ichino-Templier}
We shall start with a brief review of Hankel transforms over an \textit{archimedean} local field\footnote{For a nonarchimedean local field, Hankel transforms can also be constructed in the same way.} in the work of Ichino and Templier \cite{Ichino-Templier} on the \Voronoi summation formula. 
For the theory of $L$-functions and local functional equations over a local field the reader is referred to Cogdell's survey \cite{Cogdell}. We shall then study  Hankel transforms using the Langlands classification.
For this,  Knapp's article \cite{Knapp} is used as our reference, with some change of notations for our convenience.
\vskip 5 pt

Let $\BF$ be an archimedean local field with normalized absolute value $\|\, \| = \|\, \|_{\BF}$ defined as in \S \ref{sec: R+, Rx and Cx}, 
and let $\psi $ be a given additive character on $\BF$.  For $s \in \BC$, let $\omega_s $ denote the character $\omega_s (x) = \|x\|^s$. Let $\eta (x) = \sgn (x)$ for $x \in \BRx$ and  $\eta (z) = [z]$ for $z \in \BCx$. 


Suppose for the moment $n \geq 2$. Let $\pi $ be an infinite dimensional irreducible admissible generic representation of $\GL_n( \BF)$\footnote{Since $\pi $ is a local component of an irreducible cuspidal automorphic representation in \cite{Ichino-Templier},  \cite{Ichino-Templier} also assumes that $\pi$ is unitary. However, if one only considers the local theory, this assumption is not necessary.}, and $\mathcal W (\pi, \psi)$ be the $\psi$-Whittaker model of $\pi$. Denote by $\omega_\pi$ the central character of $\pi$. Recall that the $\gamma$-factor $\gamma (s, \pi , \psi )$ of $\pi$ is given by
\begin{equation*}
\gamma (s, \pi , \psi ) = \epsilon (s, \pi , \psi ) \frac {L(1 - s, \widetilde \pi )} { L(s,\pi ) }
\end{equation*}
where $\widetilde \pi $ is the contragradient representation of $\pi$,  $\epsilon (s, \pi , \psi )$ and $L( s, \pi)$ are the $\epsilon$-factor and the $L$-function of $\pi$ respectively. 

To a smooth compactly supported function $w$ on $\BF^{\times}$ we associate a dual function $\widetilde w$ on $\BF^{\times}$ defined by \cite[(1.1)]{Ichino-Templier},
\begin{equation}\label{4eq: Ichino-Templier}
\begin{split}
  \int_{\BF^\times} \widetilde w (x) \chiup (x)\- & \|x\|^{s - \frac {n-1} 2} d^\times x \\
& = \chiup (-1)^{n-1} \gamma (1-s, \pi \otimes \chiup, \psi ) \int_{\BF^\times} w (x) \chiup (x) \|x\|^{1 - s - \frac {n-1} 2} d^\times x,
\end{split}
\end{equation}
for all $s$ of real part sufficiently large and all unitary multiplicative characters $\chiup $ of $\BF^\times$. (\ref{4eq: Ichino-Templier}) is independent of the chosen Haar measure $d^\times x$ on $\BF^\times$, and uniquely defines $\widetilde w $ in terms of $\pi$, $\psi$ and $w$. We shall let the Haar measure be given as in \S \ref{sec: R+, Rx and Cx}. We call $\widetilde w $ the Hankel transform of $w$ associated with $\pi$.

According to \cite[Lemma 5.1]{Ichino-Templier}, there exists a smooth Whitaker function $W \in \mathcal W (\pi, \psi)$ so that
\begin{equation}\label{4eq: w = Whittaker}
w(x) = W   
\begin{pmatrix}
x & \\
  & I_{n-1}
\end{pmatrix},
\end{equation}
for all $x \in \BF^\times$. Denote by $\varw_n$ the  $n$-by-$n$ permutation matrix whose anti-diagonal entries are $1$, that is, the longest Weyl element of rank $n$,
and define
\begin{equation*}
\varw_{n, 1} = \begin{pmatrix}
1 & \\
  & \varw_{n-1}
\end{pmatrix}.
\end{equation*}
In the theory of integral representations of Rankin-Selberg $L$-functions, \eqref{4eq: Ichino-Templier} amounts to the local functional equations of 
zeta integrals for $\pi \otimes \chiup$, with
\begin{equation}\label{4eq: tilde w for n = 2}
\widetilde w (x) = \widetilde W 
\begin{pmatrix}
x & \\
  & 1
\end{pmatrix}  = 
W \lp \varw_2
\begin{pmatrix}
x\- & \\
  & 1
\end{pmatrix}\rp,
\end{equation}
if $n = 2$, and
\begin{equation}\label{4eq: tilde w for n > 2}
\widetilde w (x) = \int_{\BF^{n-2}} \widetilde W \lp
\begin{pmatrix}
x &  &  \\
y & I_{n-2} & \\
  &  & 1
\end{pmatrix} \varw_{n, 1}
\rp d y_{\psi},
\end{equation}
if $n \geq 3$, where $\widetilde W \in \mathcal W(\widetilde \pi, \psi\-)$ is the dual Whittaker function defined by $\widetilde W (g) = W(\varw_n \cdot {^t g\-})$, for $g \in \GL_n( \BF)$, and $d x_\psi$ denotes the self-dual additive Haar measure on $\BF$ with respect to $\psi$. See \cite[Lemma 2.3]{Ichino-Templier}.

\vskip 5 pt

The constraint that $\pi$ be \textit{infinite dimensional} and {\it generic} is actually dispensable for defining the Hankel transform via \eqref{4eq: Ichino-Templier}. In the following, we shall assume that $\pi$ is any irreducible admissible representation of $\GL_n(\BF)$. Moreover, we shall also include the case $n=1$. It will be seen that, after renormalizing  the functions $w$ and $\widetilde w$, the Hankel transform defined by \eqref{4eq: Ichino-Templier} turns into the Hankel transform given by \eqref{3eq: Hankel transform identity, R} or \eqref{3eq: Hankel transform identity, C}.
For this, we shall apply the Langlands classification for irreducible admissible representations of $\GL_n(\BF)$.

\subsection{Hankel transforms over $\BR$} 

 Suppose $\BF = \BR$. Recall that $\|\, \|_\BR = |\ |$ is the ordinary absolute value. 
For $  r     \in \BR^\times$ let $\psi (x) = \psi_  r     (x) = e (  r     x)$.

According to \cite[\S 3, Lemma]{Knapp}, every finite dimensional semisimple representation $\varphi$ of the Weil group of $\BR$ may be decomposed into irreducible representations of dimension one or two. The one-dimensional representations are parametrized by $ (\mu, \delta) \in \BC \times \BZT$. We denote by $\varphi_{(\mu, \delta)}$ the representation given by $(\mu, \delta)$. $\varphi_{(\mu, \delta)}$ corresponds to the representation $\chiup_{(\mu, \delta)} = \omega_{\mu} \eta^{\delta}   $ of $\GL_1( \BR)$ under the Langlands correspondence over $\BR$. The irreducible two-dimensional representations are parametrized by $(\mu, m) \in \BC \times \BN_+$. We denote by $\varphi_{( \mu, m)}$ the representation given by $( \mu, m)$. $\varphi_{( \mu, m)}$ corresponds to the representation $\sigma(m) \otimes \omega_{\mu}(\det) $ of $\GL_2( \BR)$, where $\sigma(m)$ denotes the discrete series representation of weight $m$. 

In view of the formulae \cite[(3.6, 3.7)]{Knapp}\footnote{The formulae in \cite[(3.6, 3.7)]{Knapp} are for $\psi_1$. The relation between the epsilon factors $\epsilon (s, \pi , \psi_r )$ and $\epsilon (s, \pi , \psi )$ is given in \cite[\S 3]{Tate} (see in particular \cite[(3.6.6)]{Tate}).} of   $L$-functions and $\epsilon$-factors, the definitions  of $G_\delta$ and $G_m$ in \eqref{1def: G delta} and \eqref{1def: G m (s)}, along with the formula \eqref{1f: G m (s) = G 1 G delta(m)}, we deduce that
\begin{equation}\label{4eq: gamma, character}
\begin{split}
 \gamma ( s, \varphi_{( \mu, \delta)} , \psi)  
= \sgn (  r     )^{\delta}  |  r     |^{s  + \mu - \frac 1 2 } G_{\delta} (1 - s  - \mu),
\end{split}
\end{equation} 
whereas
\begin{equation}\label{4eq: gamma, discrete}
\begin{split}
\gamma (s, \varphi_{( \mu, m)} , \psi) & = \sgn (  r     )^{\delta(m)+1} |  r     |^{ 2s + 2\mu - 1} i G_m(1 - s - \mu),
\end{split}
\end{equation}
and
\begin{equation}\label{4eq: gamma, discrete, 2}
\begin{split}
\gamma (s, \varphi_{( \mu, m)} , \psi) & = \gamma  ( s, \varphi_{ ( \mu + \frac 12 {m} ,\, \delta(m) + 1  )} , \psi ) \gamma ( s, \varphi_{ ( \mu - \frac 12 {m} ,\, 0  )} , \psi ) \\
& = \gamma  ( s, \varphi_{( \mu + \frac 12 {m} ,\, \delta (m) )} , \psi ) \gamma  ( s, \varphi_{( \mu - \frac 12 {m} ,\, 1 )} , \psi ).
\end{split}
\end{equation}
To $ \varphi_{( \mu, m)}$ we shall attach either one of the following two parameters 
\begin{equation}\label{4eq: two parameters, discrete}
\textstyle \big( \mu + \frac 12 {m} , \mu - \frac 12 {m} , \delta (m) + 1, 0 \big), \ \big(\mu + \frac 12{m}  , \mu - \frac 12 {m}  , \delta (m), 1 \big).
\end{equation}

\begin{rem}\label{4rem: discrete series}
\eqref{4eq: gamma, discrete, 2} reflects the isomorphism $\varphi_{( 0, m)} \otimes \varphi_{( 0, 1)} \cong \varphi_{( 0, m)}$ of representations of the Weil group {\rm ({\it here $(0, 1)$ is an element of $\BC \times \BZT$})}, as well as 
the isomorphism $\sigma(m) \otimes \eta \cong \sigma(m)$ of representations of $\GL_2(\BR)$. 
\end{rem}

For $\varphi$ reducible, 
$\gamma (s, \varphi , \psi)$ is the product of the $\gamma$-factors of  the irreducible constituents of $\varphi$. Suppose $\varphi$ is $n$-dimensional. It follows from (\ref{4eq: gamma, character}, \ref{4eq: gamma, discrete}, \ref{4eq: gamma, discrete, 2}) that there is a parameter $(\umu, \udelta) \in \BC^{n} \times (\BZT)^n$ attached to $\varphi$ such that
\begin{equation}\label{4eq: gamma factor of phi}
\gamma ( s, \varphi, \psi) = \sgn(  r     )^{|\udelta|} |  r     |^{n\lp s - \frac 1 2 \rp + |\umu|} G_{(\umu, \udelta)} (1 - s ).
\end{equation}
The irreducible constituents of $\varphi$ are unique up to permutation, but, in view of the two different parameters  attached to $\varphi_{(\mu, m)}$ in \eqref{4eq: two parameters, discrete}, the  parameter $(\umu, \udelta)$ attached to $\varphi$ may not.

Suppose that $\pi$ corresponds to $\varphi $ under the Langlands correspondence over $\BR$. We have $\gamma ( s, \pi, \psi) = \gamma ( s, \varphi, \psi)$. It is known that $\pi$ is an irreducible constituent of the principal series representation unitarily induced from the character $ \bigotimes _{l      = 1}^n \chiup_{(\mu_l     , \delta_l     )}$ of the Borel subgroup. 
In particular, 
\begin{align}\label{4eq: central char, R}
\omega_\pi (x) =  \omega_{|\umu|} (x) \eta^{|\udelta|} (x) = \sgn (x)^{|\udelta|} |x|^{|\umu|}.
\end{align}

Now let $\chiup = \chiup_{(0,\delta)} = \eta^{\delta}$ in \eqref{4eq: Ichino-Templier}, $\delta \in \BZT$.
In view of \eqref{4eq: gamma factor of phi} and \eqref{4eq: central char, R}, one has the following expression of the $\gamma$-factor in   \eqref{4eq: Ichino-Templier},
\begin{equation}\label{4eq: gamma factor real case}
\gamma (1-s, \pi \otimes \eta^{\delta} , \psi) = \omega_\pi (  r     ) \big( \sgn (  r     )^{ \delta} |  r     |^{ \frac 1 2 - s} \big)^n  G_{(\umu, \udelta + \delta \ue^n)}  (s ).
\end{equation}
Some calculations show that (\ref{4eq: Ichino-Templier}) is exactly translated into \eqref{3eq: Hankel transform identity, R} if one let
\begin{equation}\label{4eq: upsilon and w, real case}
\begin{split}
\upsilon (x) & =  \omega_\pi (  r     ) w \big(  |  r     |^{- \frac n 2}x \big) |x|^{ - \frac {n-1} 2}, \\
 \Upsilon (x) & = \widetilde w \big( (-)^{n-1} \sgn (  r     )^{ n } |  r     |^{- \frac n 2}x \big) |x|^{ - \frac {n-1} 2}.
\end{split}
\end{equation}
Then, \eqref{3eq: Hankel transform, with Bessel kernel, R} can be reformulated as
\begin{equation}\label{4eq: Hankel transform tilde w and w, real case}
\widetilde w \big((-)^{n-1} x \big) = \omega_\pi (  r     ) |  r     |^{\frac n 2} |x|^{\frac {n-1} 2} \int_{\BR ^\times} w (y) J_{(\umu, \udelta)} (   r     ^n xy ) |y|^{1 - \frac {n-1} 2} d^\times y.
\end{equation}

\subsection{Hankel transforms over $\BC$} \label{sec: Hankel transform over C}
Suppose $\BF = \BC$. Recall that $\|\, \|_\BC = \|\, \| = |\  |^2$, where $|\  |$ denotes the ordinary absolute value. 
For $  r     \in \BC^\times$ let $\psi (z) = \psi_  r     (z) = e (  r     z + \overline {  r     z})$.

The Langlands classification and correspondence for $\GL_n( \BC)$ are less complicated. First of all,
the Weil group of $\BC$ is simply $\BC^\times$. Any $n$-dimensional semisimple representation $\varphi$ of the Weil group $\BC^\times$ is the direct sum of one-dimensional representations. The one-dimensional representations are of the form $\chiup_{( \mu, m)}  = \omega_{\mu} \eta^m  $, with $ ( \mu, m) \in \BC \times \BZ.$
In view of the formulae \cite[(4.6, 4.7)]{Knapp} of $L(s, \chiup_{( \mu, m)})$ and $\epsilon (s, \chiup_{( \mu, m)}, \psi)$ as well as the definition  of  $G_m$ in \eqref{1def: G m (s)}, we have
\begin{equation}
\gamma (s, \chiup_{( \mu, m)}, \psi) = [  r     ]^m \|  r     \|^{ s + \mu - \frac 1 2 } G_m (1-s-\mu).
\end{equation}
Thus $\varphi$ is parametrized by some $(\umu, \um)\in \BC^{n} \times \BZ^n$ and
\begin{equation}\label{4eq: gamma factor of chi, complex case}
\gamma (s,  \varphi , \psi) = [  r     ]^{|\um|}  \|  r     \|^{n \lp s - \frac 1 2\rp + |\umu|} G_{(\umu, \um)} (1-s).
\end{equation}
This parametrization is unique up to permutation, in contrast to the case $\BF = \BR$.

If $\pi$ corresponds to $\varphi $ under the Langlands correspondence over $\BC$, then one has $\gamma ( s, \pi, \psi) = \gamma ( s, \varphi, \psi)$. Moreover, $\pi$ is an irreducible constituent of the principal series representation unitarily induced from the character $ \bigotimes _{l      = 1}^n \chiup_{(\mu_l     , m_l     )}$ of the Borel subgroup. Note that 
\begin{align}\label{4eq: central char, C}
\omega_\pi  (z) = \omega_{|\umu|} (z) \eta^{|\um|} (z) = [z]^{|\um|} \|z\|^{|\umu|}. 
\end{align}

Now let $\chiup = \chiup_{(0, m)} =\eta^m$ in \eqref{4eq: Ichino-Templier}, $m \in \BZ $. Then \eqref{4eq: gamma factor of chi, complex case}  and \eqref{4eq: central char, C} imply
\begin{equation}\label{4eq: gamma factor, complex case}
\gamma (1-s,  \pi \otimes \eta^m, \psi) = \omega_\pi (  r     ) \big(  [  r     ]^m \|  r     \|^{ \frac 1 2 - s } \big)^n  G_{(\umu, \um + m \ue^n)} (s).
\end{equation}
By putting  
\begin{equation}\label{4eq: upsilon = w Upsilon = tilde w, complex}
\begin{split}
\upsilon (z) &= \omega_\pi (  r     ) w \big( \|  r     \|^{- \frac n 2} z \big) \|z\| ^{ - \frac {n-1} 2}, \\
 \Upsilon (z) &=   \widetilde w \big( (-)^{n-1} [  r     ]^{- n} \|  r     \|^{-  \frac n 2}  z \big) \|z\| ^{ - \frac {n-1} 2},
\end{split}
\end{equation}
the identity (\ref{4eq: Ichino-Templier}) is translated into \eqref{3eq: Hankel transform identity, C}, and \eqref{3eq: Hankel transform, with Bessel kernel, C} can be reformulated as
\begin{equation}\label{4eq: Hankel transform tilde w and w, complex case}
\widetilde w \lp (-)^{n-1}  z \rp = \omega_\pi (  r     )  \|  r     \|^{\frac n 2} \| z\|^{\frac {n-1} 2} \int_{\BC^\times } w \lp u \rp J_{(\umu, \um)} (   r     ^n z u) \| u\|^{1 - \frac {n-1} 2} d^\times u.
\end{equation}



\subsection{Some new notations} Let $\pi$ be an irreducible admissible representation of $\GL_n(\BF)$. 
For $\BF = \BR$, respectively $\BF = \BC$, if $\pi$ is parametrized by $(\umu, \udelta)$, respectively $(\umu, \um)$, we shall denote  simply by $J_{\pi}$ the Bessel kernel $J_{(\umu, \udelta)}$, respectively $J_{(\umu, \um)}$.
Thus, \eqref{4eq: Hankel transform tilde w and w, real case} and \eqref{4eq: Hankel transform tilde w and w, complex case} can be uniformly combined into one formula
\begin{equation}\label{4eq: Hankel transform tilde w and w}
\widetilde w \lp (-)^{n-1}  x \rp = \omega_\pi (  r     )  \|  r     \|^{\frac n 2} \| x\|^{\frac {n-1} 2} \int_{\BF^\times } w \lp y \rp J_{\pi} (   r     ^n x y) \| y\|^{1 - \frac {n-1} 2} d^\times y.
\end{equation}

Proposition \ref{3prop: properties of J, R} (1) and \ref{3prop: properties of J, C} (1) are translated into the following lemma. 
\begin{lem}\label{4lem: normalizing the index}
Let $\pi$ be an irreducible admissible representation of $\GL_n(\BF)$, and let $\chiup $ be a character on $\BFx $. We have
$
J_{\chiup \otimes \pi}(x) = \chiup \- (x) J_{\pi} (x).
$
\end{lem}

\begin{rem}\label{4rem: G/Z}
Let $\mathrm Z_n$ denote the center of $\GL_n$. In view of Lemma {\rm \ref{4lem: normalizing the index}}, no generality will be lost if one only considers $J_{\pi}$ for irreducible admissible  representations $\pi$ of $\GL_n (\BF)/ \mathrm Z_n(\BR _+)$.
\end{rem}

Let $\varphi$ be the $n$-dimensional semisimple representation of the Weil group of $\BF$ corresponding to $\pi$ under the Langlands correspondence over $\BF$.

If $\BF = \BR$,  the function space $\Ssis^{(\umu, \udelta)} (\BRx)$ depends on the choice of the parameter $(\umu, \udelta)$ attached to $\varphi$, if some discrete series $\varphi_{(\mu, m)}$ occurs in its decomposition. Thus, one needs to redefine the function spaces  for Hankel transforms according to the Langlands classification rather than   the above parametrization. For this, let $n_1, n_2 \in \BN$, $(\umu^1, \udelta^1) \in \BC^{n_1} \times (\BZT)^{n_1}$ and $(\umu^2, \um^2) \in \BC^{n_2} \times \BN_+ ^{n_2}$ be such that $n_1 + 2 n_2 = n$ and $\varphi = \bigoplus_{l      = 1}^{n_1} \varphi_{(\mu^1_{l     },\, \delta^1_{l     })} \oplus \bigoplus_{l      = 1}^{n_2} \varphi_{(\mu^1_{l     },\, m^2_{l     })}$.
We define the function space $\Ssis^{\pi} (\BRx) = \Ssis^{\varphi} (\BRx)$ to be
\begin{equation}\label{4eq: Ssis phi, R}
 \Ssis^{(- \umu^1,\, \udelta^1)} (\BRx) + \sum_{\delta \in \BZT} \sgn(x)^\delta \Ssis^{( - \umu^2 + \frac 1 2 \um^2, \boldsymbol 0 )} (\BRx),
\end{equation}
where $\Ssis^{( \umu ,\, \udelta)} (\BRx)$ is defined by \eqref{3eq: Ssis (lambda, delta), R}.

\begin{lem}\label{4lem: Ssis pi}
$ \Ssis^{\pi} (\BRx)$ is the sum of $\Ssis^{(- \umu, \udelta)} (\BRx)$ for all the parameters $(\umu, \udelta)$ attached to $\pi$.
\end{lem}

\begin{proof}
For $\delta \in \BZT$ and $j \in \BN$, we have the inclusion $$\sgn (x)^{\delta + \delta(m)} |x|^{\mu + \frac 12 m } (\log |x|)^j \SS (\BR) \subset \sgn (x)^{\delta} |x|^{\mu - \frac 12 m  } (\log |x|)^j \SS (\BR).$$
It follows that
\begin{equation*}
\begin{split}
& \sum_{\delta \in \BZT} \lp \sgn (x)^{\delta + \delta(m)} |x|^{\mu + \frac 12 m } (\log |x|)^j \SS (\BR) + \sgn (x)^{\delta + 1} |x|^{\mu - \frac 12 m } (\log |x|)^j \SS (\BR) \rp\\
& = \sum_{\delta \in \BZT} \sgn (x)^{\delta } |x|^{\mu - \frac 12 m } (\log |x|)^j \SS (\BR).
\end{split}
\end{equation*}
Then it is easy to verify this lemma by definitions.
\end{proof}

If $\BF = \BC$, we put
\begin{equation}\label{4eq: Ssis phi, C}
\Ssis^{\pi} (\BCx) = \Ssis^{\varphi} (\BCx) = \Ssis^{(- \umu ,\, -\um)} (\BCx).
\end{equation}

Let $d = [\BF : \BR]$. For each character $\chiup$ on $\BFx/\BR _+$ we define the Mellin transform $\EM_{\chiup}$ of a function $\upsilon \in \Ssis (\BFx)$ by 
\begin{equation} \label{4def: Mellin transform over F}
\EM _{\chiup} \upsilon (s) = \int_{\BF^\times} \upsilon (x) \chiup (x) \|x\|^{\frac 1 d s } d^\times x.
\end{equation} 

\begin{thm}
Let $\pi$ be an irreducible admissible representation of $\GL_n(\BF)$. Suppose $\upsilon \in \Ssis^{\pi} (\BFx)$. Then there exists a unique $\widetilde \upsilon \in  \Ssis^{\widetilde \pi} (\BFx)$ satisfying the following identity
\begin{equation*}
\EM_{\chiup\-} \widetilde \upsilon (d s) = \gamma (1-s, \pi \otimes \chiup, \psi_1) \EM_{\chiup } \upsilon (d (1-s))
\end{equation*}
for all  characters $\chiup$ on $\BFx/\BR _+$. We write $\EH_{\pi} \upsilon  = \widetilde \upsilon$ and call $\widetilde \upsilon$ the normalized Hankel transform of  $\upsilon$ over $\BFx$ associated with $\pi$. Moreover, we have the Hankel inversion formula
\begin{equation*} 
\EH_{\pi} \upsilon  = \widetilde \upsilon, \hskip 10 pt \EH_{\widetilde \pi} \widetilde \upsilon  = \upsilon.
\end{equation*}
\end{thm}

\begin{proof}
If $\BF = \BR$, 
this follows from  Theorem  \ref{3prop: H (lambda, delta)}, combined with Lemma \ref{4lem: Ssis pi}. If $\BF = \BC$, this is simply a translation of Theorem \ref{3prop: H (mu, m)}.
\end{proof}


\section{Bessel functions for $\GL_2 (\BF)$}\label{sec: Bessel, GL2(F)}


Let $n=2$ and retain the notations from \S \ref{sec: Hankel, Ichino-Templier} except for the different choice of the Weyl element
$\varw_2 = \begin{pmatrix}
 & -1 \\
1&
\end{pmatrix}$, which is more often used  for $\GL_2$ in the literature.

Let $\pi$ be an infinite dimensional irreducible admissible representation  of $\GL_2( \BF)$\footnote{It is well-known that a representation of $\GL_2 (\BF)$ satisfying these conditions is generic.}. Using (\ref{4eq: w = Whittaker}, \ref{4eq: tilde w for n = 2}), one may rewrite \eqref{4eq: Hankel transform tilde w and w} as follows,
\begin{equation}\label{7eq: Hankel transform, W}
W   
\begin{pmatrix}
    & 1\\
- x\- &  
\end{pmatrix} 
= \omega_\pi (  r     )  \|  r     \|  \int_{\BF^\times }  \| x y\|^{  \frac {1} 2} J_{\pi} (   r^2 x y) W 
\begin{pmatrix}
 y & \\
  & 1
\end{pmatrix} 
 d^\times y,
\end{equation}
for $W \in \mathcal W (\pi, \psi_r) $.
We define 
\begin{equation}\label{7eq: EJ}
\EuScript J_{\pi, \psi_r } (x) = \omega_\pi ( r  )  \|  r     \| \sqrt {\|x\|} J_{\pi} (  r     ^2 x).
\end{equation}
We call $\EuScript J_{\pi, \psi } (x)$ the {\it Bessel function associated with $\pi$ and $\psi $}. The formula \eqref{7eq: Hankel transform, W} then reads
\begin{equation}\label{7eq: Hankel transform, W, 2}
W   
\begin{pmatrix}
& 1\\
- x\- &  
\end{pmatrix}  
= \int_{\BF^\times }  \EJ_{\pi, \psi } ( x y) W 
\begin{pmatrix}
y & \\
& 1
\end{pmatrix} 
 d^\times y.
\end{equation}
Moreover, with the observation
\begin{equation*}
W   
\begin{pmatrix}
    & 1\\
- x\- &  
\end{pmatrix}   
=
\omega_\pi (-x)\- W \lp
\begin{pmatrix}
 x & \\
  & 1
\end{pmatrix} 
\varw_2 \rp,
\end{equation*}
\eqref{7eq: Hankel transform, W, 2} turns into
\begin{equation}\label{7eq: Weyl element action}
W \lp  
\begin{pmatrix}
 x & \\
 &  1
\end{pmatrix} 
 \varw_2 \rp 
= \omega_\pi (  -x  ) \int_{\BF^\times } \EJ_{\pi, \psi } ( x y) 
W 
\begin{pmatrix}
 y & \\
  & 1
\end{pmatrix} 
 d^\times y.\footnote{In the real case, this identity is given in \cite[Theorem 4.1]{CPS}. 
}
\end{equation}
Thus \eqref{7eq: Weyl element action} indicates that the action of the Weyl element $\varw_2$ on the Kirillov model 
$$\mathcal K (\pi, \psi) = \left\{ w(x) = W   \begin{pmatrix}
 x & \\
  & 1
\end{pmatrix}   : W \in \mathcal W (\pi, \psi) \right\}$$
is essentially a Hankel transform. From this perspective, the Hankel inversion formula follows from the simple identity $\varw_2^2 = I_2$. This may be seen from the following lemma.
\begin{lem}\label{7lem: n=2, Bessel} Let $\pi$ be an irreducible admissible representation of $\GL_2(\BF)$. Then we have
$
J_{\widetilde \pi} (x) = \omega_{\pi}(x) J_{\pi}(x).
$
\end{lem}
\begin{proof}
This follows from some straightforward calculations using Proposition \ref{3prop: properties of J, R} (1) and \ref{3prop: properties of J, C} (1). 
\end{proof}

\begin{rem}
	The representation theoretic viewpoint of Lemma {\rm \ref{7lem: n=2, Bessel}} is the isomorphism $\widetilde \pi \cong \omega\- \otimes \pi$. With this, Lemma {\rm \ref{7lem: n=2, Bessel}} is a direct consequence of Lemma {\rm \ref{4lem: normalizing the index}}.
\end{rem}

Finally, we shall summarize the formulae of the Bessel functions associated with infinite dimensional irreducible {\it unitary} representations of $\GL_2( \BF) $. First of all, in view of Lemma \ref{4lem: normalizing the index} and Remark \ref{4rem: G/Z}, one may assume without loss of generality that $\pi$ is trivial on $\mathrm Z_2 (\BR_+)$. Moreover, with the simple observation 
\begin{equation}\label{7eq: psi a and psi 1}
\EuScript J_{\pi, \psi_r } (x) = \omega_{\pi} (r        )  \EuScript J_{\pi, \psi_1 }  (r^2 x  ),
\end{equation}
it is sufficient to consider the Bessel function $\EuScript J_{\pi }  = \EuScript J_{\pi, \psi_1 } $  associated with $\psi_1$.


\subsection{Bessel functions for $\GL_2(\BR)$}

Under the Langlands correspondence, we have the following classification of infinite dimensional irreducible unitary representations  of $\GL_2( \BR)/\mathrm Z_2 (\BR_+)$.
\begin{itemize}
\item[-] (principal series and the limit of discrete series) $\varphi _{(i t,  \epsilon + \delta)} \oplus \varphi _{(- it,  \epsilon)}$, with $t \in \BR$ and $ \epsilon, \delta \in \BZT$,
\item[-] (complementary series) $\varphi _{(t,  \epsilon)} \oplus \varphi _{(- t,  \epsilon)}$, with $t \in \lp 0, \frac 1 2\rp$ and $ \epsilon \in \BZT$,
\item[-] (discrete series) $\varphi_{(0, m)}$, with $m \in \BN_+$.
\end{itemize}
Here, in the first case, the corresponding representation is a limit of discrete series  if $t = 0$ and $\delta = 1$ and  a principal series representation if otherwise. We shall write the corresponding representations  as
$\eta^{ \epsilon} \otimes \pi^+ (i t)$ if $\delta = 0$, $\eta^{ \epsilon} \otimes \pi^- (i t)$ if $\delta = 1$,
$\eta^{ \epsilon} \otimes \pi ( t)$
and $\sigma (m)$, respectively. We have
\begin{equation}
\omega_{\pi^+ (i t)} = 1, \ \ \omega_{\pi^- (i t)} = \eta, \ \  \omega_{\pi ( t)} = 1, \ \ \omega_{\sigma (m)} = \eta ^{m+1}.
\end{equation}
Furthermore, we have the equivalences $ \pi^+ (i t) \cong \pi^+ (- i t)$ and $ \pi^- (i t) \cong \eta \otimes \pi^- (- i t)$. 

As a consequence of Example \ref{3ex: Bessel n=1, R}, we have the following proposition.

\begin{prop} 

\  

{\rm (1).} Let $t \in \BR $. We have for $x \in \BR _+$
\begin{equation*} 
\begin{split}
\EJ_{\pi^+ (i t) } ( x) &= \frac {\pi i} {\sinh (\pi t)} \sqrt x \lp J_{2it} (4\pi \sqrt x ) - J_{-2it} (4\pi \sqrt x ) \rp,\\
\EJ_{\pi^+ (i t) } (-x) &= 4 \cosh (\pi t)  \sqrt x K_{2it} (4 \pi  \sqrt x ),
\end{split}
\end{equation*}
where it is understood that when $t = 0$ the right hand side of the first formula should be replaced by its limit, and
\begin{equation*} 
\begin{split}
\EJ_{\pi^- (i t) } ( x) &= \frac {\pi i} {\cosh (\pi t)}  \sqrt x \lp J_{2it} (4\pi   \sqrt x) + J_{-2it} (4\pi \sqrt x ) \rp,\\
\EJ_{\pi^- (i t) } ( -x) &= 4 \sinh (\pi t)  \sqrt x K_{2it} (4 \pi \sqrt x ).
\end{split}
\end{equation*}

{\rm (2).} Let $t \in \lp 0, \frac 1 2\rp$. We have for $x \in \BR _+$
\begin{equation*} 
\begin{split}
\EJ_{\pi ( t) } ( x) &= - \frac {\pi} {\sin  (\pi t)}   \sqrt x \lp J_{2t} (4\pi   \sqrt x ) - J_{-2t} (4\pi  \sqrt x ) \rp,\\
\EJ_{\pi ( t) } ( -x) &= 4 \cos (\pi t) \sqrt x K_{2t} (4 \pi  \sqrt x ).
\end{split}
\end{equation*}

{\rm (3).} Let $m \in \BN_+$. We have for $x \in \BR _+$
\begin{equation*} 
\EJ_{\sigma (m) } ( x) =  2 \pi i^{m+1} \sqrt x J_{m} \big(4 \pi \sqrt x \big),
\hskip 10 pt
\EJ_{\sigma (m) } ( -x) = 0.
\end{equation*}
\end{prop}
\begin{rem}
$\eta^{ \epsilon} \otimes \pi^+ (i t)$, $\eta^{ \epsilon} \otimes \pi (t)$ and $\sigma (2d-1)$ exhaust all the infinite dimensional irreducible unitary representations of $\PGL_2( \BR) $. Their Bessel functions are also given in \cite[Proposition 6.1]{CPS}.
\end{rem}

\subsection{Bessel functions for $\GL_2(\BC)$}

Under the Langlands correspondence, we have the following classification of infinite dimensional irreducible unitary representations  of $\GL_2( \BC)/\mathrm Z_2 (\BR_+)$.
\begin{itemize}
\item[-] (principal series) $\chiup _{(i t, k + d + \delta)} \oplus \chiup _{(- it, k - d)}$, with $t \in \BR $, $k, d \in \BZ $ and $\delta \in \BZT = \{0, 1\}$,
\item[-] (complementary series) $\chiup _{(t, k+d)} \oplus \chiup _{(- t, k-d)}$, with $t \in 
\lp 0, \frac 1 2\rp$, $k \in\BZ $ and $d \in \BZ$.
\end{itemize}
We   write the corresponding   representations 
as
$\eta^{k} \otimes \pi_{d}^+ (i t)$ if $\delta = 0$, $\eta^{k} \otimes \pi_{d}^- (i t)$ if $\delta = 1$ and $\eta^{k} \otimes \pi_{d} ( t)$, respectively.
We have
\begin{equation}
\omega_{ \pi_{d}^+ (i t) } = 1, \ \ \omega_{ \pi_{d}^- (i t) } = \eta, \ \ 
\omega_{ \pi_{d} ( t) } = 1.
\end{equation}
Furthermore, we have the equivalences $ \pi_{d}^+ (i t) \cong \pi_{- d}^+ (- i t) $, $ \pi_{d}^- (i t) \cong \pi_{- d - 1}^- (- i t) $. 

According to Example \ref{3prop: n=2, C}, we have the following proposition.

\begin{prop} \label{7prop: Bessel, C} 
	Recall the definitions {\rm(\ref{7def: J mu m (z), n=2, C}, \ref{7def: H (1, 2) mu m (z), n=2, C})} of $J_{\mu, m} (z)$ and $H^{(1, 2) }_{\mu, m} (z)$ in Example {\rm \ref{3prop: n=2, C}}. 

{\rm (1).} Let $t \in \BR $ and $d \in \BZ$. We have for $z \in \BCx$
\begin{align*}
\EJ_{\pi_{d}^+ (i t)} (z) & = - \frac {2 \pi^2 i} {\sinh (2\pi t)}   |z| \lp J_{i t, 2 d} (4 \pi  \sqrt z) -  J_{-i t, - 2 d} (4 \pi  \sqrt z) \rp\\
& = \pi^2 i  |z|  \lp e^{- 2 \pi   t} H^{(1)}_{i t, 2 d} \lp 4 \pi  \sqrt z \rp - e^{  2 \pi  t} H^{(2)}_{it, 2d} \lp 4 \pi   \sqrt z \rp \rp,\\
\EJ_{\pi_{d}^- (i t)} (z) & = \frac {2 \pi^2  i } {\cosh (2\pi t)}   {\sqrt {  |z | \, \overline z }} \lp J_{i t, 2 d + 1} (4 \pi  \sqrt z) + J_{-i t, - 2 d - 1} (4 \pi  \sqrt z) \rp \\
& = \pi^2  i \sqrt {  |z | \, \overline z }  \lp e^{- 2 \pi t} H^{(1)}_{i t, 2d+1} \lp 4 \pi  \sqrt z \rp +  e^{ 2 \pi t} H^{(2)}_{it, 2d+1} \lp 4 \pi  \sqrt z \rp \rp.
\end{align*}
 
{\rm (2).} Let $t \in 
\lp 0, \frac 1 2\rp$ and $d \in  \BZ$. We have for $z \in \BCx$
\begin{align*}
\EJ_{\pi_{d} ( t)} (z) & =  \frac {2 \pi^2 } {\sin (2\pi t)}  |z| \lp J_{ t, 2 d} (4 \pi  \sqrt z) -  J_{- t, - 2 d} (4 \pi   \sqrt z) \rp \\
& =  \pi^2 i  |z|  \lp e^{ 2 \pi i  t} H^{(1)}_{ t, 2 d} \lp 4 \pi  \sqrt z \rp - e^{-  2 \pi i t} H^{(2)}_{t, 2d} \lp 4 \pi   \sqrt z \rp \rp.
\end{align*}
\end{prop}

In view of Corollary \ref{6cor: d=1, J mu m}, we have the following integral representations of $\EJ_{\pi } \lp x e^{i\phi} \rp$ except for $\pi =  {\pi_{d} ( t)}$ and $ t  \in \left[ \frac 3 8, \frac 1 2 \rp $.

\begin{prop} \label{7prop: Bessel, C, integral} \  
	
	{\rm (1).} Let $t \in \BR $ and $d \in \BZ$. We have for $x \in \BR_+$ and $\phi \in {\BR/2\pi \BZ}$
	\begin{align*}
	\EJ_{\pi_{d}^+ (i t)} \lp x e^{i\phi} \rp &  =   4 \pi (-1)^d x e^{i d \phi} \int_0^\infty y^{4 i t - 1} \left[ y\- +  y e^{ i \phi} \right]^{-2d} J_{2 d} \lp 4 \pi \sqrt x \left|y\-  +  y e^{ i \phi}\right| \rp d y, \\
	\EJ_{\pi_{d}^- (i t)} \lp x e^{i\phi} \rp &  =   4 \pi i (-1)^d x e^{i d \phi} \int_0^\infty y^{4 i t - 1} \left[ y\- +  y e^{ i \phi} \right]^{-2d - 1} J_{2 d + 1} \lp 4 \pi \sqrt x \left|y\-  +  y e^{ i \phi}\right| \rp d y.
	\end{align*}
	
	{\rm (2).} Let $t \in 
	\lp 0, \frac 3 8\rp$ and $d \in  \BZ$. We have for $x \in \BR_+$ and $\phi \in {\BR/2\pi \BZ}$
	\begin{equation*}
	\EJ_{\pi_{d} ( t)} \lp x e^{i\phi} \rp  =  4 \pi (-1)^d x e^{i d \phi} \int_0^\infty y^{4 t - 1} \left[ y\- +  y e^{ i \phi} \right]^{-2d} J_{2 d} \lp 4 \pi \sqrt x \left|y\-  +  y e^{ i \phi}\right| \rp d y.
	\end{equation*}
	The integral on the right hand side converges  absolutely only for $ t \in \lp 0, \frac 1 8\rp$.
\end{prop}

\begin{rem}
$\pi^+_d (i t)$ and $\pi_d (t)$ exhaust all the infinite dimensional irreducible unitary representations of $\PGL_2( \BC) $. Proposition {\rm \ref{7prop: Bessel, C}} shows that the Bessel function for $\pi^+_d (i t)$ actually coincide with that given in \cite{B-Mo}. More precisely, we have the equality $\EJ_{\pi_{d}^+ (i t)} (z) =  {2 \pi^2 } |z| \EuScript K_{2 i t, - d} (4 \pi  \sqrt z),$ with $\EuScript K_{\nu, p}$ given by \cite[(6.21), (7.21)]{B-Mo}. Furthermore, the integral representation of $\EJ_{\pi_{d}^+ (i t)}$ in Proposition {\rm \ref{7prop: Bessel, C, integral}  (1)} is tantamount to \cite[Theorem 12.1]{B-Mo}. We have similar relations between the Bessel function for $\pi^-_d (i t)$ and that given in \cite{B-Mo2}.
\end{rem}

\bibliographystyle{alphanum}
\bibliography{references}

\delete{

}

\end{document}